\documentclass[reqno,10pt]{amsart}

\usepackage[a4paper,left=25mm,right=25mm,top=30mm,bottom=30mm,marginpar=25mm]{geometry}

\usepackage{amsmath}
\usepackage{amssymb}
\usepackage{amsthm}
\usepackage{amscd}
\usepackage[ansinew]{inputenc}
\usepackage{color}
\usepackage[final]{graphicx}
\usepackage{tikz}
\usepackage{bbm}
\usepackage{comment}
\usepackage{pgfplots}
\usetikzlibrary{intersections, pgfplots.fillbetween}
\usepackage{mathtools}

\usepackage{mathrsfs}
\usepackage{esint,amssymb}

\usetikzlibrary{patterns} 
\usetikzlibrary{positioning}
\usetikzlibrary{calc}

\usetikzlibrary{intersections}

\usepackage{float}
 \usepackage[colorlinks=true, pdfstartview=FitV, 
 linkcolor=olive
 ,             citecolor=olive
 , urlcolor=olive
 ]{hyperref}

\renewcommand{\epsilon}{\varepsilon}

\numberwithin{equation}{section}

\newtheoremstyle{thmlemcorr}{10pt}{10pt}{\itshape}{}{\bfseries}{.}{10pt}{{\thmname{#1}\thmnumber{
#2}\thmnote{ (#3)}}}
\newtheoremstyle{thmlemcorr*}{10pt}{10pt}{\itshape}{}{\bfseries}{.}\newline{{\thmname{#1}\thmnumber{
#2}\thmnote{ (#3)}}}
\newtheoremstyle{defi}{10pt}{10pt}{\itshape}{}{\bfseries}{.}{10pt}{{\thmname{#1}\thmnumber{
#2}\thmnote{ (#3)}}}
\newtheoremstyle{remexample}{10pt}{10pt}{}{}{\bfseries}{.}{10pt}{{\thmname{#1}\thmnumber{
#2}\thmnote{ (#3)}}}
\newtheoremstyle{ass}{10pt}{10pt}{}{}{\bfseries}{.}{10pt}{{\thmname{#1}\thmnumber{
A#2}\thmnote{ (#3)}}}

\theoremstyle{thmlemcorr}
\newtheorem{theorem}{Theorem}
\numberwithin{theorem}{section}
\newtheorem{lemma}[theorem]{Lemma}
\newtheorem{corollary}[theorem]{Corollary}
\newtheorem{proposition}[theorem]{Proposition}

\theoremstyle{thmlemcorr*}
\newtheorem{theorem*}{Theorem}
\newtheorem{lemma*}[theorem]{Lemma}
\newtheorem{corollary*}[theorem]{Corollary}
\newtheorem{proposition*}[theorem]{Proposition}
\newtheorem{problem*}[theorem]{Problem}
\newtheorem{conjecture*}[theorem]{Conjecture}

\theoremstyle{defi}
\newtheorem{definition}[theorem]{Definition}

\theoremstyle{remexample}
\newtheorem{remark}[theorem]{Remark}

\theoremstyle{ass}

\newcommand{\Acal}{\mathcal{A}}

\newcommand{\Mcal}{\mathcal{M}}

\newcommand{\Ocal}{\mathcal{O}}

\newcommand{\Ucal}{\mathcal{U}}
\newcommand{\Vcal}{\mathcal{V}}

\newcommand{\NN}{\mathbb{N}}

\DeclareMathOperator{\diverg}{div}
\DeclareMathOperator{\curl}{curl}
\DeclareMathOperator{\dist}{dist}
\DeclareMathOperator{\rank}{rank}

\newcommand{\N}{\mathbb{N}}

\newcommand{\RR}{\mathbb{R}}

\newcommand{\weaklystar}{\overset{*}\rightharpoonup}

\newcommand{\Lin}{\mathrm{Lin}}

\newcommand{\epsi}{\epsilon}

\allowdisplaybreaks

\def\RI{{\mathbb R}}

\def\O{{\Omega}}
\def\o{{\omega}}

\def\B{{\cal B}}
\def\F{{\sl F}}

\def\RI{{\mathbb R}}

\usepackage[scr=boondox]{mathalpha}

\DeclareMathOperator{\Tr}{Tr}

\def\scrk{{\mathscr k}}

\def\calO{\mathcal{O}}\def\totimes{\, \tilde\otimes \,}

\definecolor{Azul}{rgb}{0.372549,0.447059,0.784314}
\definecolor{aspargus}{rgb}{0.501961,0.501961,0}
\definecolor{vg}{rgb}{0.0, 0.40, 0.15}
\definecolor{jazzberryjam}{rgb}{0.65, 0.04, 0.37}
\definecolor{all}{rgb}{0.45, 0.14, 0.37} 
 \definecolor{Korange}{rgb}{0.945,0.561,0}

 \definecolor{Kblue}{rgb}{0,0.651,0.667}

 \definecolor{Korange}{rgb}{0.945,0.561,0}

 \definecolor{Kgreen}{rgb}{0.804,0.808,0}

 \definecolor{Kyellow}{rgb}{0.941,0.71,0}
 
 \definecolor{Mgray}{rgb}{0,0,0}    

\title[]{Junction in a thin multi-domain\\
 for nonsimple grade two materials in BH}

\author[Rita Ferreira]{Rita Ferreira}
\address{King Abdullah University of Science and Technology (KAUST), CEMSE
Division, Thuwal 23955-6900, Saudi Arabia}
\email{rita.ferreira@kaust.edu.sa}

\author[Jos\'e Matias]{Jos\'e Matias}
\address{University of Lisbon,
Portugal}
\email{jose.c.matias@tecnico.ulisboa.pt}

\author[Elvira Zappale]{Elvira Zappale}
\address{Department of Basic and Applied Sciences for Engineering,
Sapienza-University of Rome, via
A. Scarpa,  16 (00161) Roma, Italy}
\email{elvira.zappale1@uniroma1.it}

\usepackage[normalem]{ulem}

\begin{document}

 
\maketitle

 \begin{abstract}  
 \vspace{-12pt}   

We consider a thin  multi-domain in $\RI^N$, with $N\geq
2$,  consisting 
of
a vertical rod on top of a horizontal disk  made of non-simple grade-two materials or multiphase ones.  In this thin
multi-domain, we  consider a classical hyperelastic energy and complement it by adding an interfacial energy with a bulk density of the kind 
$W(D^2U)$, where $W$ is a continuous function with linear growth at $\infty$ 
and $D^2U$ denotes the Hessian tensor of a  vector-valued function $U$ that represents a deformation of the multi-domain. Considering  suitable
boundary conditions on the admissible deformations and assuming that the two volumes tend to
zero at the same rate,   we prove that the limit model is well posed
in the union of the limit domains, with
 dimensions $1$ and $N-1$, respectively,  and its limiting energy keeps memory of the original full dimensional trace constraints in a more accurate way than  previous  related  models  in the literature.  Moreover, we show that the
limit problem is uncoupled if $N\geq 3$, and 
``partially" coupled if $N=2$.

\vspace{8pt}

 \noindent\textsc{MSC (2010):}  49J45, 74B20, 74C99, 74K10, 74K20, 74K30, 74K35, 78M30,
78M35.
 
 \noindent\textsc{Keywords:} junction in thin
multi-domains, beam, wire, thin film,  bounded hessian,
dimension reduction, $\Gamma$-convergence,  plasticity,  martensite-austenite, trace constraints, non-simple grade-two materials, interfacial energy.

 \noindent\textsc{Date:} \today.
 \end{abstract}

\tableofcontents

%

\section{Introduction}\label{sect:intro}

 In the field of material science and engineering, the study
of thin structures is crucial for
understanding and predicting the behavior of complex materials, which have
a significant importance due to their applications in various industries,
including aerospace, civil engineering, and technology development.
Thin  structures are geometrically slender
and exhibit behaviors that classical three-dimensional models fail to accurately
predict due to the significant difference in scale along at least one dimension.

The variational analysis within the context of nonlinear elasticity provides
a robust framework for understanding and predicting the behavior of these
complex materials and structures,  
enabling the derivation of lower-dimensional models that are both mathematically
rigorous and physically meaningful (see \cite{ABP,BFMbend,FMP12,FMP13,FJM06,LeDR1,LeDR2,MoMu03,Sca09}
and  the references therein, for example, and \cite{KeZh19, 
ShXi17, ShXi17b, ShXuYa14,  ShZhWa22, Win17, Win18, Zhe19, Zhe22} for recent frontier research achievements on the existence and boundedness of solutions to partial differential equations and several of their  variants).

 Of particular importance
are thin multi-structures, where the interaction between different thin components
can lead to new and complex behaviors not observed in simpler structures.

In this manuscript, we adopt a variational approach to  study thin multi-structures,
such as a vertical rod on top of a horizontal disk,  in the context of non-simple
grade-two materials, including martensitic materials known for their shape-memory
and superelastic characteristics \cite{BhJ, GZsec, GZZsec, JR, LeM, SZ, S, T1,T2}. 
 Following \cite{BhJ}, the study of martensitic  thin films is performed in \cite{BEL}  by adding  to the standard hyperelastic energy term an interfacial penalization term  of the van der Walls type. Such penalization term is expressed in terms of second-order derivatives of the deformations, and aims at modeling the interfaces between austenitic and martensitic phases \cite{BhJ} (also see \cite{BEL,FFL1,  LeM,S}). 
In \cite{BEL}, the authors allow for sharp interfaces  via  a second-order penalization term expressed as the total variation of the non-necessarily regular deformation gradient that, as mentioned before,  is added to a classical hyperelastic energy dependent on the deformation gradient itself. 
 In this work,   as  we focus on thin multi-domains,  a penalization between interfaces as the one in \cite{BEL} might be too restrictive because we may lose information on the contact zone of our original joined fully-dimensional model when the underlying thin dimensions shrink in the limit. To overcome this issue, 
 we propose   an interfacial energy term described by means of a non-convex interfacial energy density with linear growth, which is  defined on more regular fields than those in \cite{BEL},   thus enabling us  to consider more accurately the boundary constraints due to the hypothesis of a partially clamped domain as in \cite{BEL,BhJ,  LeM}. 

We point out that departing from the Sobolev setting,  our asymptotic analysis leads to a more general limiting energy than the one emerged in \cite{BEL}. Indeed, our limit not only takes explicitly into account  the dependence on the possible jumps of the gradient through an appropriate `convexified' energy, but also   unfolds   an extra energetic penalization  resulting from    the boundary constraints of the thin multi-domain.  We stress that  this latter term is not present in \cite{BEL}  because though their  original full dimensional  functional space  (which the deformation belongs to)  requires less a priori regularity, it however  misses information on limiting boundary and transmission conditions between the components of the the multi-domain itself.  In the bi-dimensional case (\(N=2\) in our context), our analysis applies also to the study of junctions' problems in the plasticity framework, see \cite{DemARMA}, and 
extends the results related to the modeling of thin clamped plates to the linear growth energy setting, see \cite{DavPar, GZNODEA, GZ,  GauZa2} for the superlinear case, among a wide literature. In this context, the domain is $\perp$ shaped (see Figure 1). Finally, we also refer to \cite{DavMora, P, PT1, PT2} for related models and results in our functional framework and in the evolutive context.

Our model  is then characterized by  the sum of a classical hyperelastic energy term and an   interfacial  energy one  whose density  depends on second-order derivatives of the deformations,  as in \cite{BhJ, GZNODEA,GZ, LeM}  for instance. However,
in contrast with these works, we account here for  piecewise continuous deformations which second-order gradients may present jumps     as in \cite{BEL},   and sets our study in the framework of free-discontinuity and gradient-discontinuity problems (see \cite{FoHaPa19,MR4139444}, for instance),
which is mathematically encoded in our model by assuming  linear  growth
conditions on the stored energy density. This assumption leads us  to a variational
analysis in the space of functions of  bounded Hessian (BH)  of a model for  thin
multi-domains, 
 which we tackle by means of \(\Gamma\)-convergence techniques.
For the sake of completeness, we mention  also the theory of second-order structured deformations \cite{FoHaPa19,FLPp, FLP,  H, OP} in the $BH$ setting, accounting for the effects of microscopic jumps in the gradients and curvature in the energetic response of a given hyperelastic material.
In this respect, our current work can thus be viewed as a preliminary step towards an analysis of thin (multi-)domains in the context of second-order structured deformations,  allowing for a comparison with the results in \cite{MS} and \cite{CMMO}.
We further  refer to \cite{BlGr13, CMZ, FeZa19,  GGLM1,
 G1}, and references therein,  regarding variational
studies of  thin multi-structures in the context of simple materials, in which
the dependence on a deformation is via its  first-order derivatives.

 It is also worth  mentioning that our analysis   not only enriches the achievements of \cite{BEL}, but it also differs from existing results such as \cite{S}  because our  target  energies depend explicitly on the second-order gradient,  aiming at emphasizing the presence of bending terms in the limit model. We refer to \cite{FFL1, FFL2, DMFL}, where the bending effect is obtained considering the presence of second-order gradients as a perturbation of the energy.
\smallskip

We consider a thin multi-domain in
$\RI^{N}$, with \(N\geq  2\), consisting of a 
  thin tube-shaped type domain placed on top of a thin plate-shaped type domain
(see Figure~1), whose   associated  bulk  interfacial  energy density is of the
form $W(D^2U)$, where $W$ is a continuous function with linear growth\ and
$D^2U$ denotes the Hessian tensor of a field
function $U$.

 By assuming that the volume of the two thin cylinders
tend to zero with the same rate, we  derive the limit energy  under suitable
boundary conditions on
the top of the vertical cylinder and on the lateral surface of the
horizontal cylinder.  In particular,  we show that the
limit problem is  uncoupled,  is well posed in the union of the limit domains
with
 dimensions  $1$ and $N-1$, respectively, and involves four limit fields.
We precise next our setting and main result.

Let $N\geq 2$ be an integer number.  
Let $\omega \subset\RI^{N-1}$ be a  bounded, open, and connected
set
such that the origin in $\RI^{N-1}$, denoted by $0'$, belongs
to
$\omega$. Let $\left\{r_n\right\}_{n \in \mathbb N}$ and
$\left\{h_n\right\}_{n \in \mathbb N}\subset]0,1[$ be two
sequences such that
\begin{equation}\label{hrzero}
\lim_{n }h_n=0= \lim_{n }r_n.
\end{equation} For every  $n \in \mathbb N$, consider
a thin multi-domain consisting of a union between two vertical
cylinders, one
placed upon the other, as follows:  $\Omega_n:=\Omega_n^a\cup\Omega_n^b$
($a$ stands for
``above" and $b$ for ``below"), where  $\Omega_n^a:=r_n\omega\times
[0,1[$ has
small cross section $r_n\omega$ and constant height,
$\Omega_n^b:=\omega\times ]-h_n,0[$ has small thickness $h_n$
and
constant cross section (see Figure~1). Moreover,
set $\O:=\o \times
]-1,1[$.

\bigskip

\begin{center}
\hspace{4cm}\begin{tikzpicture}[scale=.7]

%
        \draw[Mgray, densely dashed,
           thick] (0.28,0) -- (0.28,-5.0);
         \draw[Mgray, densely dashed,
           thick] (1.72,0)-- (1.72,-5.0);

       
        \draw[Mgray, densely dashed,
           thick]
        (1,0)  ellipse (.72 and 0.25);
        
                
                \draw [   thin,<-]
(.3,0.4)
                arc (110:3:10pt);
                
                \draw(-1.7,0.4) node
{$\textcolor{black}{r_n\omega\times\{1\}=:\Gamma^a_n}$};
        
        \draw[gray,
            dashed,fill=Kyellow!40]
        (1,-5)  ellipse (.72 and 0.25);
        
         \draw(1.02,-5) node 
         {\fontsize{4}{4}\selectfont   $\textcolor{black}
         {r_n\omega\!\!\times\!\!\{0\}}$
        };


%

 \draw[densely dashed,
            thick,]
        (1,-5)  ellipse (4.32 and 1.5);

 \draw[densely dashed,
            thick,]
        (1,-5.5)  ellipse (4.32 and 1.5);


        \draw[Mgray,
        ]  (-3.35,-5)
         -- (-3.35,-5.5) ;

         \draw[Mgray,
        ]  (5.35,-5)
         -- (5.35,-5.5) ;


                \draw [   thin,->]
(5.15,-5.65)
                arc (-110:3:10pt);
                
                \draw(7.7,-4.9) node
                              {$\textcolor{black}{\Gamma^b_n:=\partial\omega\times]-h_n,0[}$};
                
                
        \draw [densely dashed,<->] (-3.6,-5.5)
                --  (-3.6,-5.09);
        \draw(-4.05,-5.3) node
                {${\textcolor{black}{h_n}}$};
       
        \draw[   thick, ->] (1,0)
--         (1,1);
        
        \draw[densely dashed, -] (1,0)
--
        (1,-5);
 
        \draw(1.45,.8) node { $x_3$};
        
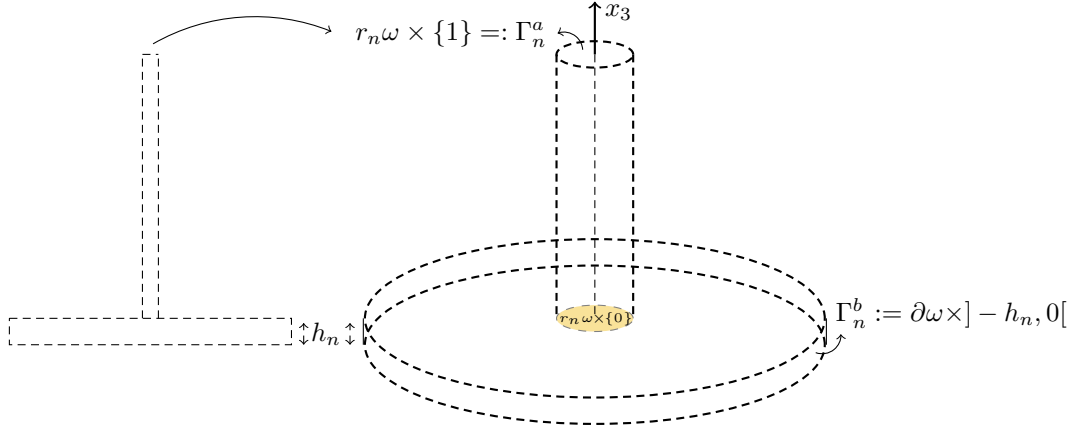
\draw(0,-8) node {Figure 1. Examples of  thin multi-domains $\Omega_n$ for \(N=2\) (left) and \(N=3\) (right).};

\draw[Mgray,densely dashed
        ]  (-10,-5) -- (-4.7,-5) -- (-4.7,-5.5)
         -- (-10,-5.5) -- (-10,-5) ;
\draw [densely dashed,<->] (-4.47,-5.5)
                --  (-4.47,-5.09);         

\draw[Mgray,densely dashed
        ]  (-7.5,-5) -- (-7.5,0) 
          ;
 \draw[Mgray,densely dashed
        ]   (-7.5,0) -- (-7.2,0) 
          ;
 \draw[Mgray,densely dashed
        ]    (-7.2,-5) -- (-7.2,0)
          ;       
               
\draw [   thin,->]
(-7.3,.2)
                arc (120:68:109pt);

               \end{tikzpicture}

\end{center}

\bigskip

In the sequel, $x=(x_1,\cdots
, x_{N-1},x_N)=(x',x_N)$  denotes a generic point of $\mathbb
R^N$; the gradient and
the Hessian tensor with
respect to the first $N-1$ variables are denoted by $D_{x'}$
and $D^2_{x'}$, respectively,
while the first and the second
derivatives with respect to  the last variable  are represented
 by $D_{x_N}$
and
$D^2_{x_N}$, respectively. Using these notations, $D^2_{x',x_N}$
stands for
$(D_{x_N})_{x'}$.
Moreover, denoting by  $\mathbb{(R}^{k
\times  k}_s)^d$, with  $k\in\NN$,   the 
space of \(\mathbb{R}^d\)-valued   symmetric
$(k\times 
k)$-matrices, we write a generic element  in \((\mathbb{R}^{{N\times  N}}_s)^d\)
  as  \(M=((m_{i,j})^k_{1\leq i,j\leq N})_{k=1}^d\);  
   for our purposes, it is convenient to decompose any such matrix \(M\)
as follows: we set     $A :=((m_{i,j})^k_{1\leq i,j\leq N-1})_{k=1}^d
\in \big(\mathbb{R}^{(N-1) \times 
(N-1)}_s\big)^d$, $B =((m_{N,j})^k_{1\leq j\leq N-1})_{k=1}^d
\in \mathbb{(R}^{N-1})^d$, and $C\in\mathbb{R}^d$, such that
\begin{equation}
\begin{aligned}\label{eq:defM}
M=\left(
\begin{array}{cc}A &{B}^T\\B &C
\end{array}\right).
\end{aligned}
\end{equation}

We  consider a bulk  interfacial  energy density function \(W :  (\mathbb{R}^{{N\times
 N}}_s)^d\to 
\mathbb{R}\)
 satisfying the following
assumptions:
\begin{align}
& W\hbox{ is  continuous in \((\mathbb{R}^{{N\times
 N}}_s)^d\),  }\label{contvsBorel}\\
& \frac{1}{C} |M| -C  \leq W(M) \leq C(1+|M|),\quad\forall M\in (\mathbb{R}^{{N\times
N}}_s)^d, \label{coerci}
\end{align}
for some \(C>0
\),  and a hyperelastic continuous energy density \(V :  \mathbb{R}^{N\times
d}\to 
\mathbb{R}\) satisfying
\begin{equation}\label{Vgrowth}
\frac{1}{C'} | \zeta |- C' \leq V( \zeta ) \leq C'(| \zeta |+ 1),
\end{equation}
for some positive constant $C'$
and for every $ \zeta  \in \mathbb R^{N\times d}$.
Accordingly, for every $n\in \mathbb N$, we are interested on energy
functionals  of the form \(I:W^{2,1}(\Omega_n;\mathbb
R^d)\to \RR \) defined by
\begin{equation*}
\begin{aligned}
I[U_n]:=  &\int_{\O_n} \left(W\left(D^2U_n\right)  +  V(D u_n) +H_n \cdot U_n\right)dx
\\=&\int_{\O_n}
\left(W\left(\left(\begin{array}{cc}D^2_{x'}U_n &\left(D^2_{x',x_N}U_n\right)^T\smallskip\\
D^2_{x',x_N}U_n &D^2_{x_N}U_n
\end{array}\right)\right) +  V(D_{x'}U_n, D_{x_N}U_n) + H_n \cdot U_n\right) dx,
\end{aligned}
\end{equation*} 
where $H_n \in L^\infty(\Omega_n;\mathbb R^d)$ and the last summand represent the loads that will be specified in the sequel. 
In addition, we require an admissible deformation  $U_n\in W^{2,1}(\Omega_n;\mathbb
R^d)$ to satisfy the   Dirichlet boundary conditions on  the top of $\Omega_n^a$, \(\Gamma^a_n:=r_n\omega\times \{1\}\),
and on the lateral surface of
$\Omega_n^b$,  \(\Gamma^b_n:=\partial\omega\times ]-h_n,0[\), given by
\begin{equation}
\label{eq:Ubc}
\begin{aligned}
&U_n = \theta^a  \text{ on } \Gamma^a_n=r_n\omega \times\{1\}\enspace \text{ with }
\theta^a(x'):=\theta^a_1+\theta^a_2 x' ,\\
&U_n = \theta^b  \text{ on } \Gamma^b_n=r_n\partial \omega \times]-h_n,0[\enspace \text{
with } \theta^b(x',x_N):=\theta^b_1(x')+\theta^b_2(x')
x_N ,
\end{aligned}
\end{equation}
for some 
\begin{equation*}
\begin{aligned}
\theta^a_1\in \mathbb{R}^d,\enspace \theta^a_2\in
\mathbb{R}^{d\times (N-1)},\enspace \theta^b_1\in W^{2,1}(\omega;\mathbb
R^d),\enspace \text{and}\enspace \theta^b_2\in W^{2,1}(\omega;\mathbb
R^d). 
\end{aligned}
\end{equation*}

Our mail goal is to study the asymptotic behavior as \(n\to\infty\) of the
minimization problem
\begin{equation}\label{eq:original}
\begin{aligned}
\inf_{U_n\in W^{2,1}(\Omega_n;\RR^d) \atop U_n \text{ satisfies \eqref{eq:Ubc}}
} I[U_n].
\end{aligned}
\end{equation}

As usual in dimension-reduction problems, we reformulate this
minimization problem on a fixed
domain through an appropriate rescaling  mapping from $\O_n$ into $\O$
(see Figure~1), as proposed by P.G. Ciarlet and P.
Destuynder in \cite{CD}.  Precisely,  we set
\begin{equation*}
\begin{aligned}
 u_n(x)=\begin{cases}
u^{a}_n(x',x_N):=
 U_n(r_nx',x_N), & \text{if } (x',x_N)\in\Omega^a:=\omega\times ]0,1[,\\[1.5mm]
 u^{b}_n(x',x_N):= U_n(x',h_n x_N), & \text{if }
(x',x_N)\in\Omega^b:=\omega\times ]-1,0[.
\end{cases}
\end{aligned}
\end{equation*}
Then, $ u^a_n\in
W^{2,1}(\O^a;\mathbb R^d)$ assumes the rescaled Dirichlet boundary condition
\begin{equation*}
\begin{aligned}
u^a_n=\theta_n^a,\quad \text{where } \,\theta_n^a(x') :=\theta^a_1+r_n\theta^a_2
x' \text{ with } \theta^a_1\in \mathbb{R}^d, \,\theta^a_2\in
\mathbb{R}^{d\times (N-1)},
\end{aligned}
\end{equation*}
 on the top of $\O^a$,  \(\Gamma^a:=\o\times \{1\}\), while $ u^b_n\in
W^{2,1}(\O^b;\mathbb R^d)$ assumes the rescaled Dirichlet boundary condition
\begin{equation*}
\begin{aligned}
u^b_n=\theta_n^b,\quad \text{where }\, {\theta_n^b(x',x_N):=\theta^b_1}(x')+h_n{\theta^b_2}(x')x_N
\text{ with } \theta^b_1\in W^{2,1}(\omega;\mathbb
R^d),\,\theta^b_2\in W^{2,1}(\omega;\mathbb
R^d),
\end{aligned}
\end{equation*}
on the lateral boundary of $\O^b$,   \(\Gamma_b:=\partial\o\times]-1,0[\).
 Moreover,
 $ u_n=( u^a_n, u^b_n)$  satisfies
the 
junction conditions 
\begin{equation}
\label{eq:bc}
\begin{aligned}
\begin{cases}
u^a_n(x',0)= u^b_n(r_nx',0),& \hbox{ for a.e.
}x' \in\o ,\\[1.5mm]
\displaystyle{ \frac{1}{r_n}}D_{x'}
u^a_n(x',0)=(D_{x'} u^b_n)(r_nx',0),& \hbox{ for a.e.
}x' \in\o ,\\[1.5mm]
D_{x_N} u^a_n(x',0) =
\displaystyle{\frac{1}{h_n}}D_{x_N} u^b_n(r_nx',0), & \hbox{ for a.e.
}x' \in\o.
\end{cases}
\end{aligned}
\end{equation}
where the functions in \eqref{eq:bc} are intended in the sense of traces.
Without loss of generality, one can assume that
 \begin{equation}\label{f=0=g}\theta^b_1=0=\theta^b_2\hbox{ a.e. in }B
 \end{equation}
for some $(N-1)$-dimensional ball $B$ such that $0'\in B\subset
\subset \o$. This condition is not restrictive because we are only interested
 in the values of \(\theta^b_1\) and \(\theta^b_2\) in a neighborhood of
\(\partial\omega\); we observe further that, as  in  \cite{GZNODEA}, condition
\eqref{f=0=g}  enables us to
 more easily deal with the limiting junction condition.
In summary, \(u_n \in\mathcal{U}_n \) with 
\begin{equation}\label{Vn}
\begin{aligned} \mathcal{U}_n:=\Big\{(u^a, u^b) \in \big(\theta_n^a+W_{\Gamma^a}^{2,1}
(\O^a;\RR^d)\big)\times \big(\theta_n^b +W^{2,1}_{\Gamma^b}(\O^b;\RR^d)\big)\colon
u^a \text{ and } u^b \text{ satisfy (\ref{eq:bc})}\Big\},
\end{aligned}
\end{equation}
where $W_{\Gamma^a}^{2,1}(\O^a;\mathbb
R^d)$   is the
closure of
 $\left\{u^a\in C^\infty(\overline{\O^a};\RR^d)\colon
 u^a=0 \hbox{ in a  neighborhood of }\Gamma^a=\o\times \{1\}\right\}$
 with respect to the $W^ {2,1}$-norm, and
  $W^{2,1}_{\Gamma^b}(\O^b;\RR^d)$   the closure
of
 $\big\{u^b\in C^\infty(\overline{\O^b};\RR^d)\colon
 u^b=0 \hbox{ in a neighborhood of }$ $\Gamma^b=\partial\o\times]-1,0[\big\}$
with respect to the $W^{2,1}$-norm.

Analogously, we  rescale the loads by setting
\begin{equation}\label{eq:loads}
\begin{aligned}
 H_n(x)=\begin{cases}
H_n^a(x):= H_n(r_n x', x_N), & \text{if } (x',x_N)\in\Omega^a=\omega\times
]0,1[,\\[1.5mm]
 H_n^b(x):= H_n(x', h_n x_N), & \text{if }
(x',x_N)\in\Omega^b=\omega\times ]-1,0[,
\end{cases}
\end{aligned}
\end{equation}
and we assume that there exists  $H \in L^\infty(\Omega;\RR^d)$ such that 
  \begin{equation}\label{loadconv}
 H_n^a\weaklystar  H  \hbox{ in } L^\infty(\Omega^a;\RR^d) \quad\text{and}\quad              H_n^b \weaklystar  H  \hbox{ in } L^\infty(\Omega^b;\RR^d).
 \end{equation}

    It is  worth  mentioning that it is also possible to take
into account loads  expressed in a divergence form as in \cite[Section~6.2]{FeZa19}, which we refer to for the Sobolev first-order setting.
Precisely, we may consider loads of the form
   \begin{equation*}
   {\rm div}\,h^{a}_n= H^a_n \hbox{ in } \Omega^a \hbox{ and
}{\rm div}\,h^{b}_n= H^{b}_n \hbox{ in } \Omega^b, 
   \end{equation*} 
   under suitable regularity conditions and \(L^\infty\) bounds on  $h^{a}_n$ and
$h^{b}_n$, to be compatible with the coercivity constant
of $W$. 
   By divergence theorem,
this case will naturally give rise to the presence of surface forces on the
boundary of our multi-domain, thus calling for the analysis
of possible  effects described in terms of the squared rescaled second-order derivatives.  The study of this case is left for a future work.  

 We also observe that other forces as in \cite{BEL2} could be considered but we omit to treat them here, referring to the aforementioned paper. 

In the  setting described above, the  rescaled energy (divided through
$r_n^{N-1}$) becomes
\(F_n:W^{2,1}(\Omega;\mathbb
R^d)\to \RR \) given by
\begin{align}\label{energy}
 F_n[u_n]=F_n[( u^a_n, u^b_n)] := K_n^a[ u^a_n]+ \frac{h_n}{r_n^{N-1}}K_n^b[
u^b_n],
\end{align}
where
\(K_n^a:  W^{2,1}(\O^a;\RR^d)\to\RR\) and 
\(K_n^b:  W^{2,1}(\O^b;\RR^d)\to\RR\) are the
functionals defined, respectively, by
\begin{equation}\label{kna}
K_n^a[ u^a] := \int_{{\O}^a} \left[
 W\left(\left(\begin{array}{cc}\displaystyle{\frac{1}{r_n^2}D^2_{x'}u^{a}}
&\displaystyle{\left(\frac{1}{r_n}D^2_{x',x_N}u^{a}\right)^T}\smallskip\\
\displaystyle{\frac{1}{r_n}D^2_{x',x_N}u^{a}} &D^2_{x_N}u^{a}
\end{array}\right)\right) +  V\left(\frac{1}{r_n} D_{x'}u^a, D_{x_N}u^a\right)+ H^a_n \cdot u^a\right]dx,
\end{equation}\medskip

\begin{equation}\label{knb}
K_n^b[ u^b]:= \int_{{\O}^b}\left[
W\left(\left(\begin{array}{cc}D^2_{x'}u^b
&\left(\displaystyle{\frac{1}{h_n}D^2_{x',x_N}u^b}\right)^T\smallskip\\
\displaystyle{\frac{1}{h_n}D^2_{x',x_N}u^b}
&\displaystyle{\frac{1}{h^2_n}D^2_{x_N}u^b
}\end{array}\right)\right)  + V\left(D_{x'}u^b, \frac{1}{h_n}D_{x_N}u^b\right) + H^b_n \cdot u^b\right]dx.
\end{equation}
Finally, we note that the rescaled minimization problem corresponding to \eqref{eq:original}
reads as \begin{equation}\label{minimumproblem_n}
\inf_{(u^a_n, u^b_n)\in \mathcal{U}_n}\left\{K_n^a[u^a_n]+
\frac{h_n}{r_n^{N-1}}K_n^b[u^b_n] \right\}.
\end{equation}

The
 aim of this paper consists of describing the asymptotic behavior,
as $n\to +\infty$,  of the energy  in
(\ref{energy}) and of the minimization problem in \eqref{minimumproblem_n} when the volumes of
 $\Omega_n^a$ and   $\Omega_n^b$ tend to zero at the same
rate; that is, when\begin{equation}\label{rate} \lim_{n\to\infty
}\frac{h_n}{r_n^{N-1}}=q\in\,\,]0,+\infty[.
\end{equation}
This limiting  behavior is encoded in our main results,  Theorem~\ref{thm:mainFMZ} and Corollary~\ref{cor:min} below.

 We refer to Section~\ref{pre} for a detailed description of the notations
used in this manuscript. Next, we introduce only the main  ones that allow
for general understanding of our main result.

Set
\begin{equation}
\label{eq:UcalXi}
\begin{aligned}
& \mathcal{U}:=\big\{(u^a, u^b)\in BH(]0,1[;\RR^d) \times BH(\omega;\RR^d)\colon u^a(1) =\theta^a_1 \text{ and }u^b=\theta^b_1  \text{ on }\partial\omega \text{ for any } N\geq 2, \text{ and}\\& 
\hskip71.15mm u^a(0)=u^b(0) \hbox{ if }N=2
\big\}, 
\\
&  \Xi:= BV(]0,1[; \mathbb R^{d \times (N-1)})\times 
                BV(\omega;\mathbb R^d),
\end{aligned}
\end{equation}
and consider  the functional \(K^a\colon  BV(]0,1[;\mathbb
R^d)\times BH(]0,1[;\mathbb R^{d\times (N-1)}) \to \RR\)  defined
  for each \((u^a,\xi^a)  \in \allowbreak
    BV(]0,1[;\mathbb
R^d)\times BH(]0,1[;\mathbb R^{d\times (N-1)})\) by 
\begin{equation}
\label{Ka}
\begin{aligned}
K^a[(u^a,\xi^a)]:= &\int_0^1{\hat W}^{**}\big(
        \nabla_{x_N}\xi^a, \nabla^2_{x_N}u^a\big) \,dx_N + \int_{0}^1({\hat
W}^{**})^{\infty}\bigg(\frac{D^s (\xi^a, \nabla u^a)}{|D^s (\xi^a, \nabla
u^a)|}\bigg)d |D^s (\xi^a, \nabla u^a)|\\
&+ (\hat W^{\ast\ast})^\infty \Big(- (\xi^a)^-(1)+\theta^a_2,- \Big(\frac{d
u^a}{d x_N}\Big)^-(1)
\Big)   +   \int_0^1 V(\xi^a, D_{x_N} u^a )\, d x_N + \int_{\Omega^a} H \cdot u^a \,dx, 
\end{aligned}
\end{equation}
and the functional  \(K^b\colon  BH(\O^b;\RR^d) \times  BV(\O^b;\RR^d) \to
\RR\)  defined
  for each \((u^b,\xi^b)
 \in \allowbreak  BH(\O^b;\RR^d) \times \allowbreak  BV(\O^b;\RR^d)\) by
\begin{equation}
\label{Kb}
\begin{aligned}
K^b[(u^b,\xi^b)]:= &\int_{\omega}Q_{\widetilde{\mathcal A}}\,W_0\big(\nabla^2_{x'}u^b,
\nabla_{x'}\xi^b\big)\,dx' +\int_{\omega}(Q_{\widetilde{\mathcal A}}\,W_0)^{\infty}\bigg(\frac{D^s(\nabla_{x'}u^b,\xi^b)}{|D^s(\nabla_{x'}u^b,
\xi^b)|}\bigg)d |D^s(\nabla_{x'}u^b, \xi^b)|\\
&+\int_{\partial \omega} (Q_{\widetilde{\mathcal A}}\,W_0)^\infty \big(-\big((\nabla_{x'}u^b)^-
- \nabla_{x'}\theta^b_1,(\xi^b)^- - \theta^b_2\big)\otimes{\nu_{\partial
\Omega}}\big)
d \mathcal
H^{N-2} \\   & + \int_\omega V(D_{x'}u^b, \xi^b)\,dx' + \int_{\Omega^b} H \cdot u^b \,dx, 
\end{aligned}
\end{equation}
where, recalling \eqref{eq:defM},  \(\hat W: \RR^{d \times (N-1)}\times \RR^d
\to\RR \) and \(W_0: \big(\mathbb{R}^{(N-1) \times 
(N-1)}_s\big)^d \times \RR^{d \times (N-1)} \to \RR \) are the functions
defined by
\begin{equation}\label{eq:defhatzero}
\begin{aligned}
\hat W(B,C):=\inf_A W\left( \left(
\begin{array}{cc}A &{B}^T\\B &C
\end{array}\right) \right) \quad \text{ and } \quad W_0 (A,B): = \inf_C W\left(
\left(
\begin{array}{cc}A &{B}^T\\B &C
\end{array}\right) \right)
\end{aligned};
\end{equation}
moreover, \({\hat W}^{**}\) denotes the convex envelope of \(\hat W\) and
$({\hat W}^{**})^\infty$ the recession function of \({\hat W}^{**}\), while
\(Q_{\widetilde{\mathcal A}}\,W_0\)  denotes the \(\widetilde{\mathcal A}\)-quasiconvex
envelope of \(W_0\) and \((Q_{\widetilde{\mathcal A}}\,W_0)^{\infty}\) the
recession function of \(Q_{\widetilde{\mathcal A}}\,W_0\) with \(\widetilde{\mathcal
A}\) a convenient differential linear operator with constant coefficients,
corresponding to a cross
2-quasiconvexification--quasiconvexification of \(W_0\), which we detail
in Subsection~\ref{sect:Aqcx}. (also see \eqref{S101}).

\begin{remark}
   Note that in the  $N=2$ case,  the expression for \eqref{Kb} is very similar to that of \eqref{Ka} because \(Q_{\widetilde{\mathcal A}}\,W_0\) = \({\hat W_0}^{**}\)  and the integral term on $\partial \omega$
   is just the sum of the corresponding values at the boundary points,  given that the limit sample is one-dimensional. Precisely,  if $\omega= (\alpha, \beta)$, with $-\infty<\alpha < 0< \beta <+\infty$, then
 \begin{equation*}
\begin{aligned}
K^b[(u^b,\xi^b)]:= & \int_{\omega}(W_0)^{**}\big(\nabla^2_{x'}u^b,
\nabla_{x'}\xi^b\big)\,dx' +\int_{\omega}(W_0^{**})^{\infty}\bigg(\frac{D^s(\nabla_{x'}u^b,\xi^b)}{|D^s(\nabla_{x'}u^b,
\xi^b)|}\bigg)d |D^s(\nabla_{x'}u^b, \xi^b)|\\
&+(W_0^{**})^\infty \big(-\big((u^b)')^-(\alpha)
- (\theta^b)'_1(\alpha),(\xi^b)(\alpha) - \theta^b_2(\alpha)\big) \\
&+ (W_0^{**})^\infty \big(\big((u^b)')^-(\beta)
- ((\theta^b)')_1(\beta),(\xi^b)(\beta) - \theta^b_2(\beta)\big)\\
&+ \int_\omega V(D_{x'}u^b, \xi^b) \,dx'+
\int_{\Omega^b} H \cdot u^b \,dx. 
\end{aligned}
\end{equation*}
  
\end{remark}

Our main theorem is the following.
\begin{theorem}\label{thm:mainFMZ}  Let    \(W \colon  (\mathbb{R}^{{N\times 
N}}_s)^d\to 
\mathbb{R}\)  be a function satisfying conditions  (\ref{contvsBorel})--(\ref{coerci}),   let $V \colon \mathbb R^{d \times N}\to \mathbb R$ be a continuous function satisfying \eqref{Vgrowth},  and consider the corresponding sequence of  functionals
\((\F_n)_{n\in\N}\) introduced by  \eqref{energy}, where the parameters \(r_n\)
and \(h_n\) satisfy conditions     
         (\ref{hrzero}) and (\ref{rate}),  and where the loads
\(H_n\) satisfy 
\eqref{eq:loads}--\eqref{loadconv}.  Consider further the spaces  $\mathcal U_n$ and 
 \(\mathcal
U \times \Xi\) and  the functionals  \(K^a\) and \(K^b\)  introduced
in \eqref{Vn}, \eqref{eq:UcalXi}, \eqref{Ka}, and \eqref{Kb}, respectively. Finally, for \(n\in\NN\),
consider the functional \[E_n\colon \left(L^1(\Omega^a;\RR^d) \times L^1(\Omega^b;\RR^d)\right)
\times \left(L^1( \Omega^a; \mathbb
R^{d \times (N-1)}\big) \times L^1(\Omega^b;\RR^d)\right) \to\RR\] defined
  by
\begin{equation*}
\begin{aligned}
E_n[(u^a, u^b), (\xi^a, \xi^b)]:=\begin{cases}
F_n[(u^a,u^b)] & \text{if } (u^a,u^b)\in \Ucal_n,\,\, \xi^a =\frac{1}{r_n}D_{x'}u^{a},
 \text{ and }\, \xi^b =\frac{1}{h_n}D_{x_N}u^b,\\
 +\infty &\text{otherwise.} 
\end{cases}
\end{aligned}
\end{equation*}
Then, the 
sequence \((E_n)_{n\in\NN}\) \(\Gamma\)-converges to the functional 
\[E\colon \left(L^1(\Omega^a;\RR^d) \times L^1(\Omega^b;\RR^d)\right)
\times \left(L^1( \Omega^a; \mathbb
R^{d \times (N-1)}\big) \times L^1(\Omega^b;\RR^d)\right) \to\RR\]
defined   by
\begin{equation*}
\begin{aligned}
E[(u^a, u^b), (\xi^a, \xi^b)]:=\begin{cases}
K^a[({u}^a, \xi^a)]+ q
K^b[({u}^b,{\xi}^b)] & \text{if } (( u^a,  u^b),  (\xi^a,  \xi^b))\in \mathcal
U \times \Xi,\\
 +\infty &\text{otherwise.} 
\end{cases}
\end{aligned}
\end{equation*}

 \end{theorem}

\begin{corollary}\label{cor:min}
Under the notation and  assumptions of Theorem~\ref{thm:mainFMZ}, let \[((u^a_n, u^b_n), (\xi^a_n, \xi^b_n))_{n\in\NN}\subset  \left(L^1(\Omega^a;\RR^d) \allowbreak \times
L^1(\Omega^b;\RR^d)\right)
\times \left(L^1( \Omega^a; \mathbb
R^{d \times (N-1)}\big) \times L^1(\Omega^b;\RR^d)\right) \]
be a diagonal infimizing sequence for the sequence of minimizing
problem  \((\text{\ref{minEn}})_{n\in\NN}\) defined by
\begin{equation*}
\label{minEn}
\tag{\(\mathcal{E}_n\)}
\begin{aligned}
\inf\Big\{ E_n[(u^a, u^b), (\xi^a, \xi^b)]\colon &(u^a, u^b)\in   L^1(\Omega^a;\RR^d) \times
L^1(\Omega^b;\RR^d), \\
& (\xi^a, \xi^b)\in  L^1( \Omega^a; \mathbb
R^{d \times (N-1)}\big) \times L^1(\Omega^b;\RR^d)\Big\}.
\end{aligned}
\end{equation*}
Then, \(((u^a_n, u^b_n), (\xi^a_n, \xi^b_n))_{n\in\NN}\) is compact
in \( \left(L^1(\Omega^a;\RR^d) \times L^1(\Omega^b;\RR^d)\right)
\times \left(L^1( \Omega^a; \mathbb
R^{d \times (N-1)}\big) \times L^1(\Omega^b;\RR^d)\right) \);
moreover, if
\((\bar u^a, \bar u^b), (\bar \xi^a, \bar \xi^b)\in \left(L^1(\Omega^a;\RR^d) \times L^1(\Omega^b;\RR^d)\right)
\times \left(L^1( \Omega^a; \mathbb
R^{d \times (N-1)}\big) \times L^1(\Omega^b;\RR^d)\right)\) is
a corresponding accumulation point, then \(((\bar  u^a,  \bar u^b),  (\bar \xi^a,  \bar \xi^b))\in
\mathcal
U \times \Xi\) and  %
\begin{equation*}
\begin{aligned}
K^a[(\bar {u}^a, \bar \xi^a)]+ q
K^b[(\bar {u}^b,{\bar \xi}^b)] &= \min \Big\{ E[(u^a, u^b), (\xi^a, \xi^b)]\colon (u^a, u^b)\in
  L^1(\Omega^a;\RR^d) \times
L^1(\Omega^b;\RR^d), \\
&\hskip44.5mm (\xi^a, \xi^b)\in  L^1( \Omega^a; \mathbb
R^{d \times (N-1)}\big) \times L^1(\Omega^b;\RR^d)\Big\}\\
&= \min \big\{ K^a[({u}^a, \xi^a)]+ q
K^b[({u}^b,{\xi}^b)] \colon (( u^a,  u^b),  (\xi^a,  \xi^b))\in
\mathcal
U \times \Xi \Big\}.
\end{aligned}
\end{equation*}
\end{corollary}

The fact that the limit problem is partially coupled  by the condition $u^a(0)=u^b(0)$ for \(N=2\) is not surprising and  is a parallel result  to those  contained in \cite[Section 4]{GZ} and \cite{GauZa2} regarding  the $BH$ setting.

In the sequel,  as in \cite{BEL, BhJ, GauZa2,  LeM}, we  focus   our  analysis mainly on the interfacial energy term because it concentrates   the more technical difficulties. We further notice that there is no loss of generality in neglecting the hyperelastic energy term as it can be regarded as a continuous perturbation  with respect to our topology; more precisely,  dominated convergence theorem applies, and we  can easily pass  to the limit the associated terms. Similarly, under suitable assumptions, also the bulk loads constitute a continuous term for our asymptotic analysis, and thus can be omitted from our proofs.  We refer to Remark~\ref{rmk:Hnzero} for more details on these continuity arguments.

 For the sake of completeness,  we conclude this introduction by observing   that one could also consider  more general  energies in \eqref{kna} and \eqref{knb} of the type
\begin{equation}\label{knabis}
K_n^a[ u^a] := \int_{{\O}^a} \left[
 W_a\left(\left(\begin{array}{cc}\displaystyle{\frac{1}{r_n^2}D^2_{x'}u^{a}}
&\displaystyle{\left(\frac{1}{r_n}D^2_{x',x_N}u^{a}\right)^T}\smallskip\\
\displaystyle{\frac{1}{r_n}D^2_{x',x_N}u^{a}} &D^2_{x_N}u^{a}
\end{array}\right),\left(\frac{1}{r_n} D_{x'}u^a, D_{x_N}u^a\right), u^a\right)+ H^a_n \cdot u^a\right]dx,
\end{equation}\medskip

\begin{equation}\label{knbbis}
K_n^b[ u^b]:= \int_{{\O}^b}\left[
W_b\left(\left(\begin{array}{cc}D^2_{x'}u^b
&\left(\displaystyle{\frac{1}{h_n}D^2_{x',x_N}u^b}\right)^T\smallskip\\
\displaystyle{\frac{1}{h_n}D^2_{x',x_N}u^b}
&\displaystyle{\frac{1}{h^2_n}D^2_{x_N}u^b
}\end{array}\right), \left(D_{x'}u^b, \frac{1}{h_n}D_{x_N}u^b\right), u^b \right)+ H^b_n \cdot u^b\right]dx,
\end{equation}
respectively, for suitable (continuous) functions $W_a$ and $W_b$.  For simplicity,  this more general setting   is not addressed  here.  We foresee that   it could be handled with arguments closer to the so-called Global Method for Relaxation, as in \cite{FoHaPa19}, but it would  however  lead to less explicit limit energy densities  than those in our main results stated above.

\section{Notation and Preliminaries}\label{pre}

We start 
 by collecting the main functional spaces used in our analysis. We then recall
the trace operator in the space of  functions of bounded variation, \( BV\),
and of bounded Hessian,  \(BH\), and  address some associated preliminary
results. We conclude this section with a brief discussion on $\mathcal A$-quasiconvexity.

\subsection{Functional Spaces}

Let \(\ell, m\in\N\) and let  \(\calO\subset\RR^\ell\) be
an open set.
We represent by \(\mathcal  B(\Ocal)\) the Borel \(\sigma\)-algebra on  \(\Ocal\),
and by \(\mathcal M(\Ocal;\RR^m)\) the Banach space of all \(\RR^m\)-valued
Radon measures endowed with the total variation norm \(|\cdot|(\Ocal)\).
The duality pairing between  \(\mathcal M(\Ocal;\RR^m)\) and \(C_0  (\Ocal;\RR^m)\)
is denoted by \(\langle\cdot, \cdot\rangle\), where, as usual, \(C_0  (\Ocal;\RR^m)\)
is the closure with respect to the supremum norm, \(\Vert\cdot\Vert_\infty\),
of the space of all continuous functions with compact support, \(C_c(\Ocal;\RR^m)\).
The space of smooth functions with compact support is denoted either by \(\mathcal D(\Ocal;\RR^m)\) or by \(C^\infty_c(\Ocal;\RR^m)\) and its dual by \({\mathcal D}'(\Ocal;\RR^m\)).
 We use the subscript ``per" within functional spaces when considering functions
that   are \(1\)-periodic in all directions; or, in other words,
that are $Q$  periodic, with \(Q\) representing the unit cube in the ambient
space. For instance, \(C^\infty_
{\rm per}(\mathbb R^\ell;\mathbb R^m)\) is the space of all  \(\RR^m\)-valued
 \(C^\infty\) functions in \(\RR^\ell\) that  are $Q:=(0, 1)^\ell$  periodic.
Similarly,
 \(L^p_{per}(Q;\RR^m)\) with $Q=(0, 1)^\ell$ is the space of all  locally integrable  \(\RR^m\)-valued
 functions on \(\RR^\ell\) that are \(Q\) periodic.

We denote by \(BV(\Ocal;\RR^m)\) the space of function of bounded variation,
which is the space of all functions \(u\in L^1(\Ocal;\RR^m)\) whose distributional
derivative, \(Du\), can be identified with a Radon measure in \(\mathcal
M(\Ocal;\RR^{m\times \ell})\). We use the standard notation for \(BV\) functions,
cf.~\cite{AFP}. 
   Moreover, the space of \(\RR^m\)-valued functions with bounded
hessian on \(\Ocal\), \(BH(\Ocal;\RR^m)\), defined in 
\cite{DT1,D}, is the space of functions  $u\in W^{1,1}(\Ocal;\RR^m)$
for which the distributional derivative of the gradient, $D(\nabla u)$, can
 be identified with a Radon measure in \(\mathcal
M(\Ocal;(\RR^{\ell\times \ell})^m)\); precisely, 
\begin{align*}
        BH(\Ocal;\RR^m):=&\, \{u\in W^{1,1}(\Ocal;\RR^m): |D (\nabla u)|(\Ocal)<+\infty\}\\
        =&\, \{u \in W^{1,1}(\Ocal;\RR^m): |D^2 u| \in \mathcal M(\Ocal)\}\\
        =&\, 
        \{u \in L^1(\Ocal;\RR^m): D u \in BV(\Ocal;\RR^{m\times \ell
        })\}.
\end{align*}
We observe that in the above and in the sequel, we adopt the notation for second-order derivatives of vector-valued functions used in the papers \cite{GZNODEA, GZ} dealing with junction problems that are related to ours, for a more straightforward reference. As such, if $u=(u^1,...,u^m):\Ocal\to\RR^m$ is smooth, then $D^2u$ is written as a vector in $\RR^m$ whose $i$th component is the symmetric matrix $(\frac{\partial^2 u^i}{\partial x_k\partial x_j})_{1\leq k,j\leq \ell}$ in $\RR_s^{\ell\times\ell}$. However, we also observe  that   in the framework of $BH$  as in \cite{FLP} for instance,  $D^2u$ is often written as a third-order symmetric tensor.

Endowed with
the norm $\|u\|_{BH(\Ocal;\RR^m)}= \|u\|_{W^{1,1}(\Ocal;\RR^m)}+ |D^2
u|(\Ocal)$,  $BH(\Ocal;\RR^m)$ is a Banach space. Moreover, the Hessian tensor,
\(D(\nabla u)\), can be decomposed into an absolutely
continuous part, a jump part, and a Cantor part, written
$D^2 u= \nabla^2 u+ D^2_j u+ D^2_C u$. The jump part is concentrated in the
jump set of $J_{\nabla u}$.

The following two well-know results (see, for instance,   \cite[Theorems~3.1 and 3.3 and Remark~3.2]
{D} and \cite[Theorems~2.1
and 2.2]{FLPp}, respectively) will be useful in the sequel.

\begin{proposition}\label{embedding}
         Let $ \calO \subset \mathbb R^\ell$ be a Lipschitz and bounded open
set.
Then,
$BH(\calO)\subset W^{1,p}(\calO)$
with continuous embedding if $p \leq \frac{\ell}{\ell-1}$, and   compact
embedding if $p <\frac{\ell}{\ell-1}$.
 If $\ell\leq 2$ and   \(\Ocal\) is, in addition, of class \(C^2\) except at a finite number of points, then $BH(\calO)\subset C(\overline{\calO})$
with continuous embedding.
\end{proposition}

\begin{theorem}[Interpolation inequality]\label{interpolBH}
         Let $\calO \subset \mathbb R^\ell$ be a Lipschitz and bounded
open set. Then, for every $\varepsilon > 0$ there is a constant, $C = C(\varepsilon)$,
such that
$\|\nabla u\|_{L^1(\calO;\RR^\ell)} \leq C\|u\|_{L^1(\calO)} + \varepsilon|D^2u|(\calO)$
for all $u \in BH(\calO)$.
\end{theorem}

Similarly to the \(BV\) case, bounded sequences in \(BH\) are pre-compact
with respect to the weak-\(\ast\) convergence in \(BH\), in which case we
can find a subsequence for which the corresponding sequence of the distributional
derivative of the gradients converge weakly-\(\ast\) in the sense of measures.
In some instances, however, we will need a stronger notion of convergence
for measures, known as area-strict  convergence that we recall next.

\begin{definition}\label{1100}
Let \(\tilde{\mathcal{O}}\subset \RR^\ell\) be a Borel set and, for \(n\in\NN\),
let $\mu_n=m_n^a\mathcal L^\ell+\mu_n^s\in\mathcal M(\tilde{\mathcal{O}};\mathbb
R^{m})$  and  $\mu=m^a\mathcal L^\ell+\mu^s\in\mathcal M(\tilde{\mathcal{O}};\mathbb
R^{m})$  be Radon measures on \(\tilde{\mathcal{O}}\). We say that
\begin{itemize}
\item[(i)]  $(\mu_n)_{n\in\NN}$ converges weakly-\(\ast\) to $\mu$,  written

$\mu_n\overset{   \ast}{\rightharpoonup}\mu$, if
\begin{equation*}
\lim_{n\to\infty}\int_{\tilde{\mathcal{O}}} \varphi(x)\cdot d\mu_n(x) = \int_{\tilde{\mathcal{O}}}
\varphi(x)\cdot d\mu(x)
\quad\text{for every $\varphi\in C_0(\tilde{\mathcal{O}};\mathbb R^{m})$};
\end{equation*}
\item[(ii)] $(\mu_n)_{n\in\NN}$ converges locally weakly-\(\ast\) to $\mu$
 if
\begin{equation*}
\lim_{n\to\infty}\int_{\tilde{\mathcal{O}}} \varphi(x)\cdot d\mu_n(x) = \int_{\tilde{\mathcal{O}}}
\varphi(x)\cdot d\mu(x)
\quad\text{for every $\varphi\in C_c({\tilde{\mathcal{O}}};\mathbb R^{m})$};
\end{equation*}
\item[(iii)]  $(\mu_n)_{n\in\NN}$ converges
strictly to $\mu$ if $\mu_n\overset{*}{\rightharpoonup}\mu$
and $|\mu_n|({\tilde{\mathcal{O}}})\to|\mu|({\tilde{\mathcal{O}}})$;
\item[(iv)] $(\mu_n)_{n\in\NN}$ converges
(area) $\langle\cdot\rangle$-strictly to $\mu$ if $\mu_n\overset{*}{\rightharpoonup}\mu$
and $\langle\mu_n\rangle({\tilde{\mathcal{O}}})\to\langle\mu\rangle({\tilde{\mathcal{O}}})$,
where
\begin{equation*}
\langle\mu\rangle(A)\coloneqq \int_A \sqrt{1+|m^a(x)|^2}\,d x+|\mu^s|(A)
\quad\text{for every \(A\in \mathcal B({\tilde{\mathcal{O}}})\)}.
\end{equation*}
\end{itemize}
\end{definition}

Given a matrix \(A\in (\RR^{\ell\times \ell})^m\), we also write \(\langle
A \rangle:= \sqrt{1+|A|^2}\), where here, and elsewhere in this manuscript,
\(|A|\) denotes the norm of \(A\)
given by
\begin{equation*}
\begin{aligned}
|A|=  \bigg(\sum_{k=1}^m\sum_{i,j=1}^\ell((a_{i,j})^k)^2\bigg)^{\frac12}.
\end{aligned}
\end{equation*}

Note that for all \(A\), \(B\in (\RR^{\ell\times \ell})^m\), we have that
\begin{equation}
\label{eq:estsumang}
\begin{aligned}
\langle
A + B \rangle \leq \langle A \rangle + |B|.
\end{aligned}
\end{equation}

\subsection{Trace operator}\label{subsect:trace}

We start by recalling the trace theorem for Sobolev functions, which can be
found  in  \cite[Section 4]{CDA02}
 for instance. 

\begin{theorem}[{\bf Trace in \(W^{1,p}\), cf.~\cite[Theorem~4.3.12]{CDA02}}]\label{TrBV}
Let $\Ocal\subset \RR^\ell$ be a bounded and open set with Lipschitz boundary,
  $p \in [1,+\infty]$, and 
\begin{equation*}
\begin{aligned}
X({\partial\Ocal}):=\begin{cases}
L^{\frac{(\ell-1)p}{\ell-p}}(\partial
\Ocal;\mathcal H^{\ell-1}) &\text{if } p \in[1,\ell[,\\
L^q(\partial
\Ocal;\mathcal H^{\ell-1}) \text{ with \(q\in [1,+\infty[\)} &\text{if }
p =\ell,\\
C(\partial
\Ocal)  &\text{if } p \in\,]\ell,+\infty].
\end{cases}
\end{aligned}
\end{equation*}
Then, there exists a  bounded linear operator $\Tr:W^{1,p}(\Ocal)\to X({\partial\Ocal})$
such that
for every $u \in W^{1,p}(\Ocal)\cap C(\overline \Ocal)$, we have 
\begin{equation}
\label{eq:usustr}
\begin{aligned}
 \Tr u(x)= u(x) \quad \text{for $\mathcal
H^{\ell-1}$-a.e.~$x \in \partial \Ocal$.}
\end{aligned}
\end{equation}
If \(p \in\,]\ell,+\infty]\), then \eqref{eq:usustr} holds for every \(u
\in W^{1,p}(\Ocal)\) and \(x\in \partial \Ocal\). Moreover, for  \(p \in
[1,+\infty]\), we have for every  $u \in W^{1,p}(\Ocal)$ and $\varphi \in
C^1(\mathbb R^\ell;\RR^\ell)$ that
  \begin{equation}
\label{eq:usustr2}
\begin{aligned}
\int_\Ocal u \diverg \varphi\, dx = -\int_\Ocal \varphi \cdot \nabla u\,
dx
+\int_{\partial \Ocal} \varphi\Tr u \cdot \nu_\Ocal\, d {\mathcal H}^{\ell-1},
\end{aligned}
\end{equation}
    and 
\begin{equation}
\label{eq:lebptr}
\begin{aligned}
 \lim_{r \to +\infty} \frac{1}{\mathcal L^N(\Ocal\cap B_r(x_0))}\int_{\Ocal
\cap B_r(x_0)}|u(x)-\Tr u(x_0)|\,d x =0 \quad \text{for $\mathcal H^{\ell-1}$-a.e.~$x_0
\in \partial \Ocal$.}
\end{aligned}
\end{equation}
Finally, the operator $\Tr$ is surjective for \(p=1\).
\end{theorem}

\begin{remark}[Trace in \(BV\)]\label{rmk:trBV} The preceding theorem can
be extended to \(BV\), giving rise to a bounded, linear, and surjective operator
 $\Tr: BV(\Ocal) \to
L^1(\partial \Ocal;\mathcal H^{\ell-1})$ such that \eqref{eq:usustr} and
\eqref{eq:lebptr} hold for every $u \in BV(\Ocal) \cap C(\overline \Ocal)$,
while \eqref{eq:usustr2} becomes
\begin{equation*}
\begin{aligned}
\int_\Ocal u \diverg \varphi\, dx = -\int_\Ocal \varphi \cdot d D u +\int_{\partial
\Ocal} \varphi \Tr u \cdot \nu_\Ocal \,d {\mathcal H}^{\ell-1}.
\end{aligned}
\end{equation*}
We further observe that this trace operator in \(BV\) is continuous with
respect to the strict convergence in \(BV\).   
\end{remark}

As usual, both in \(W^{1,p}\)
and \(BV\), the bounded linear operator \( \Tr\) above is called the trace
operator on \(\partial \Ocal\) and \(\Tr u\) the trace of \(u\) on \(\partial\Ocal\).

\begin{remark}[Trace in \(BH\)] Recalling the embedding in Proposition~\ref{embedding}, we give sense to \(\Tr(u)\) for \(u\in BH(\Ocal)\) using Theorem~\ref{TrBV}. We note that are using this  identification in \eqref{eq:UcalXi}.  
\end{remark}

 Following \cite[Appendix]{D}, we can also  give sense to the trace
 for $W^{2,1}$ functions as follows.  
 If $\calO \subset \RR^\ell$ is an open and
$C^2$-uniformly regular set, then the trace operator  $\Tr \colon W^{2,1}(\calO)
\to W^{1,1}(\partial \calO)$ is linear and continuous, in the sense that
the tangential gradient of the trace $\nabla_\tau(\Tr u)$ coincides with
the projection of the trace $\Tr
(\nabla u)$ onto the tangent space of $\partial \calO$ where, with an abuse
of notation, this latter trace operator is the one in Theorem~\ref{TrBV};
in other words,$$
\nabla_\tau (\Tr u)= (\Tr (\nabla u))_\tau.$$

On the other hand, under the same assumptions on $\partial \calO$ and denoting
by \(\nu_{\partial\calO}\) the outer normal to \(\calO \), we can set  $(\Tr
(\nabla u))_\nu:=\Tr (\nabla u)\cdot \nu_{\partial \calO}\in L^1(\partial
\calO)$ for
every $u \in W^{2,1}(\calO)$.
Then, \cite[Theorem~1 in the Appendix]{D}, which we recall
next for the readers' convenience, hold.

\begin{theorem}[cf.~{\cite[Theorem~1]{D}}]\label{Dthm1}
The operator $\Tr^{2}\colon W^{2,1}(\calO)\to \Tr (W^{2,1}(\calO))\times
L^1(\partial \calO)$
defined for \(u\in W^{2,1}(\calO) \) by
$$
\Tr^{2}u := (\Tr (u), (\Tr (\nabla u))_{\nu})
$$
is linear, continuous, and surjective.
\end{theorem}

We further recall that  $\Tr
(W^{2,1}(\calO))$ is a proper subset of $W^{1,1}(\calO)$, as proved in \cite[Theorem~2
in the Appendix]{D}. 

In this manuscript, we mostly deal with functions with prescribed trace on a portion of the boundary with positive Hausdorff measure.
Precisely, for \(\Ocal:=\Omega^a\) and \(\Gamma:=\Gamma^a=\omega\times\{1\} \) or  \(\Ocal:=\Omega_b\) and \(\Gamma:=\Gamma_b=\partial\omega\times\,]-1,0[ \), as in the Introduction, we set
\begin{equation*}
\begin{aligned}
&W^{1,1}_{\Gamma}(\Ocal;\RR^m):=\left\{u\in W^{1,1}(\Ocal;\RR^m) \colon \Tr(u)\lfloor\Gamma=0\right\},\\
&W^{2,1}_{\Gamma}(\Ocal;\RR^m):=\left\{u\in W^{2,1}(\Ocal;\RR^m) \colon \Tr(u)\lfloor\Gamma=0, \, \Tr(\nabla u)\lfloor\Gamma=0\right\}.
\end{aligned}
\end{equation*}
In the particular case  in which our functions depend only on the last variable, \(\Ocal:=]0,1[\) and \(\Gamma:=\{1\}\), we use the notation above and  set
\begin{equation*}
\begin{aligned}
&W^{1,1}_{\Gamma^a}(]0,1[;\RR^m):=\left\{u\in W^{1,1}(]0,1[;\RR^m) \colon \Tr(u)(1)=0\right\},\\
&W^{2,1}_{\Gamma^a}(]0,1[;\RR^m):=\left\{u\in W^{2,1}(]0,1[;\RR^m) \colon \Tr(u)(1)=0,
\, \Tr(\nabla u)(1)=0\right\}.
\end{aligned}
\end{equation*}
However, in the case in which our functions depend only on the first \(N-1\) variables,
\(\Ocal:=\omega\) and \(\Gamma:=\partial\omega\), we use the standard notation %
\begin{equation*}
\begin{aligned}
&W^{1,1}_{0}(\omega;\RR^m):=\left\{u\in W^{1,1}(\omega;\RR^m) \colon
\Tr(u)=0\right\},\\
&W^{2,1}_{0}(\omega;\RR^m):=\left\{u\in W^{2,1}(\omega;\RR^m) \colon
\Tr(u)=0,
\, \Tr(\nabla u)=0\right\}.
\end{aligned}
\end{equation*}

If we do not assume   \(\partial\Ocal\) to be Lipschitz,  we follow   \cite{KR} and  give sense to the trace of a \(BV\) or \(BH\) function
as follows. Given
 $u \in BV(\calO;\mathbb R^m)$, with \(\Ocal\subset\RR^\ell\) open and bounded,
we define the spaces 
\begin{equation}\label{W11u}W^{1,1}_u(\calO;\mathbb R^m):=\big\{v \in W^{1,1}(\calO;\mathbb
R^m)\colon v_u \in BV(\mathbb R^\ell;\mathbb R^m) \hbox{ and } |D v_u|(\partial
\calO)=0\big\}\end{equation}
and
\begin{equation*}
        BV_u(\calO;\mathbb R^m):=\big\{v \in BV(\calO;\mathbb R^m)\colon
 v_u\in BV(\mathbb R^\ell;\mathbb R^m) \hbox{ and } |Dv_u|(\partial \calO)=0\big\},
\end{equation*}
where
\begin{equation*}
        \begin{aligned}
                v_u:= \begin{cases}
                        u-v &\hbox{ in }\calO,\\
                        0 & \hbox{ in }\mathbb R^\ell \setminus \overline
\calO.
                \end{cases}
        \end{aligned}
\end{equation*}

Similarly, if  $u \in BH(\calO;\mathbb R^m)$, we set 
 \begin{equation}\label{W21u}W^{2,1}_u(\calO;\mathbb R^m):=\big\{v \in W^{2,1}(\calO;\mathbb
R^m)\colon  v_u \in BH(\mathbb R^\ell;\mathbb R^m) \hbox{ and } |D^2 v_u|(\partial
\calO)=0\big\}\end{equation}
and
\begin{equation*}
        BH_u(\calO;\mathbb R^m):=\big\{v \in BH(\calO;\mathbb R^m)\colon
v_u\in
BH(\mathbb R^\ell;\mathbb R^m) \hbox{ and } |D^2v_u|(\partial \calO)=0\big\}.
\end{equation*}

The next two results  allow us to approximate functions in \(BV(\calO;\mathbb
R^m)\)  or 
\( BH(\calO;\mathbb R^m)\) by a sequence of smooth functions in \(W^{1,1}_u(\calO;\mathbb
R^m)\) or \(W^{1,2}_u(\calO;\mathbb
R^m)\), respectively, in the (area) \(\langle\cdot\rangle\)-strict sense.
Note that only a mild regularity
        assumption is imposed on $\partial \calO$.
The first of these results is proved in  \cite[Lemma~1]{KR}. Using similar
arguments, we prove the second one in Section~\ref{sect:aux}. 

\begin{theorem}[cf.~{\cite[Lemma~1]{KR}}]\label{thm:KRBV}
        Let $\calO\subset \mathbb R^\ell$ be a nonempty bounded
and open set.  For each
$\xi \in BV(\calO;\mathbb R^m)$, there exists $\{\xi_j\}_j\subset W^{1,1}_\xi(\calO;\mathbb
R^m)\cap C^\infty(\calO;\mathbb R^m)$ such that $\xi_j \to \xi$ in $L^{1}(\calO;\mathbb
R^m)$ and $D\xi_j$
(area) $\langle\cdot\rangle$-strictly in $\mathcal M(\calO;\mathbb
R^{ m\times \ell})$ to $D\xi $. If $\xi
\in W^{1,1}(\calO;\mathbb R^m)$,
then we may further assume  that $\xi_j \to \xi$ strongly
in $W^{1,1}(\calO;\mathbb R^m)$.
\end{theorem}

\begin{theorem}\label{Lemma1KRBH}
        Let $\calO\subset \mathbb R^\ell$ be a nonempty bounded
and open set, with  \(\mathcal{L}^\ell(\partial\calO)=0\).  For each
$u \in BH(\calO;\mathbb R^m)$, there exists $\{v_j\}_j\subset
W^{2,1}_u(\calO;\mathbb
R^m)\cap C^\infty(\calO;\mathbb R^m)$ such that $v_j \to u$ in
$W^{1,1}(\calO;\mathbb R^m)$ and ${{D(\nabla v_j)}}
\to D^2 u$
(area) $\langle\cdot\rangle$-strictly in $\mathcal M(\calO;{{\mathbb
(\RR^{\ell \times \ell}}})^m)$. If $u
\in W^{2,1}(\calO;\mathbb R^m)$,
        then we may further assume  that $v_j \to u$ strongly
in $W^{2,1}(\calO;\mathbb R^m)$.
        \end{theorem}

Using the these two last results,  we obtain the following  area-strict convergence
result in  $BH \times BV$.

\begin{theorem}\label{thm:mixedKRBV}
Let $\calO\subset \mathbb R^\ell$ be a nonempty bounded
and open set, with  \(\mathcal{L}^\ell(\partial\calO)=0\).  For each
$(u,\xi) \in BH(\Omega; \mathbb R^{m_1})\times BV(\calO;\mathbb R^{m_2})$,
there exists $\{(v_j, \xi_j)\}_j\subset W^{2,1}_u(\calO;\mathbb
R^{m_1})\cap C^\infty(\calO;\mathbb R^{m_1})\times W^{1,1}_\xi(\calO;\mathbb
R^{m_2})\cap C^\infty(\calO;\mathbb R^{m_2})$ such that $(v_j,\xi_j) \to
(u,\xi)$ in $W^{1,1}(\Omega;\mathbb R^{m_1})\times L^1(\calO;\mathbb R^{m_2})$
and $(D(\nabla v_j), D\xi_j) \to (D^2 u, D\xi)$
(area) $\langle \cdot\rangle$-strictly  in $\mathcal M(\calO;\mathbb {{\mathbb
(\RR^{\ell \times \ell}}})^{m_1}\times  {{
\RR^{ m_2 \times \ell}}})$. If $(u,\xi)
\in W^{2,1}(\calO;\mathbb R^{m_1})\times W^{1,1}(\calO;\mathbb R^{m_2})$,
then we may further assume  that $(v_j,\xi_j) \to (u,\xi)$ strongly
in $W^{2,1}(\calO;\mathbb R^{m_1})\times W^{1,1}(\calO;\mathbb R^{m_2})$.
\end{theorem}

\begin{proof}
By ~Theorems \ref{thm:KRBV} and \ref{Lemma1KRBH}, we have that $v_j \to u$
in $W^{1,1}(\calO;\mathbb R^{ m_1})$ and $D(\nabla v_j) \to D^2 u$ (area)
$\langle \cdot\rangle$-strictly in $\mathcal M(\calO;\mathbb {{\mathbb
(\RR^{\ell \times \ell}}})^{m_1})$ and $\xi_j \to \xi $ in $L^1(\calO;\mathbb
R^{ m_2})$ and $D\xi_j \to D\xi $ (area) $\langle
\cdot\rangle$-strictly in $\mathcal \Mcal(\calO; {{
\RR^{ m_2 \times \ell}}})$.

From \eqref{eq:estsumang} and the separate (area) $\langle
\cdot\rangle$-strict convergence of $D(\nabla v_j)$ and $D\xi_j$, we have
that the sequence
$\langle (D(\nabla v_j), D \xi_j)\rangle_j$ is uniformly bounded. Thus, up
to a subsequence, it converges weakly-\(\ast\) in $\mathcal M(\calO;\mathbb
\mathbb {{\mathbb
(\RR^{\ell \times \ell}}})^{m_1}\times  {{
\RR^{ m_2 \times \ell}}})$ to a measure $\lambda$ which is absolutely continuous
with respect to $\langle( D^2)^a u\rangle (\cdot) + |(D^2)^s u|(\cdot)+ \langle
D^a\xi \rangle (\cdot )+ |D^s\xi|(\cdot)$. On the other hand, the separate
area strict convergence guarantees that $\lambda(\overline \calO)=\lambda(\calO)$;
i.e., area-strict
convergence implies strict convergence, which does not charge
the boundary.
Hence, the result is achieved by applying \cite[Proposition 1.62]{AFP}.
\end{proof}

\subsection{$\mathcal A$-quasiconvexity}\label{sect:Aqcx}

Here, we recall the notion of $\mathcal A$-quasiconvexity and some of its
main properties (cf.~\cite{FMAq}, \cite{Dac1,Dac2}), including its  relationship with convexity and quasiconvexity.
These properties will be useful in the proof of our representation result.

We consider a finite family of linear operators 
$\left(A^{(i)}\right)_{i=1}^\ell \subset \Lin(\mathbb R^m ,\mathbb R^\kappa)$,
with $\Lin(X, Y )$  the vector space of linear mappings from
a vector space $X$
into a vector space $Y$, which associated matrices in \(\RR^{\kappa\times
m}\) we still represent by \(A^{(i)}\). Then, we
define a first-order linear operator \(\Acal\) by setting 
\begin{equation*}
\begin{aligned}
\mathcal A v :=
\sum_{
i=1}^\ell
A^{(i)} \frac{\partial v}
{\partial x_i} \quad \text{for }
v: \mathbb R^\ell \to \mathbb R^m.
\end{aligned}
\end{equation*}
We assume that  $\mathcal A$ satisfies the constant-rank
property, which states that there exists $r \in \mathbb{N}$ such that
\begin{equation*}
\begin{aligned}
\rank \left( \sum^\ell_{i=1}A^{(i)}w_i \right) = r \quad \text{for all }w
\in S^{\ell-1},
\end{aligned}
\end{equation*}
where $S^{\ell-1}$ is the unit sphere in $\mathbb R^\ell$. 

\begin{definition}[\(\Acal\)-free fields] Let \(\Ocal\subset\RR^\ell \) be
an open set. For \(p\in[1,+\infty)\), we say that a function \(u\in L^p(\calO;\RR^m)
\) is \(\Acal\)-free, written \(u\in \ker \Acal\) or \(\Acal u =0\), if
\begin{equation*}
\begin{aligned}
-\sum_{i=1}^\ell \int_{\calO} \left( A^{(i)} u \right) \cdot \frac{\partial
\psi}
{\partial x_i} \, dx =0 \quad\text{for all } \psi\in C^1_c (\calO;\RR^\kappa).
\end{aligned}
\end{equation*}
Similarly, we say
that a function  \(v\in L^p_{per}(\RR^\ell;\RR^m)
\) is \(\Acal\)-free, written \(v\in \ker \Acal\) or \(\Acal v =0\), if
\begin{equation*}
\begin{aligned}
-\sum_{i=1}^\ell \int_{Q} \left( A^{(i)} v \right) \cdot \frac{\partial
\psi}
{\partial x_i} \, dx =0 \quad\text{for all } \psi\in C^1_{per} (\RR^\ell;\RR^\kappa),
\end{aligned}
\end{equation*}
where, we recall, $Q$  is the
unit cube in $\mathbb R^\ell$. 
Analogously, we say that a measure $\mu \in \mathcal M (\Ocal;\RR^m)$ is \(\Acal\)-free, written \(\mu \in \ker \Acal\) or \(\Acal \mu =0\),  if
\begin{equation*}
\Acal \mu=\sum_{i=1}^{\ell} A^{(i)} \partial_i \mu = 0 \hbox{ in }{\mathcal D}'(\Ocal;\RR^k).  
\end{equation*}
\end{definition}

\begin{definition}[\(\Acal\)-quasiconvex functions]\label{def:Aqcx}
We say that a Borel function $f : \mathbb R^m \to \mathbb R$
is $\mathcal A$-quasiconvex if
$$f(\xi) \leq \int_{Q} f(\xi + w(x))\,dx$$ for every \(\xi\in\RR^m\)
and $w \in C^\infty_
{\rm per}(\mathbb R^\ell;\mathbb R^m) \cap \ker\mathcal A$, with
$
\int_{Q}
w(y)dy = 0.$
\end{definition}

\begin{definition}[\(\Acal\)-quasiconvex envelope] Given a Borel
function $f : \mathbb R^m\to R$, the \(\Acal\)-quasiconvex envelope
(or  $\mathcal A$-quasiconvexification)
of $f$ at $\xi \in \mathbb R^m$ is given by
$$Q_{\mathcal A}f(\xi) := \inf\left\{
\int_Q
f(\xi + w(x))dx : w \in C^\infty_{\rm per}
(\mathbb R^\ell;\mathbb R^m) \cap \ker{\mathcal A},
\, \int_Q
w(y)dy = 0
\right\}.$$
\end{definition}

Note that \(Q_{\mathcal A}f\) is the greatest \(\Acal\)-quasiconvex
function majorized by \(f\). Moreover, as  observed in the following
remark,  the notion of $\mathcal A$-quasiconvexification  extends
to
the $\mathcal A$-free setting
the notion of convexity and quasiconvexity.

\begin{remark}[Well-known properties of $\mathcal A$-quasiconvexity,
cf.~\cite{BrFoLe00,FMAq}]
\begin{itemize}
\item[(i)] In the \(\Acal\equiv 0\) case, which corresponds to
no  constraints on the admissible fields, $\mathcal A$-quasiconvexity
reduces to convexity by Jesen's inequality. In particular, given
a function 
$f : \mathbb R^m \to \mathbb R$, we have that \(Q_{\mathcal A}f\equiv
f^{**}\)  for  \(\Acal\equiv 0\), where 
$$
f^{**} (\xi)= \sup\{g(\xi)\colon g \leq f ,\, g \hbox{ is convex}\,
\}, \quad \xi\in\RR^m. $$

\item[(ii)] In the \(\Acal\equiv \curl \) case, meaning 
$$\Acal v =0 \quad \text{if, and only if,} \quad \frac{\partial
v_{jk}}
{\partial x_i}-\frac{\partial v_{ji}}{\partial x_k}
= 0 \enspace \text{for all } 1 \leq j \leq m, 1 \leq i, k \leq
\ell,$$
 where $v : \calO \to \mathbb R^{m\times \ell}$ is a matrix-valued
function,  then
\begin{align*}
\ker\left( \sum^\ell_{i=1}A^{(i)}w_i \right) &= \left\{V \in \mathbb
R^{m\times \ell} \colon \left( \sum^\ell_{i=1}A^{(i)}w_i \right)V =
0\right\} \\
&= \left\{V \in \mathbb R^{m\times \ell} : V = a\otimes w
\hbox{ for some } a \in \mathbb R^m\right\}\quad \text{for \(w\in
S^{\ell-1}.\)}
\end{align*}
In this case,  $\mathcal A$-quasiconvexity reduces to quasiconvexity,
and the condition in Definition~\ref{def:Aqcx} becomes 
$$f(\xi)\leq  \int_Qf(\xi + \nabla w(x))\,dx \quad \hbox{ for
every  \(\xi\in\RR^{m\times \ell}\) and
 }w \in C^\infty_0(Q;\mathbb R^m),
$$
while \(Q_{\mathcal A}f\equiv Qf\),
  where
$$
Qf(\xi) := \inf\left\{
\int_Q
f(\xi + \nabla w(x))\,dx\colon w \in C^\infty_0 (Q;\mathbb R^m)\right\}.$$
\item[(iii)] In the case of second-order gradients, in which
\(v\in L^p(\Ocal;(\RR^{\ell\times\ell}_s)^m)\) and \(\Acal v=0\)
if, and only if, \(v=\nabla^2 \phi\) for some \(\phi\in W^{1,p}(\Ocal,\RR^m)\),
then    $\mathcal A$-quasiconvexity reduces to 2-quasiconvexity,
where the latter is the notion in (ii) with ``\(\nabla w\)" replaced
by ``\(\nabla^2 w\)".
\end{itemize}
\end{remark}

In our context, with the set-up introduced in the Introduction, we deal with
the  first-order linear differential
operator of constant rank given by
\begin{equation}\label{Atilde}
\begin{aligned}
\widetilde{\mathcal{A}}: (v, \zeta) \to \widetilde{\mathcal{A}}(v,
\zeta)
= ( \mathcal{A}^1v, \mathcal{A}^2\zeta)
\end{aligned}
\end{equation}
where, for  $v:\mathbb R^{N-1} \to \big( \mathbb R_s^{(n-1)\times
(N-1)} \big)^d$,
\begin{equation*}
\begin{aligned}
\mathcal{A}^1v := \Big( \frac{\partial}{\partial x_i} v_{kj}
- \frac{\partial}{\partial
x_j} v_{ki} \Big)_{1 \leq i,j\leq N-1, 1\leq k \leq d} \in \mathbb
R^d
\end{aligned}
\end{equation*}
and, for $\zeta: \mathbb R^{N-1} \to (\mathbb R^{N-1})^d$,
\begin{equation*}
\begin{aligned}
\mathcal {A}^2\zeta := \Big( \frac{\partial}{\partial x_i}\zeta_j
- \frac{\partial}{\partial
x_j}\zeta_i\Big)_{1 \leq i,j\leq N-1} \in \mathbb R^d.
\end{aligned}
\end{equation*}

Following \cite{FMAq,GZNODEA}, the Kernel of the operator $\widetilde{\mathcal
{A}}$ can be characterized as follows.
If $(x_1, x_2) \in \ker \mathcal{A},$ then
there exist \(w_1, w_2 \in \mathbb
{R}^{N-2}\) and \(b_1, b_2\in\RR^d\) such that
\begin{equation*}
\begin{aligned}
x_1 = b_1 \otimes w_1\otimes w_1 \quad \text{and} \quad x_2 =
b_2 \otimes w_2.
\end{aligned}
\end{equation*}

We further refer to  \cite[Remark 2.4]{GZNODEA} for more details
on the description of the operator  and its kernel in the case
in which $d=3$ and $N=3$. In particular, $\widetilde{\mathcal
A}$ is a constant-rank operator. Moreover (cf.~ \cite[(2.8)--(2.12)]{GZNODEA}),
the $\widetilde{\mathcal A}$-quasiconvexification of the function
\(W_0\) in \eqref{eq:defhatzero}, \(Q_{\widetilde\Acal} W_0\),
corresponds to the cross
2-quasiconvexification--quasiconvexification, where the 2-quasiconvexification
 is with respect to the first argument, while the quasiconvexification
is with respect to the  second argument.

\section{Auxiliary results}

The main goal of this section is to establish certain lower semicontinuity
results involving sequences with prescribed boundary conditions, which we
will use to prove our main results stated in the Introduction. 

We first introduce functionals defined on measures, which will be useful for
our analysis.
Let \(\Ocal\subset\RR^\ell\) be an open and bounded set. Given $\mu=m^a\mathcal
L^\ell+\mu^s\in\mathcal M(\Ocal ;\mathbb R^m)$ and $\Phi\colon \Ocal\times
\mathbb R^m\to[0,+\infty)$ continuous, let 
\begin{equation}\label{959}
        {\mathcal I}(\mu) := \int_{\Ocal} \Phi ( m^a(x))\,dx +\int_{\Ocal}
\Phi^\infty \bigg(\frac{d \mu^s}{d |\mu^s|}(x)\bigg)\,d |\mu^s|(x),
\end{equation}
where $\Phi^\infty$ is the \emph{recession function} of $\Phi$ at infinity,
defined by
\begin{equation*}\label{defrecession}
        \Phi^\infty(\xi):= \limsup_{\substack{\xi' \to \xi\\
t\to +\infty}}\frac{\Phi(x',t \xi')}{t}
\end{equation*}
for every 
$\xi\in\mathbb S^{m-1}$ and extended to
$\mathbb R^m$ by positive $1$-homogeneity.

As  in \cite[Subsection 2.3]{KR1}, we observe that if $\Phi$
is globally Lipschitz, then the above definition reduces to
 \begin{equation}\label{S101}
            \Phi^\infty(\xi):=\limsup_{ t\to+\infty}\frac{\Phi(t\xi)}{t},
\quad \xi\in\RR^\scrk.
                \end{equation}
In this particular case, Reshetnyak upper-semicontinuity theorem (see \cite[Corollary 2.11]{BCMS}) specializes as follows.
                
 \begin{theorem}\label{ReshetnyaktheoremBCMS}
                For \(n\in\NN\), let $\mu_n$, $\mu \in\mathcal M(\Ocal;\mathbb R^m)$ 
                be such that  $(\mu_n)_{n\in\NN}$ (area) $\langle\cdot\rangle$-strictly
converges to $\mu$.
                Let $\Phi\colon \Ocal\times\mathbb R^{m}\to[0,+\infty)$ be
a Lispchitz function, and consider the functional  
\(\mathcal{I}\) in \eqref{959}. Then, 
                                $${\mathcal I
                }(\mu)\geq\limsup_{n\to\infty} \mathcal I(\mu_n).$$
               
        \end{theorem}

Observe also that if $\Phi$ a Lipschitz function, then there exists $C_\Phi>0$ such that \begin{equation*}
 \Phi(\xi)\leq C_\Phi(1+ |\xi|)
\end{equation*}
for every $\xi \in \mathbb
R^\scrk$.

To handle the lower semicontinuity  in the $\mathcal A$-free setting with
linear growth conditions, we can refer to \cite[Theorem 1.1]{BCMS} or \cite[Theorem~1.2]{ARDPR}.
Indeed the following result, which will be applied in the sequel with $\mathcal A$ replaced by $\tilde{\mathcal A}$ in \eqref{Atilde}, holds.

\begin{theorem}\label{thmARDPR}
Let $\scrk \in \mathbb N,$   $ \Phi\colon \mathbb R^\scrk \to [0,+\infty)$
 a Lipschitz and $\mathcal A$-quasiconvex
integrand, and \(\calO\subset\RR^\ell\) an open and bounded Lipschitz set. 
Then, the functional  
\(\mathcal{I}\) in \eqref{959} 
is sequentially weakly-\(\ast\) lower semicontinuous on the space
$\mathcal M(\calO;\mathbb R^\scrk )\cap{\rm ker}\mathcal A :=
\{
\mu \in \mathcal M(\calO;\mathbb R^\scrk )\colon \mathcal  A \mu = 0\}$.
\end{theorem}

\begin{remark}[On the assumptions in Theorem~\ref{thmARDPR}]\label{hypfinfty}
In \cite{ARDPR}, Theorem~\ref{thmARDPR} is proved under additional assumptions
on the recession function, \(\Phi^\infty\) (cf. \cite[Theorem~1.2]{ARDPR}).
However, as pointed out in \cite[Remark~1.11]{ARDPR}, these additional conditions
 are only needed  to handle differential operators of order higher than one.
In the first-order case, which is the context of  this manuscript, such conditions
on \(\Phi^\infty\) are not necessary.

We further emphasize that Theorem~\ref{thmARDPR} is proved in \cite[Theorem
1.1]{BCMS}, also under additional assumptions that, as we address next, are
not needed in our case.
Precisely, our fields $(D^2_{x'} u^b_n, \frac{1}{h_n} D^2_{x',x_N}u^b_n)$
and $( \frac{1}{r_n}D^2_{x',x_N}u^a_n, D^2_{x_N}u^a_n)$, regarded as measures
$\mu_n^b$ and $\mu^a_n$ on \(\Omega ^a\) and \(\Omega^b\), respectively,
satisfy all the assumptions of \cite[Theorem 1.1]{BCMS}, up to the fact that
\begin{align}\label{noLambda0}
 \Lambda^a(\partial \Omega^a)= \Lambda^b(\partial \Omega^b)=0,
\end{align} 
where \(\Lambda^a\in\mathcal M(\overline
\Omega^a)  \) and \(\Lambda^b\in\mathcal  M(\overline
\Omega^b) \) satisfy $|\mu_n^a|\overset{\ast}{\rightharpoonup} \Lambda^a$
in $\mathcal M(\overline
\Omega^a)$ and $|\mu_n^b|\overset{\ast}{\rightharpoonup} \Lambda^b$ in $\mathcal
M(\overline
\Omega^b)$.  
 However, a careful inspection of the proof shows that the above condition
\eqref{noLambda0} on the limiting measure on the boundary  was needed to
get \cite[(3.5)]{BCMS}, which is trivially verified in our context because
our  densities are positive.
Moreover,
 \cite[Theorem 1.1]{BCMS} relies also on \cite[Proposition 3.1]{BCMS} which
proves \cite[(3.12) and (3.13)]{BCMS}, where the condition \eqref{noLambda0}
above is used just for  \cite[(3.5)]{BCMS}. Also,  the remaining arguments
 invoking conditions of the type $\Lambda^a(\partial Q')$ and $\Lambda^b(\partial
Q')=0$ can be easily verified without the assumption \eqref{noLambda0}. For
instance,  such condition on the boundary of cubes is necessary to apply
\cite[Lemma 2.20]{BCMS}  but also in this case, the argument  relies on a
suitable choice of cubes and not on the assumption \eqref{noLambda0}. Finally,
we observe that the conditions of \cite[Proposition~2.22]{BCMS}, which plays
an important role in the proof of \cite[Theorem 1.1]{BCMS}, are automatically
verified by our fields provided we ensure that \cite[(iii) Proposition
2.22]{BCMS} holds. The latter   follows from \cite[Corollary 5.8]{H} for
the fields that
are hessians (or part of hessians) and \cite[Theorem 1.4]{Giusti}
for the fields that are gradients (or part of gradients).

We also remark that in our subsequent application the linear growth of $\Phi$ and its $\mathcal A$-qasiconvexity with respect to the operator $\tilde{\mathcal A}$ in \eqref{Atilde} imply that $\Phi$ is Lipschitz.
\end{remark}

Next, using the trace operator discussed in Subsection~\ref{subsect:trace},
 we establish several results concerning approximating sequences with certain
prescribed boundary traces.

\begin{proposition}\label{asLemma4.1BFMTraces}
Let $\calO \subset \RR^\ell$ be a bounded  open set with \(C^1\) boundary.
For every $\theta \in L^1(\partial \calO;\mathbb
R^m)$  and $\zeta \in BV(\calO;\mathbb
R^m)$, there exists a sequence $\{\zeta _\varepsilon\}_\epsi \subset W^{1,1}(\calO;\mathbb
R^m)$ such that $\Tr \zeta _\varepsilon= \theta$, $\zeta _\varepsilon
\to \zeta $ in $L^1(\calO;\mathbb R^m)$, 
and 
$$
\limsup_{\varepsilon \to 0} \int_\calO \sqrt{1+ |\nabla\zeta _\varepsilon|^2}
dx \leq \int_\calO \sqrt{1+ |\nabla\zeta |^2}dx + |D^s \zeta |(\calO
)+ \int_{\partial \calO} |\theta- \Tr \zeta | d \mathcal H^{\ell-1}.
$$
\end{proposition}

\begin{proof}[Proof]
By  \cite[Lemma~11.1]{Rin}, there exists a sequence
\(\{v_\varepsilon\}_\varepsilon
\subset W^{1,1}(\calO;\mathbb
R^m)\cap C^\infty(\calO;\mathbb
R^m)\) that area strictly converges to $\zeta $
in $\calO$, with $\Tr v_\varepsilon = \Tr  \zeta $. On the other 
hand, by \cite[Theorem~2.16 and Remark~2.7]{Giusti}, we can find a
sequence \(\{\sigma_\varepsilon\}_\varepsilon \subset W^{1,1}(\calO;\mathbb
R^m)\)
such that \(\Tr \sigma_\varepsilon= \theta- \Tr \zeta \), $\int_\calO
|\sigma_\varepsilon|\, dx \leq \varepsilon \int_{\partial \calO} |\theta
-\Tr \zeta |\,d \mathcal H^{\ell-1}$, and $\int_\calO |\nabla
\sigma_\varepsilon|\,dx
\leq (1+\varepsilon )\int_{\partial \calO} |\theta -\Tr \zeta |\,d \mathcal
H^{\ell-1} $.

Then, defining $\zeta _\varepsilon:= v_\varepsilon + \sigma_\varepsilon$,
we have  that $\Tr \zeta _\varepsilon= \theta$ and, also using \eqref{eq:estsumang},
\begin{align*}
\limsup_{\varepsilon \to 0 }\int_\calO \sqrt{1+ |\nabla\zeta _\varepsilon|^2}\,dx
&\leq \limsup_{\varepsilon\to 0}\bigg(\int_{\calO} \sqrt{1+
|\nabla v_\varepsilon|^2\,}dx + \int_{\calO} |\nabla
\sigma_\varepsilon | \,dx \bigg)
\\
&\leq \int_{\calO} \sqrt{1+ |\nabla\zeta |^2}\,dx
+ |D^s \zeta |(\calO)+ \int_{\partial \calO} |\theta - \Tr \zeta |\,d
\mathcal H^{\ell-1}.\qedhere
\end{align*}
\end{proof}

\begin{remark} If $\theta \in W^{1,1}(\calO;\mathbb R^\scrk ) $ is extended
outside $\calO$,
 then we can extend $v_\varepsilon $ as $\theta$ (its boundary trace on $\partial
\Omega$) outside \(\calO\), say \(\calO'\supset \calO\), to get that $v_\varepsilon$
 area-strictly converges in \(\calO'\) to \[\zeta ':=\begin{cases}
\zeta  & \hbox{ in } \calO,\\
 \theta &\hbox{ in } \calO'\setminus \calO.
\end{cases}\]
\end{remark}

\begin{proposition}
    \label{thm2.16G_BH}
    Let $\calO\subset\RR^\ell$ be a bounded open set with $C^1$ boundary,
and let $\varphi \in L^1(\partial \calO;\RR^m)$.  For every $\varepsilon
>0$, there exists  $\psi \in W^{2,1}(\calO;\RR^m)$ such that 
\begin{align}
&(\Tr  (\nabla\psi))_\nu= \varphi, &\label{normalderivative}\\
&\Tr (\psi) =0, \; (\Tr (\nabla\psi))_{\tau}=0, &\label{tracebd}
\end{align}
and
\begin{equation}
\label{normboundtrace}
\begin{aligned}
&\int_{\calO} |\psi|dx \leq \varepsilon \int_{\partial \calO} |\varphi| d
{\mathcal H}^{\ell-1},\\
& \int_{\calO} |\nabla\psi| dx \leq \varepsilon\int_{\partial \calO}
|\varphi| d {\mathcal H}^{\ell-1},\\
&\int_{\calO} |\nabla^2 \psi| dx \leq (1+\varepsilon )\int_{\partial \calO}
|\varphi| d {\mathcal H}^{\ell-1}.
\end{aligned}
\end{equation}
\end{proposition}

 \begin{proof}[Proof]
   Fix \(\varepsilon>0\). Following the proof of \cite[Proposition 1]{D},
which relies on the same construction of \cite[Theorem 2.16]{Giusti} and
on a partition of the unity argument, we can assume $\calO$ to be the cylinder
\(B(0,1)\times]0,1[\), where \(B(0,1)\subset\RR^{\ell-1}\) is the unit ball
centered at the origin of \(\RR^{\ell-1}\), and \(\partial\calO\) to be 
\(B(0,1)\times \{0\}\) (or \(B(0,1)\times \{1\})\), which we identify with
\(B(0,1)\). Moreover, using a density argument,  we can construct a sequence
of $C^\infty(\partial\calO;\RR^m)$ functions,  $(\varphi_k)_{k\in\N}$, such
that 
\begin{equation}
 \begin{aligned}\label{eq:cauchyseq}
 &\lim_k \Vert \varphi_k - \varphi\Vert_{L^1(\partial\calO;\RR^m)} = 0,\\
 &\sup_k\|\varphi_k\|_{L^1(\partial\calO;\RR^m)} \leq 2 \|\varphi\|_{L^1(\partial\calO;\RR^m)},\\
 &\sum_{k=0}^\infty \|\varphi_k- \varphi_{k+1}\|_{L^1(\partial\calO;\RR^m)}
 \leq  (1+\tfrac\varepsilon4)\|\varphi\|_{L^1(\partial\calO;\RR^m)}. &
 \end{aligned}
 \end{equation}
 We can further consider a decreasing sequence $t_k$ of nonnegative numbers
converging to $0$ and such that
\begin{equation}
\label{eq:defts}
\begin{aligned}
t_k- t_{k+1}\leq \frac{\varepsilon\,2^{-k-4}  \|\varphi\|_{L^1(\partial\calO;\RR^m)}
}{{{}}\sup_{
\tilde k\in\{ k, k+1\}} (1+ {{2}}\|\varphi_{\tilde k}\|_{W^{1,2}(\partial\calO;\RR^m)})}.
\end{aligned}
\end{equation}

 Then, for $\bar x: =(x_1,...,x_{\ell-1})\in \partial\calO$ and $t\in(0,1)$,
we define 
 \begin{align*}
 v(\bar x,t):=\left\{
 \begin{array}{ll} 0 &\hbox{ if } t \geq t_0,\\
 \frac{t-t_{k+1}}{t_k- t_{k+1}}\varphi_k(\bar x)+ \frac{t_k-t}{t_k-t_{k+1}}
\varphi_{k+1}(\bar x) & \hbox{ if } t_{k+1}\leq t \leq t_k.
 \end{array}
 \right.\end{align*}
 We start by proving that 
\begin{align}\label{eq:L1conv}
\lim_{t \to 0} v(\cdot, t)= \varphi(\cdot) \hbox{ in }L^1(\partial \calO;\RR^m).
 \end{align}
In fact, given \(t\in]0,1[\) with \(t<t_0\), there exists \(k\in\N\) such
that \(t\in [t_{k+1}, t_k]\). Hence,
\begin{equation*}
\begin{aligned}\int_{\partial \calO} |v(\bar x, t) - \varphi(\bar x)|\, d\bar
x &=
\int_{\partial \calO} \Big|\frac{t-t_{k+1}}{t_k- t_{k+1}}\big(\varphi_k(\bar
x)-\varphi(\bar x)\big)+ \frac{t_k-t}{t_k-t_{k+1}}
\big(\varphi_{k+1}(\bar x)-\varphi(\bar x)\big) \Big|\, d\bar x \\ 
&\leq  \int_{\partial \calO} |\varphi_k(\bar
x) - \varphi(\bar x)|\, d\bar
x + \int_{\partial \calO} |\varphi_{k+1}(\bar
x) - \varphi(\bar x)|\, d\bar
x,
\end{aligned}
\end{equation*}
which, together with \eqref{eq:cauchyseq}  and the fact that \(k\to\infty
\) as \(t\to0\), yields \eqref{eq:L1conv}. 

Consequently,   setting $$\psi(\bar x,x_\ell):=\int_0^{x_\ell} v(\bar x,
t) dt,$$
 we have that \(\psi\in W^{2,1}(\calO;\RR^m) \) satisfies  \eqref{normalderivative}--\eqref{tracebd},
with
\begin{equation*}
\begin{aligned}
&\nabla \psi(\bar x, x_\ell) = \left(\int_0^{x_\ell} \nabla_{\bar x} v(\bar
x,
t) dt, v(\bar x,
 x_\ell) \right) \quad \text{and}\\ 
&\nabla^2  \psi(\bar x, x_\ell) = \left(\begin{matrix}\displaystyle \int_0^{x_\ell}
\nabla^2_{\bar x} v(\bar
x,
t) dt 
&\left(\nabla_{\bar x} v (\bar x, x_\ell)\right)^T\\ \nabla_{\bar x} v (\bar
x, x_\ell) &\displaystyle\frac{\partial v}{\partial x_\ell}(\bar x, x_\ell)
\end{matrix}\right).  
\end{aligned}
\end{equation*}

Furthermore, using \eqref{eq:defts}, we  have that 
\begin{equation*}
\begin{aligned}
\int_{\calO} | \psi|dx &\leq\int\!\!\!\int_{\partial \calO\times ]0,1[}\left|v(\bar
x,
 t)\right|d\bar
x dt \\
&\leq \sum_{k=0}^\infty \int_{\partial\calO} \int_{t_{k+1}}^{t_k}\left|\frac{t-t_{k+1}}{t_k-
t_{k+1}}\varphi_k(\bar x)+ \frac{t_k-t}{t_k-t_{k+1}}
\varphi_{k+1}(\bar x)\right|d\bar
x dt\\
&\leq\sum_{k=0}^\infty |t_k - t_{k+1}|\left(\Vert \varphi_k\Vert_{L^1(\partial
\calO;\RR^m)}
+ \Vert \varphi_{k+1}\Vert_{L^1(\partial \calO;\RR^m)}\right)  \leq \frac\varepsilon4
\|\varphi\|_{L^1(\partial
\calO;\RR^m)}
\end{aligned}
\end{equation*}
and, for \(j, j_1, j_2 \in\{1,...,\ell -1\} \), 
\begin{equation*}
\begin{aligned}
\int_{\calO} |\nabla \psi|dx &\leq \int\!\!\!\int_{\partial \calO\times ]0,1[}\left|\nabla_{\bar
x} v(\bar
x,
t)\right|d\bar
x dt +\int\!\!\!\int_{\partial \calO\times ]0,1[}\left|v(\bar x,
 x_\ell)\right|d\bar
x dx_\ell\\
& =\sum_{k=0}^\infty \int_{\partial\calO} \int_{t_{k+1}}^{t_k}\left|\frac{t-t_{k+1}}{t_k-
t_{k+1}}\nabla\varphi_k(\bar x)+ \frac{t_k-t}{t_k-t_{k+1}}
\nabla \varphi_{k+1}(\bar x)\right|d\bar
x dt+ \frac\varepsilon4 \|\varphi\|_{L^1(\partial
\calO;\RR^m)}  \\
&\leq\sum_{k=0}^\infty |t_k - t_{k+1}|\left(\Vert\nabla  \varphi_k\Vert_{L^1(\partial
\calO;\RR^{m\times (\ell -1)})}
+ \Vert\nabla \varphi_{k+1}\Vert_{L^1(\partial \calO;\RR^{m\times (\ell -1)})}\right)
+ \frac\varepsilon4 \|\varphi\|_{L^1(\partial
\calO;\RR^m)}  \\
&\leq \frac\varepsilon2 \|\varphi\|_{L^1(\partial
\calO;\RR^m)}.
\end{aligned}
\end{equation*}
Similarly,
recalling \eqref{eq:cauchyseq},%
\begin{equation*}
\begin{aligned}
\int_{\calO} |\nabla^2 \psi|dx &\leq\frac{3\varepsilon}4 \|\varphi\|_{L^1(\partial
\calO;\RR^m)} + \sum_{k=0}^\infty \int_{\partial\calO} \int_{t_{k+1}}^{t_k}\left|\frac{\varphi_k(\bar
x) -\varphi_{k+1}(\bar x) }{t_k-
t_{k+1}}\right|d\bar
x dt \\
&\leq\frac{3\varepsilon}4 \|\varphi\|_{L^1(\partial
\calO;\RR^m)} +  \left(1+\frac\varepsilon4\right)\|\varphi\|_{L^1(\partial\calO;\RR^m)}
=(1+\varepsilon)\|\varphi\|_{L^1(\partial\calO;\RR^m)},
\end{aligned}
\end{equation*}
which concludes the proof of  \eqref{normboundtrace} and of the proposition.
  \end{proof}

\begin{proposition}\label{asLemma4.1BFMTracesBH}
Let $\calO\subset\RR^\ell$ be a bounded open set with $C^1$ boundary.
For every $\vartheta \in W^{2,1}(\calO;\mathbb R^\scrk )$  and $w \in {{\vartheta+BH(\calO;\mathbb
R^\scrk )}}$ with \(\Tr(w)=\theta\), there exists a sequence $w_\varepsilon \in W^{2,1}(\calO;\mathbb
R^\scrk )$ such that $\Tr ^2(w_\varepsilon)= \Tr ^2(\vartheta)$, $w_\varepsilon
\to w$ in $W^{1,1}(\calO;\mathbb R^\scrk )$,
and 
\begin{equation*}
\begin{aligned}
\limsup_{\varepsilon \to 0} \int_\calO \sqrt{1+ |\nabla^2 w_\varepsilon|^2}
\,dx \leq \int_\calO \sqrt{1+ |\nabla^2w|^2}\,dx + |D^s(\nabla w) |(\calO)+
\int_{\partial
\calO} |(\Tr (\nabla\vartheta-\nabla w))_{\nu_{\partial\calO}}| d \mathcal
H^{\ell-1}.
\end{aligned}
\end{equation*}
\end{proposition}

\begin{proof}
As we show next, it suffices to define $w_\varepsilon:= v_\varepsilon + \psi_\varepsilon$,
where $v_\varepsilon \in W^{2,1}_w(\calO ;\mathbb R^\scrk )$ is given by
Lemma~\ref{Lemma1KRBH} with \(u:=w\) and $\psi_\varepsilon\in W^{2,1}(\calO;\RR^m)$
is given by Proposition \ref{thm2.16G_BH} with \(\varphi:=\Tr (\nabla\vartheta-
\nabla w)\cdot\nu_{\partial\calO}\).

By 
Lemma~\ref{Lemma1KRBH}, we have that \( \Tr ^2(v_\varepsilon) = \Tr ^2 (w)\),
$v_\varepsilon \to w$ in $W^{1,1}(\calO ;\mathbb R^\scrk ), \nabla^2 v_\varepsilon
\overset{\ast}{\rightharpoonup} D^2 w$ in $\mathcal M(\calO;(\mathbb R^{\ell\times\ell})^\scrk
 )$, and 
\begin{equation*}
\begin{aligned}
\lim_{\varepsilon\to0}\int_\calO \sqrt{1+ |\nabla^2 v_\varepsilon|} \,dx
= \int_\calO \sqrt{1+
|\nabla^2 w|}\,dx + |D^s(\nabla w)|(\calO).
\end{aligned}
\end{equation*}

On the other hand, by Proposition \ref{thm2.16G_BH}, we have that  $\Tr (\psi_\varepsilon)=0$,
 $(\Tr (\nabla\psi_\varepsilon))_{\nu_{\partial \calO}}=(\Tr (\nabla \vartheta-
\nabla w))_{\nu_{\partial \calO}} $,
 $\psi_\varepsilon \to 0$ in $L^1(\calO;\mathbb R^\scrk )$, and
\begin{equation*}
\begin{aligned}
\limsup_{\varepsilon \to 0}\int_\calO|\nabla^2 \psi_\varepsilon|\,dx \leq
\int_{\partial \calO} |(\Tr (\nabla\vartheta-\nabla w))_{\nu_{\partial \calO}}|\,
d {\mathcal H}^{\ell-1}.
\end{aligned}
\end{equation*}
Observe further that \eqref{normboundtrace}, together with the compact embedding
of $BH$ in $W^{1,1}$,
guarantee that $\psi_\varepsilon \to 0$ in $W^{1,1}(\calO;\mathbb R^\scrk
)$.

Hence, recalling \eqref{eq:estsumang},
\begin{align*}
\limsup_{\varepsilon \to 0 }\int_\calO \sqrt{1+ |\nabla^2 w_\varepsilon|^2}\,dx&
\leq \limsup_{\varepsilon\to 0}\left(\int_{\calO} \sqrt{1+ |\nabla^2 v_\varepsilon|^2}\,dx
+ \int_{ \calO} |\nabla^2 \psi_\varepsilon|\, dx \right)
\\
&  \leq \int_\calO \sqrt{1+ |\nabla^2 w|^2}\,dx + |D^s (\nabla w)|(\calO)+
\int_{\partial
\calO} |(\Tr (\nabla\vartheta-\nabla w))_{\nu_{\partial \calO}}|\,d \mathcal
H^{\ell-1}.\qedhere
\end{align*}
\end{proof}

\begin{proposition}\label{BFMTracestog}
Let $\calO\subset\RR^\ell$ be a bounded open set with $C^1$ boundary.
For every  $w \in {{BH_0(\calO;\mathbb R^\scrk )}}$
 with \(\Tr(w)=\vartheta\) for some $\vartheta \in W^{2,1}(\calO;\mathbb R^\scrk )$,  and for every $\theta \in L^1(\partial
\calO;\RR^m)$ and $\zeta \in BV(\calO;\mathbb R^m)$, there exists a sequence
$(w_\varepsilon, \zeta _\varepsilon) \in W^{2,1}(\calO;\mathbb R^\scrk )\times
W^{1,1}(\calO;\mathbb
R^m)$ such that $(w_\varepsilon, \zeta _\varepsilon) \to (w,\zeta )$ in $W^{1,1}(\calO;\mathbb
R^\scrk )\times L^1(\calO;\mathbb R^m)$, $\Tr ^2(w_\varepsilon)= \Tr ^2(\vartheta)$,
$\Tr (\zeta _\varepsilon)= \theta$, and 
\begin{equation*}
\begin{aligned}
\limsup_{\varepsilon \to 0} \int_\calO \sqrt{1+ |(\nabla ^2 w_\varepsilon,
\nabla\zeta _\varepsilon)|^2}
\,dx &\leq \int_\calO \sqrt{1+ |(\nabla^2 w, \nabla\zeta )|^2}\,dx  + |(D^2_s
w, D^s
\zeta )|(\calO
)\\&\qquad+ \int_{\partial\calO} |((\Tr (\nabla \vartheta-\nabla w))_{\nu_{\partial
\calO}},\theta-
\Tr \zeta )|\, d \mathcal H^{\ell-1}.
\end{aligned}
\end{equation*}
\end{proposition}

\begin{proof}[Proof]
Take $w_\varepsilon:= v_\varepsilon + \psi_\varepsilon$ and $\zeta _\varepsilon:=
\tilde v_\varepsilon + \sigma_\varepsilon$  as in the proof of Propositions
\ref{asLemma4.1BFMTracesBH} and \ref{asLemma4.1BFMTraces}, respectively.
Then, applying \eqref{eq:estsumang}, Theorem \ref{thm:mixedKRBV} to the sequence
$(v_\varepsilon, \tilde v_\varepsilon)$, and with a construction entirely
analogous to the ones in the latter results for $\psi_\varepsilon$  and $\sigma_\varepsilon$,
we conclude the proof.
\end{proof}

\begin{remark}
    \label{afterProp 4.8}
    The argument to construct $\sigma_\varepsilon$ and $\psi_\varepsilon$
in Propositions \ref{asLemma4.1BFMTraces} and \ref{asLemma4.1BFMTracesBH},
respectively, is based on a partition of the unity, restricting first to
cylinders where the boundary data $\vartheta$ and $\theta$ are prescribed
just either on the top or bottom of the cylinders. Hence, if $\calO  =\,\,]0,1[$,
we can say that there exists a sequence $(w_\varepsilon, \zeta _\varepsilon)
\in W^{2,1}(]0,1[, \mathbb R^\scrk )\times W^{1,1}(]0,1[;\mathbb R^m) $ such
that $\Tr^2 (w_\varepsilon)= \Tr^2 (\vartheta)$ on $\{1\}$,  $\Tr \zeta _\varepsilon=
\Tr \theta $ on $\{1\}$,   $(w_\varepsilon, \zeta _\varepsilon) \to (w,\zeta
)$ in $W^{1,1}(\calO;\mathbb R^\scrk )\times L^1(\calO ;\mathbb R^m)$,  and
\begin{equation*}
\begin{aligned}
 \limsup_{\varepsilon \to 0} \int_0^1\sqrt{1+ |(\nabla ^2 w_\varepsilon,
\nabla\zeta _\varepsilon)|^2}
\,dx_N 
& \leq \int_0^1 \sqrt{1+ |(\nabla^2 w, \nabla\zeta )|^2}\,dx_N  + |(D^2_s
w, D^s \zeta )|(]0,1[)\\ &\qquad +  \left|\bigg( \frac{d\vartheta}{d x_N}(1)-
\left(\frac{d w}{d x_N}\right)^-(1),
 \theta(1)-  \zeta ^-(1)\bigg)\right|.
\end{aligned}
\end{equation*}
\end{remark}

We are now in position to apply the first result above to the context of
our $ \widetilde\Acal$-quasiconvex operator (hessian-gradient), along sequences
with prescribed boundary conditions.

\begin{theorem}\label{Thm4.3}
Let $\omega\subset \RR^{N-1}$ be a bounded open set with uniformly $C^2$
boundary, and let $\Psi\colon \allowbreak\big(\RR_s^{(N-1)\times (N-1)}\big)^d\times
\RR^{d\times (N-1)} \to{{[0,+\infty)}}$ be a cross $2$-quasiconvex-quasiconvex
function with linear growth at $\infty$. Assume that    $\vartheta^b_1 \in
W^{2,1}(\omega';\mathbb R^d)$ and $\vartheta^b_2 \in W^{1,1}(\omega';\mathbb
R^d)$, where    $\omega' \supset
\supset \omega$ is a bounded open set. Let  $ u_n \in \vartheta^b_1 + W^{2,1}_0(\omega;\mathbb
R^d)$ and $v_n \in \vartheta^b_2+ W^{1,1}_0(\omega;\mathbb R^d)$ be such
that $u_n \overset{*}{\rightharpoonup} u$ in $BH(\omega;\mathbb R^d) $ and
$v_n \overset{*}{\rightharpoonup} v$ in $BV(\omega;\mathbb R^d)$. Then,
\begin{equation}
\label{lscARDPR}
\begin{aligned}
&\int_\omega \Psi(\nabla^2 u, \nabla v)\, dx + \int_\omega \Psi^\infty\left(\frac{d
((D^2)^s u, D^s v)}{d |((D^2)^s u, D^s v)|}\right)d |((D^2)^s u, D^s v)|\\
&\quad + \int_{\partial \omega} \Psi^\infty(-((\nabla u)^- - \nabla \vartheta^b_1,
v^--\vartheta^b_2)\otimes
\nu_{\partial \omega})\,d \mathcal H^{N-2} \leq \liminf_{n\to +\infty} \int_\omega
\Psi(\nabla^2 u_n, \nabla v_n )\,dx.
\end{aligned}
\end{equation}
\end{theorem}

\begin{remark}
    \label{HessianJump}
    As proven in \cite[(2.5) in Theorem 2.1]{D} and in \cite[(2.2)]{FLP},
we have that 
    \begin{align*}
    [(\nabla u)^+ - \nabla \vartheta^b_1]\otimes \nu_{\partial \omega} =
((\Tr  (\nabla u- \nabla \vartheta^b_1)) \cdot \nu_{\partial \omega})\nu_{\partial
\omega}\otimes \nu_{\partial \omega}.
    \end{align*}
\end{remark}

\begin{proof}
[Proof of Theorem \ref{Thm4.3}]
We start by observing  that the assumptions of the theorem yield ${\Tr}(u_n)=\vartheta^b_1$
 on $\partial \omega$, in the sense of Theorem \ref{Dthm1}, and $\Tr (v_n)=\vartheta^b_2$
on $\partial \omega$, in the sense of Theorem~\ref{TrBV}. Moreover, $\Tr  (u_n)\to \Tr  (u)=\vartheta^b_1$
and \(\Tr  (v_n)\to \Tr  (v)=\vartheta^b_2\) strongly in $L^1(\partial \omega;\mathbb
R^d)$.

Let $\omega'$ be as in the statement, and set 
\begin{equation*}
\begin{aligned}
& u':= \begin{cases}
u &\hbox{ in } \omega,\\
\vartheta^b_1 & \hbox{ in } \omega'\setminus \overline \omega,
\end{cases}
\qquad 
u'_n:= \begin{cases}
u_n &\hbox{ in } \omega,\\
\vartheta^b_1 & \hbox{ in } \omega'\setminus \overline \omega,
\end{cases}\\
& v':= \begin{cases}
v &\hbox{ in } \omega,\\
\vartheta^b_2 & \hbox{ in } \omega'\setminus \overline \omega,
\end{cases}
\qquad 
v'_n:= \begin{cases}
v_n &\hbox{ in } \omega,\\
\vartheta^b_2 & \hbox{ in } \omega'\setminus \overline \omega,
\end{cases}\\
\end{aligned}
\end{equation*}
for \(n\in\N\). 
Then, it can be checked  that $u_n' \in W^{2,1}(\omega';\mathbb R^d)$ and
 $v_n' \in W^{1,1}(\omega' ;\mathbb R^{d \times N})$, with  $u_n'\overset{*}{\rightharpoonup}
u$ in $BH(\omega';\mathbb R^d)$ and $v_n' \overset{*}{\rightharpoonup}v$
 in $BV(\omega' ;\mathbb R^{d \times N})$.

Because  $\Psi$ is a  $\widetilde{\mathcal A}$-quasiconvex function
with
linear growth,   it also is Lipschitz continuous; this follows
from the fact that \(\Psi\) is  separately
convex in each direction of the so-called characteristic cone associated
with \(\widetilde \Acal\) by arguments entirely similar to those in \cite[Proposition
3.6]{SZ}.
Thus, \(\Psi\) satisfies the assumptions of Theorem~\ref{thmARDPR}, so that  it can be  applied to $\calO=\omega'$,  yielding 
\begin{align*}
&\int_{\omega'} \Psi(\nabla^2 u', \nabla v')\, dx + \int_{\omega'} \Psi^\infty\left(\frac{d
((D^2)^s u', D^s v')}{d |((D^2)^s u', D^s v')|}\right)d |((D^2)^s u', D^s
v')|
\leq 
\liminf_{n\to +\infty} \int_{\omega'} \Psi(\nabla^2 u'_n, \nabla v'_n )\,dx,
 \end{align*}
from which we get that
\begin{align*}
&\int_\omega \Psi(\nabla^2 u, \nabla v)\, dx + \int_\omega \Psi^\infty\left(\frac{d
((D^2)^s u, D^s v)}{d |((D^2)^s u, D^s v)|}\right)d |((D^2)^s u, D^s v)|
\\
&\quad +\int_{\partial \omega} \Psi^\infty(-((\nabla u)^- - \nabla \vartheta^b_1,
v^- -\vartheta^b_2)\otimes \nu_{\partial \omega})\,d \mathcal H^{N-2} +\int_{\omega'
\setminus \omega} \Psi(\nabla^2 \vartheta^b_1, \nabla \vartheta^b_2)\,dx
 \nonumber \\
&\qquad\leq \liminf_{n\to +\infty} \int_\omega \Psi(\nabla^2 u_n, \nabla
v_n )\,dx + \int_{\omega' \setminus \omega} \Psi(\nabla^2 \vartheta^b_1,
\nabla \vartheta^b_2)\,dx. 
\end{align*}
Hence, \eqref{lscARDPR} holds.
\end{proof}

\begin{theorem}\label{Thm4.3conv}
Let $\Psi\colon\mathbb R^{d \times (N-1)} \times \mathbb R^d\to{{[0,+\infty)}}$
be a convex function with linear growth at $\infty$.  Fix $\delta >0 $ and assume that $\vartheta^a_1 \in W^{2,1}(]0, 1+\delta[;\mathbb
R^d)$ and $\vartheta^a_2 \in W^{1,1}(]0,1+\delta[;\mathbb R^{d\times (N-1)})$.
Let  $ u_n \in \vartheta^a_1 + {{W^{2,1}_{\Gamma^a}}}(]0,1[;\mathbb R^d)$
and $v_n \in \vartheta^a_2 + {{W^{1,1}_{\Gamma^a}}}(]0,1[;\mathbb R^{d\times
(N-1)})$ be such that $u_n \overset{*}{\rightharpoonup} u$
in $BH(]0,1[;\mathbb R^d) $ and $v_n \overset{*}{\rightharpoonup} v$ in $BV(]0,1[;\mathbb
R^{d\times (N-1)})$. Then,
\begin{equation}
\label{eq:lscARDPRconv}
\begin{aligned}
&\int_0^1 \Psi( \nabla v,\nabla^2 u)\, dx_N + \int_0^1 \Psi^\infty\left(\frac{d
(D^s v,(D^2)^s
u)}{d |(D^s v,(D^2)^s u)|}\right)d |(D^s v,(D^2)^s u)| \\
&\quad + \Psi^\infty\big( -v^-(1)+\vartheta^a_2(1),-(\tfrac{d u}{d x_N})^-(1)+
\tfrac{d\vartheta^a_1}{d
x_N}(1)\big)
   \leq \liminf_{n\to +\infty} \int_0^1 \Psi ( \nabla v_n
,\nabla^2 u_n)\,dx_N.
\end{aligned}
\end{equation}
\end{theorem}

\begin{proof}[Proof]
The proof is identical to the one of Theorem \ref{Thm4.3}, provided we extend
the functions $u_n'$ and $v_n'$  in $]0,1+\delta[$  as follows:
\begin{align*}
& u'(x_N):= \begin{cases}
u(x_N) &\hbox{ if \( x_N \in\ ]0,1[\)} ,\\
\vartheta^a_1(x_N) &\hbox{ if \( x_N \in\ ]0,1+\delta[\)} ,
\end{cases} \qquad 
u'_n(x_N):=\begin{cases}
u_n(x_N) &\hbox{ if \( x_N \in\ ]0,1[\)} ,\\
\vartheta^a_1(x_N) &\hbox{ if \( x_N \in\ ]0,1+\delta[\)} ,
\end{cases} \\
& v'(x_N):= \begin{cases}
v(x_N) &\hbox{ if \( x_N \in\ ]0,1[\)} ,\\
\vartheta^a_2(x_N) &\hbox{ if \( x_N \in\ ]0,1+\delta[\)} ,
\end{cases} \qquad 
v'_n(x_N):=\begin{cases}
v_n(x_N) &\hbox{ if \( x_N \in\ ]0,1[\)} ,\\
\vartheta^a_2(x_N) &\hbox{ if \( x_N \in\ ]0,1+\delta[\)}.
\end{cases}
\end{align*}
Arguing as in Theorem \ref{Thm4.3}, we obtain \eqref{eq:lscARDPRconv}.
\end{proof}

\section{Proof of the main results}\label{MR}

In this section, we prove our main results, Theorem~\ref{thm:mainFMZ}
and Corollary~\ref{cor:min}.
 We
start by characterizing the asymptotic behavior of sequences that are uniformly
bounded in energy, which is encoded in the space \(\Vcal\) given by
\begin{equation}
\label{eq:spaceV}
\begin{aligned}
\Vcal:=
        \mathcal{U} \times \Xi \times Z,
\end{aligned}
\end{equation}
where \(\Ucal\) and \(\Xi\) are the spaces in \eqref{eq:UcalXi} and
\begin{equation*}
\begin{aligned}
&    Z=\mathcal M(]0,1[;BH_{av}\big(\omega;\big(\mathbb{R}^{(N-1) \times (N-1)}\big)^d\big) \times  {\mathcal
M}(\omega;BH_{av}(]-1,0[;\RR^d)),                           
\end{aligned}
\end{equation*}
       with
\begin{equation*}
\begin{aligned}
&  BH_{av}(A;\RR^\ell)=\bigg\{v \in
        BH(A;\RR^\ell): \,\int_{A}v \,d x=0,  \int_{A}\nabla
v
        \,d x=0\bigg\}
\end{aligned}
\end{equation*}
for any subset $A\subset\mathbb{R}^k $. As we next remark, there
is no loss of generality in assuming null bulk
loads, \(H_n^a=0\) and \(H_n^b=0\) in \eqref{kna} and \eqref{knb}, respectively. 
 Analogously, as  mentioned  in the Introduction, we neglect the presence of the hyperelastic energy since, under  the continuity and \eqref{Vgrowth}  assumptions  on the function $V$,  it represents a continuous term under $\Gamma$-convergence.

\begin{remark}\label{rmk:Hnzero}
 For the sake of exposition,  we assume null bulk
loads in the proof of Proposition~\ref{Compactnessres} below. The general
case under assumption \eqref{loadconv}  can be tackled similarly by exploiting the
Poincar\'e inequality and suitable uniform \(L^\infty\) bounds of
\(H^a_n\) and \(H^b_n\). We refer the reader     to the arguments in \cite[Lemma 2.2]{GGLM1} and \cite[Proposition
5.1]{GZ} for a more  detailed justification. 

We further observe that, under condition \eqref{loadconv}, the load terms in
\eqref{kna} and \eqref{knb} are
continuous with respect to the $\Gamma$-convergence  in
Theorem~\ref{thm:mainFMZ}. Precisely, 
   \[
    \lim_{n\to\infty}\int_{\O^a} H_n^a\cdot u_n^a\, dx = \int_{\O^a}
H \cdot u^a \,dx \quad \text{and} \quad \lim_{n\to\infty}\int_{\O^b}
H_n^{b} \cdot
u_n^b \,dx = \int_{\O^b} H \cdot u^b \,dx
    \]
    whenever $u_n^a\to u^a$ in $L^1(\O^a;\RR^d)$ and
$u_n^b\to u^b$ in $L^1(\O^b;\RR^d)$.
 Consequently,  we can neglect these load terms in our subsequent
analysis, without loss of generality.

 Similarly, using \eqref{Vgrowth} and Proposition~\ref{Compactnessres} below, we can also neglect the hyperelastic energy terms because
\begin{equation*}
        \begin{aligned}
                &\lim_{n\to\infty}\int_{\O^a} V\left(\frac{1}{r_n} D_{x'}u^a_n, D_{x_N}u^a_n\right)   \, dx = \int_{\O^a} V(\xi^a, D_{x_N} u^a ) \,dx ,\\ 
                &   \lim_{n\to\infty}\int_{\O^b} V\left(D_{x'}u^b_n, \frac{1}{h_n}D_{x_N}u^b_n\right)   \, dx = \int_{\O^b} V(D_{x'}u^b, \xi^b) \,dx,    \end{aligned}
\end{equation*}
whenever $({\frac{1}{r_{n}}}D_{x'} u^a_{n}, D_{x_N}u^a_n )
\to  ({\xi}^a,D_{x_N}u^a ) $ in $L^1(\O^a; \mathbb R^{
        d\times N})$ and   $(D_{x'}u^b_n, \frac{1}{h_n}D_{x_N}u^b_n) \to (D_{x'}u^b, \xi^b)$ in $L^1(\Omega^b;\mathbb R^{
        d\times N})$. 
\end{remark}

\begin{proposition}\label{Compactnessres}
       Let    \(W :  (\mathbb{R}^{{N\times 
N}}_s)^d\to 
\mathbb{R}\)  be a function satisfying conditions  (\ref{contvsBorel})
        and (\ref{coerci}),  let $V \colon \mathbb R^{d \times N}\to \mathbb R$ be a continuous function satisfying \eqref{Vgrowth},  and consider the corresponding sequence of  functionals
\((F_n)_{n\in\N}\) introduced in  \eqref{energy}--\eqref{knb}, where the parameters \(r_n\)
and \(h_n\) satisfy conditions     
         (\ref{hrzero}) and (\ref{rate}), and where the loads
\(H_n\) satisfy 
\eqref{eq:loads}--\eqref{loadconv}. Consider further  the spaces
\(\mathcal{U}_n\) and \(\Vcal\) introduced in \eqref{Vn} and \eqref{eq:spaceV},
respectively. Then, for every  
        $( u_n^a, u_n^b) \in \mathcal{U}_n$  such that \(\sup_{n\in\N} |F_n[(
 u_n^a,  u_n^b)]|
<+\infty\), there exist an increasing sequence of positive integer
        numbers $\{n_i\}_{i \in \mathbb N}$ and
        $(( {u}^a, {u}^b),
        ( {\xi}^a, {\xi}^b),
        ( {z}^a, {z}^b))\in \Vcal$, depending possibly on
        the selected subsequence $\{n_i\}_{i \in \mathbb N}$, such that
\begin{align}
&\begin{cases}
u^a_{n_i}\overset{\ast}{\rightharpoonup}  {u}^a  
\hbox{ in }BH(\O^a;\RR^d),\\[1.5mm]
u^b_{n_i}\overset{\ast}{\rightharpoonup }{u}^b \hbox{ in }BH(\O^b;\RR^d),
\end{cases}\label{weakdefinitive}\\[2.5mm]
&\begin{cases}
\displaystyle{\frac{1}{r_{n_i}}}\nabla_{x'}  u^a_{n_i}
\overset{\ast}{\rightharpoonup}  {\xi}^a \hbox{ in } BV(\O^a; \mathbb R^{
d\times (N-1)}),\\[2.5mm]
\displaystyle{\frac{1}{h_{n_i}}}\nabla_{x_N}  u^b_{n_i}
\overset{\ast}{\rightharpoonup}  {\xi}^b \hbox{ 
in  }BV(\O^b;\mathbb R^d),
\end{cases}\label{weakgradiente}\\[2.5mm]
&\begin{cases}
\displaystyle{\frac{1}{r^2_{n_i}}}D^2_{x'}
 u^a_{n_i} \overset{\ast}{\rightharpoonup}  {z}^a\hbox{ in
 the sense of measures in } \mathcal M\big(\O^a;\big(\mathbb{R}^{(N-1)
\times (N-1)}\big)^d\big),\\[2.5mm]
\displaystyle{\frac{1}{h^2_{n_i}}}D^2_{x_N}  u^b_{n_i}
 \overset{\ast}\rightharpoonup  {z}^b \hbox{ in the
sense of measures in }\mathcal M(\O^b;\RR^d),
\end{cases}\label{weakhessiano}
\end{align}        
as $ i \to + \infty$.
\end{proposition}
\begin{proof}[Proof] As  noted in Remark~\ref{rmk:Hnzero}, we
may assume that \(H_n^a=0\) and \(H_n^b=0\) without loss of
generality.  Moreover, as we show next, \eqref{weakdefinitive}--\eqref{weakhessiano} hold even when $V=0$, which we assume in the sequel.  We also observe that we have adopted the symbols $D^2_{x'}{
 z^a}$ and $D^2_{x_N}{  z^b}$  in \eqref{weakhessiano} to underline the coincidence
with the distributional convergence.

The coercivity assumption on $W$ yields        \begin{equation}\label{boundedenergy}
                \begin{array}{ll}
                        \displaystyle{\sup_{n\in\NN}}\Bigg\{&\displaystyle{\int_{\O^a}\left(\frac{1}{{r^{2}
                                                _n}}|D^2_{x'}  u^a_n|+
                                \frac{1}{r_n}|D^2_{x',x_N}  u^a_n|
                                +|D^2_{x_N}  u^a_n|\right)dx  }\smallskip\\
                        & \quad + \displaystyle{\frac{h_n}{r^{N-1}_n}\int_{\O^b}\left(|D^2_{x'}

                                u^b_n|+ \frac{1}{{h_n}}|D^2_{x',x_N}  u^b_n|
                                +\frac{1}{h_n^{2}}|D^2_{x_N}  u^b_n|\right)dx\quad
                                \Bigg\}< \infty.}
                \end{array}
        \end{equation}
        Using this estimate, conditions    (\ref{hrzero}) and  \eqref{rate},
and the boundary conditions in the definition of \(\mathcal{U}_n\)  combined
with Poincar\'e's inequality (see \cite[Theorem 5.10]{EG} and \cite{Ziemer:1989aa}), applied twice,   we conclude
that
\(\{u_n^a\}_{n\in\NN}\) is uniformly bounded in \(W^{2,1}(\Omega^a;\RR^d)\)
and 
\(\{u_n^b\}_{n\in\NN}\) is uniformly bounded in \(W^{2,1}(\Omega^b;\RR^d)\),
with \(\Vert D^2_{x'}
 u^a_n\Vert_{L^1} \leq C r_n^2 \), \(\Vert D^2_{x',x_N}  u^a_n\Vert_{L^1}
\leq C r_n \), \(\Vert D^2_{x',x_N}
 u^b_n\Vert_{L^1} \leq \tilde C h_n \), and \(\Vert D^2_{x_N,x_N}
 u^b_n\Vert_{L^1} \leq \tilde C h_n^2 \) for some positive constants
 \(C\) and \(\tilde C\) that independent of \(n\).
 Thus, using the compactness properties in \(BH\) of bounded sequences in
\(W^{2,1}\) (see \cite{D,D1}), the  compact embedding of \(W^{2,1}\)
in \(W^{1,1^*}\)
(see \cite{D, D1}), and the continuity of the trace operator in
Sobolev
spaces see \cite{CDA02}, we can find an increasing sequence of
        positive integer numbers $\{n_i\}_{i \in \mathbb N}$ and 
        functions $ {u}^a \in BH(\O^a;\RR^d)$ and  $ {u}^b \in BH(\O^b;\RR^d)$
    for which (\ref{weakdefinitive}) holds and, in addition,  $ u^a_{n_i}\to
 u^a$ in $W^{1,1^*}(\Omega^a;\mathbb R^d)$, $ u^b_{n_i}\to  u^b$ in $W^{1,1^*}(\Omega^b;\mathbb
R^d)$, and 
\begin{equation*}
\begin{aligned}
&u^a = \theta^a_1 \text{ on } \o\times \{1\}, \enspace D^2_{x'} {u}^a =0
\text{
a.e.~in } \Omega^a, \enspace  D^2_{x',x_N} {u}^a =0 \text{
a.e.~in } \Omega^a,\\
& u^b = \theta^b_1 \text{ on }
\partial\o\times(-1,0), \enspace   D^2_{x',x_N}
 u^b = 0 \text{
a.e.~in } \Omega^b,  \enspace   D^2_{x_N}
 u^b = 0 \text{
a.e.~in } \Omega^b.
\end{aligned}
\end{equation*}
In particular,  $ {u}^a$ is independent of
        $x'$ and  $ {u}^b$ is independent of
        $x_N$, which allows us to identify \((u^a, u^b)\) with an element
in
\(\mathcal{U}\) if \(N\geq3\).
On the other hand, if $N=2$,  we further have that $BH(\Omega^a;\mathbb R^d)\subseteq W^{1,2}(\Omega^a;\mathbb R^d)\cap C(\overline{\Omega^a};\mathbb R^d)$ and  $BH(\Omega^b;\mathbb R^d)\subseteq W^{1,2}(\Omega^b;\mathbb R^d)\cap C(\overline{\Omega^b};\mathbb R^d)$, with continuous embeddings, by Proposition~\ref{embedding}. Then, we can use \cite[Proposition 2.1 and Corollary 2.3]{GGLM1} to conclude  that $u^a(0)=u^b(0)$. We further observe that the functions $u^a$ and $u^b$ are Lipschitz continuous because they are defined on intervals (see \cite{D}).

To prove 
(\ref{weakgradiente}), set \(\xi^a_n:=\frac{1}{r_{n}}\nabla_{x'}
 u^a_{n} \) and \(\xi^b_n:=\frac{1}{h_{n}}\nabla_{x_N}
 u^a_{n} \). Then, by definition of \(\mathcal{U}_n\), \(\xi^a_n
 = \theta^a_2\) on \(\omega\times \{1\}\) and \(\xi^b_n = \theta^b_2\) on
\(\partial
 \omega \times\, ]-1,0[\). Thus, arguing as above, from 
        (\ref{boundedenergy}) and   Poincar\'e's inequality,
 we may find  \({\xi}^a
\in BV(\O^a; \mathbb R^{d \times (N-1)})\), with 
        $D_{x'} {\xi}^a =0$ a.e.~in $\O^a$, and \({\xi}^b
\in BV(\O^b;\mathbb R^d)\), with 
        $D_{x_N} {\xi}^b =0$ a.e.~in $\O^b$, for which  
(\ref{weakgradiente}) holds (extracting a further  subsequence
if necessary). Moreover, we can identify \((\xi^a,
\xi^b)\) with an element  of \(\Xi\).

        Finally, we establish  (\ref{weakhessiano}).
        For every $n \in \mathbb N$, set
        $$
        m^{(1)}_n(x_N):=  \fint_\omega
        \nabla_{x'}  u^a_n(x',x_N)\,dx',\quad  \hbox{for  a.e.~}x_N\in\,]0,1[,
        $$
        and
        $$
        m^{(2)}_n(x_N)= \fint_\omega \big( 
        u^a_n(x',x_N)- m^{(1)}_n(x_N) x'\big)\,dx',\quad \hbox{for  a.e.~}x_N\in\,]0,1[.
        $$
        Applying (twice) the Poincar\'e--Wirtinger inequality,
        there exists $c\in\,]0,+\infty[$ (depending only on $\omega$ but
not
        on $x_N$) such that
        $$
        \begin{array}{ll}
                \left\Vert\displaystyle{\frac{1} {r^2_n}}\left( 
                u^a_n(\cdot,x_N)-m^{(1)}_n(x_N) x' - m^{(2)}_n(x_N)
                \right)\right\Vert_{W^{2,1}(\omega;\RR^d)}\leq
                \displaystyle{\frac{c} {r^2_n} }\left\Vert \nabla^2_{x'}
                u^a_n(\cdot,
                x_N) \right\Vert_{L^{1}(\omega;(\mathbb{R}^{(N-1) \times

(N-1)})^d) },
        \end{array}
                        $$
for all \(n\in\N\) and for a.e.~\(x_N\in\,]0,1[\).                      
 
        Thus, by integrating this inequality over $x_N\in\, ]0,1[$,
        and extracting a further subsequence if necessary, which
        we do not relabel, estimate
        (\ref{boundedenergy}) allows us to find a  measure \(z^a
\in \mathcal M(]0,1[;BH_{av}(\omega;\RR^d)) \)  such
        that
\begin{equation}
\label{eq:44a}
\begin{aligned}
\frac{1}{r^2_{n_i}}\left(  u^a_{n_i} - m^{(1)}_{n_i}
                x' - m^{(2)}_{n_i} \right)\rightharpoonup  {z}^a
\hbox{
                        weakly* in } \mathcal M(]0,1[;BH_{av}(\o;\RR^d)).
\end{aligned}
\end{equation}

Analogously, setting for a.e.~\(x'\in\omega\),
\begin{equation*}
\begin{aligned}
& {\mu}^{(1)}_n(x'):=\int_{-1}^0 D_{x_N}  u^b_n(x',x_N)\,dx_N,\\
& {\mu}^{(2)}_n(x'):=\int_{-1}^0 \big(  u^b_n(x',x_N)-
        \mu^{(1)}_n(x')
        x_N\big)\,dx_N,
\end{aligned}
\end{equation*}
we can find \(z^b\in M(\o; BH_{av}(]-1,0[;\RR^d) \)
such that
\begin{equation}
\label{eq:44b}
\begin{aligned}
\frac{1}{h_{n_i}}\left(  u^b_{n_i}- {\mu}^{(1)}_{n_i}
        x_N - {\mu}^{(2)}_{n_i}\right) \rightharpoonup  {z}^b
        \hbox{ weakly* in } \mathcal M(\o; BH_{av}(]-1,0[;\RR^d).
\end{aligned}
\end{equation}

 To conclude, we observe that \eqref{weakhessiano}
follows from \eqref{eq:44a} and \eqref{eq:44b}. \end{proof}

As
it will become clear, Theorem~\ref{thm:mainFMZ}  will be a consequence of
the following theorem.

\begin{theorem}\label{thm:FMZ}  Let    \(W :  (\mathbb{R}^{{N\times 
N}}_s)^d\to 
\mathbb{R}\)  be a function satisfying conditions  (\ref{contvsBorel})--(\ref{coerci}),  let $V \colon \mathbb R^{d \times N}\to \mathbb R$ be a continuous function satisfying \eqref{Vgrowth},  and consider the corresponding sequence of  functionals
\((F_n)_{n\in\N}\) introduced by  \eqref{energy}--\eqref{knb}, where the parameters \(r_n\)
and \(h_n\) satisfy conditions     
         (\ref{hrzero}) and (\ref{rate}), and where the loads
\(H_n\) satisfy 
\eqref{eq:loads}--\eqref{loadconv}. Consider further the space \(\mathcal
U \times \Xi\) and  the functionals  \(K^a\) and \(K^b\)  introduced
in \eqref{eq:UcalXi}--\eqref{Kb}, and define, for  every $(( u^a,  u^b),
 (\xi^a,  \xi^b))\in \mathcal U \times \Xi$,
\begin{equation}
\label{Gammalimitdef}
\begin{aligned}
F[(u^a, u^b), (\xi^a, \xi^b)]:= \inf \bigg\{ \liminf_{n\to +\infty} F_n[u_n]
\colon& u_n=(u^a_n, u^b_n) \in{\mathcal
U}_n,\,
u^a_n\overset{\ast}{\rightharpoonup} u^a \hbox{ in }BH(\Omega^a;\RR^d),\\
&u_n^b  \overset{\ast}{\rightharpoonup} u^b \hbox{ in }BH(\Omega^b;\mathbb
R^d),\\ 
& \frac{1}{r_n}\nabla_{x'}u^a_n \overset{\ast}{\rightharpoonup} \xi^a\hbox{
in }BV\big(\Omega^a; \mathbb
R^{d \times (N-1)}\big),\\
&\frac{1}{h_n}\nabla_{x_N}u^b_n
\overset{\ast}{\rightharpoonup} \xi^b\hbox{ in } BV(\Omega^b;\RR^d)\bigg\},
\end{aligned}
\end{equation}
where \({\mathcal U}_n\)  is given by \eqref{Vn}.
Then, for  every $(( u^a,  u^b),  (\xi^a,  \xi^b))\in \mathcal U \times \Xi$,
we have that  %
\begin{equation*}
\begin{aligned}
F[(u^a, u^b), (\xi^a, \xi^b)]=K^a[({u}^a, \xi^a)]+ q
K^b[({u}^b,{\xi}^b)].
\end{aligned}
\end{equation*}
 \end{theorem}

We will prove Theorem~\ref{thm:FMZ} in the Subsections~\ref{sect:lb} and
\ref{sect:ub} below, while Theorem~\ref{thm:mainFMZ} and
Corollary~\ref{cor:min} are addressed
in Subsection~\ref{sect:proofmain}.

\subsection{Lower bound}\label{sect:lb}

In this subsection, we prove the lower bound for \(F\) in Theorem~\ref{thm:FMZ}.

\begin{proposition}\label{lbprop}
 Under the notations and assumptions of Theorem~\ref{thm:FMZ}, we have that if $(( u^a,
 u^b), \allowbreak (  \xi^a,  \xi^b))\in \mathcal U \times \Xi$ and    $(u_n^a, u_n^b)
\in \mathcal U_n$ are such that $u_n^a \overset{\ast}{\rightharpoonup}
u$  in $BH(\Omega^a; \mathbb R^d)$, 
$u_n^b  \overset{\ast}{\rightharpoonup} u^b$ in $BH(\Omega^b;\mathbb
R^d)$,   $\frac{1}{r_n}\nabla_{x'}u^a_n \overset{\ast}{\rightharpoonup} \xi^a$
in $BV(\Omega^a;\mathbb R^{d \times (N-1)})$, and $ \frac{1}{h_N} \nabla_{x_N}u^b_n
\overset{\ast}{\rightharpoonup} \xi^b$ in $ BV(\Omega^b;\mathbb R^d)$, then
\begin{equation}
\label{liminfineq1}
\begin{aligned}
\liminf_{n\to\infty} F_n[u_n]=\liminf_n\left(K^a_n[  u^a_n]+
\frac{h_n}{r_n^{N-1}}K^b_n[  u^b_n]\right) \geq  K^a[({u}^a, \xi^a)]+ q
K^b[({u}^b,{\xi}^b)].
\end{aligned}
\end{equation}
\end{proposition}

\begin{proof}[Proof]
The proof follows along the lines of \cite[Proposition 4.2]{GZNODEA} and
\cite[Theorem 11]{BFLM}, exploiting Fubini's theorem and the  extension results
addressed  in Theorem~\ref{Thm4.3}   and Theorem~\ref{Thm4.3conv}.

We start by observing that the only  nontrivial case is when  the limit inferior
on left-hand side of \eqref{liminfineq1} is finite, which we assume hereafter.

Then,  exploiting condition \eqref{rate}, the superadditivity of the liminf,
the definitions of \(W_0\) and \(\hat W\) together with the  inequalities
$\hat W\geq\hat W^{**}$ and $W_0 \geq Q_{\widetilde{\mathcal A}}W_0$
(cf.~Subsection~\ref{sect:Aqcx}), and Fubini's theorem combined with Fatou's
lemma, yields  %
\begin{equation*}
\begin{aligned}
+\infty&>\liminf_{n\to +\infty}\left[\int_{\Omega^a} W\left(\left(\begin{array}{cc}\displaystyle{\frac{1}{r_n^2}D^2_{x'}u^{a}_n}
&\displaystyle{\left(\frac{1}{r_n}D^2_{x',x_N}u^{a}_n\right)^T}\smallskip\\
\displaystyle{\frac{1}{r_n}D^2_{x',x_N}u^{a}_n} &D^2_{x_N}u^{a}_n
\end{array}\right)\right)dx\right.
\\
&\qquad+\left.\frac{h_n}{r_n^{N-1}}\int_{{\O}^b}
W\left(\left(\begin{array}{cc}D^2_{x'}u^b_n
&\left(\displaystyle{\frac{1}{h_n}D^2_{x',x_N}u^b_n}\right)^T\smallskip\\
\displaystyle{\frac{1}{h_n}D^2_{x',x_N}u^b_n}
&\displaystyle{\frac{1}{h^2_n}D^2_{x_N}u^b_n
}\end{array}\right)\right)dx\right]\\ 
&\geq \liminf_{n\to +\infty}\left[\int_{\Omega^a} {\hat W}\left(
\displaystyle{\frac{1}{r_n}D^2_{x',x_N}u^{a}_n,} D^2_{x_N}u^{a}_n
\right)dx\right.
\\
&\qquad+\left.\frac{h_n}{r_n^{N-1}}\int_{{\O}^b}
W_0\left(\left(D^2_{x'}u^b_n,
\displaystyle{\frac{1}{h_n}D^2_{x',x_N}u^b_n}
\right)\right)dx\right]\\
&\geq\liminf_{n\to +\infty}\int_{\Omega^a} {\hat W}^{\ast \ast}\left(
\displaystyle{\frac{1}{r_n}D^2_{x',x_N}u^{a}_n}, D^2_{x_N}u^{a}_n
\right)dx
\\
&\qquad+q \liminf_{n \to +\infty}\int_{{\O}^b}
Q_{\widetilde{\mathcal A}}W_0\left(D^2_{x'}u^b_n,
\displaystyle{\frac{1}{h_n}D^2_{x',x_N}u^b_n}
\right)dx, \\
&\geq \int_\omega\bigg[\liminf_{n\to +\infty} \int_0^1 {\hat W}^{\ast \ast}\left(
\displaystyle{\frac{1}{r_n}D^2_{x',x_N}u^{a}_n}, D^2_{x_N}u^{a}_n
\right)dx_N\Bigg] dx'
\\
&\qquad+q \int_{-1}^0 \bigg[\liminf_{n \to +\infty} \int_\omega
Q_{\widetilde{\mathcal A}}W_0\left(D^2_{x'}u^b_n,
        \displaystyle{\frac{1}{h_n}D^2_{x',x_N}u^b_n}
\right)dx'\bigg]dx_N.
\end{aligned}
\end{equation*}

Next, we use Proposition~\ref{lbprop} to conclude that
\begin{equation}\label{eq:convlb}
\begin{aligned}
& u^a_n (x',\cdot) \weaklystar u^a\enspace \hbox{ in }BH(]0,1[;\mathbb R^d)
\hbox{ for a.e. } x' \in \omega,\\
&\frac{1}{r_n} \nabla_{x'} u^a_n(x',\cdot)  \weaklystar  \xi^a\enspace \hbox{
in } BV(]0,1[;\mathbb R^{d\times (N-1)}) \hbox{ for a.e. } x' \in \omega,\\
&u^b_n(\cdot,x_N)    \weaklystar u^b\enspace \hbox{ in } BH(\omega;\mathbb
R^d)
\hbox{ for a.e. } x_N \in ]-1,0[,\\
&\frac{1}{h_n} \nabla_{x_N} u^b_n(\cdot,x_N)  \weaklystar \xi^b\enspace \hbox{
in } BH(\omega;\mathbb R^d) \hbox{
for a.e. } x_N \in ]-1,0[.
\end{aligned}
\end{equation}
In fact, the bounds established in Proposition~\ref{lbprop}, the    compact
embedding of \(W^{2,1}\) in \(W^{1,1}\), and  Fubini's theorem combined with
Fatou's
lemma, imply that
\begin{equation*}
\begin{aligned}
u^a_n (x',\cdot) \to u^a \hbox{ in }W^{1,1}(]0,1[;\mathbb R^d) \quad \text{
and } \quad \liminf_{n\to\infty} \int_{0}^{1}|\nabla^2 u^a_n (x',x_N)| dx_N<+\infty
\end{aligned}
\end{equation*}
 for a.e.~\(x' \in \omega\). Consequently, the uniqueness of weak-\(\ast\)
limits yields the first convergence in \eqref{eq:convlb}. The remaining convergences
in \eqref{eq:convlb} can be justified similarly.

Hence, applying Theorem~\ref{Thm4.3conv} in $]0,1[$ to $u^a_n(x',\cdot)$
and $\frac{1}{r_n} \nabla_{x'} u^a_n(x',\cdot)$ for a.e. $x' \in \omega$,
and  Theorem~\ref{Thm4.3}  in $\omega$  to $u^b_n(\cdot, x_N)$ and $\frac{1}{h_N}
\nabla_{x_N}u^b_n(\cdot, x_N)$ for a.e. $x_N \in\, ]-1,0[$, we get 
\begin{equation*}
\begin{aligned}
&\liminf_{n\to +\infty} \int_0^1 {\hat W}^{\ast \ast}\left(
\displaystyle{\frac{1}{r_n}D^2_{x',x_N}u^{a}_n}, D^2_{x_N}u^{a}_n
\right)dx_N\\
&\qquad \geq  \int_0^1 {\hat W}^{\ast \ast} 
\big(\nabla_{x_N}\xi^a(x_N),D^2_{x_N}u^{a}(x_N)\big) dx_N \\
&\qquad\quad + \int_0^1 (\hat W^{**})^\infty\left(\frac{d( (D_{x_N})^s \xi^a,
(D^2_{x_N})^s
u^a )}{d |( (D_{x_N})^s \xi^a,(D^2_{x_N})^s u^a)|}(x_N)\right)d |( (D_{x_N})^s
\xi^a,(D^2_{x_N})^s u^a)|(x_N) \\
&\qquad\quad+  ({\hat W}^{**})^\infty \left(  - (\xi^a)^-(1)+\theta^a_2,
-\left(\frac{d u^a}{dx_N}   \right)^-(1)\right) 
\end{aligned}
\end{equation*}
and 
\begin{equation*}
\begin{aligned}
&\liminf_{n \to +\infty} \int_\omega
Q_{\widetilde{\mathcal A}}W_0\left(D^2_{x'}u^b_n,
        \displaystyle{\frac{1}{h_n}D^2_{x',x_N}u^b_n}
\right)dx'\\
&\qquad  \geq \int_\omega
Q_{\widetilde{\mathcal A}}W_0 \big( D^2_{x'}u^b(x'),\nabla_{x'}\xi^b(x')
\big)dx'\\
 &\qquad\quad+
 \int_\omega (Q_{\widetilde{\mathcal A}}W_0)^\infty \left(\frac{d( (D^2_{x'})^s
u^b , (D_{x'})^s \xi^b)}{d |((D^2_{x'})^s u^b , (D_{x'})^s \xi^b)|}(x')\right)d
|((D^2_{x'})^s u^b , (D_{x'})^s \xi^b)|(x')\\
&\qquad\quad + \int_{\partial \omega}  (Q_{\widetilde{\mathcal A}}W_0)^\infty
(-((\nabla_{x'} u^b)^-(x')- \nabla_{x'} \theta^b_1(x'), (\xi^b)^-(x')- \theta^b_2(x'))\otimes
\nu_{\partial \omega})  d{\mathcal H^{N-2}},
 \end{aligned}
\end{equation*}
which, together with Remark~\ref{rmk:Hnzero}, concludes the proof.
\end{proof}

\subsection{Upper bound}\label{sect:ub}

In this subsection, we prove the upper bound for \(F\) in Theorem~\ref{thm:FMZ}.
We start by addressing the case in which the limit fields are Sobolev functions, which uses the following density result.

\begin{proposition}
        \label{density}
        Let   $\mathcal{S}$  be the space given by
\begin{equation}
\label{V}
\begin{aligned}
 \mathcal{S}:= \Big\{(u^a, u^b, \xi^a, \xi^b)\colon  
&u^a \in  (\theta^a_1 + C^\infty_{\Gamma^a}([0,1];\mathbb R^d)),\, \, u^b\in
(\theta^b_1
+ C^\infty_0(\omega;\mathbb R^d)),\\
 &\xi^a \in (\theta^a_2 + C_{\Gamma^a}^\infty([0,1];(\mathbb R^{d\times (N-1)})),\,\,
\xi^b \in (\theta^b_2 + C^\infty_0(\omega;\mathbb R^d)), \\ 
& u^a(0) = u^b(0'),\,\,
 \nabla_{x'}
u^b(0') = \xi^a(0),\,\,\nabla_{x_N} u^a(0)
= \xi^b(0')\Big\}, 
\end{aligned}
\end{equation}
with $C^\infty_{\Gamma^a}([0, 1];\mathbb R^d) =
\left\{u^a \in C^\infty([0, 1];\mathbb R^d)\colon u^a(1) = 0,\, \nabla^i
u^a(1) =
0
\enspace\forall i \in \mathbb N\right\}$.
 Then, \(\mathcal{S}\) is dense in the space 
 \begin{equation}\label{Vtilde} 
        \widetilde{\mathcal{S}}:=(\theta^a_1+
W^{2,1}_{\Gamma^a}(]0,1[;\mathbb R^d)) \times(\theta^b_1+{{
W^{2,1}_0}}(\omega;\mathbb R^d))\times
(\theta^a_2+ {{W^{1,1}_{\Gamma^a}}}(]0,1[;\mathbb R^{d\times (N-1)})) \times
 (\theta^b_2+ {{W^{1,1}_{0}}}(\omega;\mathbb R^d))
 \end{equation}
        with respect to the strong topology of  $W^{2,1}\times W^{2,1}\times
W^{1,1}\times W^{1,1}$. If $N=2$, we further have that $u^a(0)=u^b(0)$ for every $(u^a,u^b,\xi^a, \xi^b)\in    \widetilde{\mathcal{S}}$.
\end{proposition}
\begin{proof}
        This result follows as in \cite[(iii) in Proposition 5.1]{GZ}. Indeed, we can
rely on the estimates \cite[(5.2), (5.4)]{GZ} that hold also for $p=1$ in $]0,1[$,
and on the estimates   \cite[(5.10) and (5.11)]{GZ} that are still valid for $p=1$ in $\mathbb
R^{N-1}$. 
Observing that \cite[(5.11)]{GZ} with $p=1$ provides only a bound on the second-order derivatives if $N=3$,  we can apply Mazur's Theorem to conclude the argument. 
 Finally, if $N=2$, we can consider the case \cite[(ii) in Proposition 5.1]{GZ} regarding the convergences in \cite[(5.3)]{GZ}, which also hold in $C(\overline \omega;\mathbb R^d)$; hence, the partial junction condition is satisfied. 
 \end{proof}

\begin{lemma}\label{ubSobolev}  Under the notations and assumptions of Theorem~\ref{thm:FMZ},
with  $V=0$,   \(H_n^a=0\), and \(H_n^b=0\),  we have that
\begin{equation}\label{eq:uppersmooth}
\begin{aligned}
&F[(u^a, u^b), (\xi^a, \xi^b)]\leq   \int_{0}^1 \hat W( \nabla_{x_N}\xi^a,\nabla^2_{x_N}
u^a)\, dx_N+ q \int_\omega W_0( \nabla^2_{x'} u^b,
\nabla_{x'} \xi^b)\, dx'.
\end{aligned}
\end{equation}
for every $(u^a, u^b,\xi^a, \xi^b) \in          \widetilde{\mathcal{S}} $, where  $     \widetilde{\mathcal{S}}$ is the space in \eqref{Vtilde}. 
\end{lemma}

\begin{proof}[Proof]
We start proving \eqref{eq:uppersmooth} for fields $(u^a, u^b, \xi^a, \xi^b)
\in   {\mathcal{S}}$, where \( {\mathcal{S}}\) is the space in \eqref{V}.
  To this end, we take $z^a \in W^{2,1}_0(]0,1[;  (\RR_s^{(N-1)\times)N-1)})^d)$,
$z^b\in W^{2,1}_0(\omega;\mathbb R^d)$,  and we define  $(u^a_n, u^b_n) \in
\mathcal U_n$ as follows.
For some $\{\varepsilon_n\}_{n\in N} \subset ]0, 1[$ suitably chosen (for
example, one can choose $\varepsilon_n = r^2_n$) and  a.e.~$x = (x', x_N)
\in \omega\times]\varepsilon_n, 1[$, we set
\begin{equation*}
\begin{aligned}
u^a_n(x', x_N) : = u^a(x_N) + r_n\xi^a(x_N)  x' + r^2_n (x' )^T  z^a(x_N)
 x',
\end{aligned}
\end{equation*}
while for a.e.~$x = (x', x_N) \in \omega  \times ]0,\varepsilon_n[$, we set
\begin{equation*}
\begin{aligned}
u^a_n(x', x_N) &:= \Big[ -\tfrac{2}{\varepsilon_n^3} u^a(\varepsilon_n)-
2 \tfrac{r^n}{\varepsilon_n^3}\xi^a(\varepsilon_n)
 x'-2 \tfrac{r^2_n}{\varepsilon^3_n}
(x' )^T z^a(\varepsilon_n)
 x' + \tfrac{2}{\varepsilon_n^3}
u^b(r_nx')\\
&\qquad +\tfrac{1}{\varepsilon^2_n}
\nabla_{x_N}u^a(\varepsilon_n) + \tfrac{r_n}{\varepsilon_n^2}
\nabla_{x_N}\xi^a(\varepsilon_n)  x'
+\tfrac{r^2_n}{\varepsilon^2_n}(x')^T
\nabla_{x_N}z^a(\varepsilon_n)  x'+ 
\tfrac{1}{\varepsilon_n^2} \xi^b(r_nx')\Big]x_N^3\\&\quad +\left[\tfrac{3}{\varepsilon^2_n}
u^a(\varepsilon_n) +\tfrac{3 r_n}{\varepsilon^2_n}\xi^a(\varepsilon_n) x'
+ \tfrac{3 r^2_n}{\varepsilon^2_n} (x')^T z^a(\varepsilon_n) x'-\tfrac{3}{\varepsilon^2_n}
u^b(r_n x')
\right.\\
&\qquad \left. - \tfrac{1}{\varepsilon_n} \nabla_{x_N}
u^a(\varepsilon_n) -\tfrac{r_n}{\varepsilon_n} \nabla_{x_N}\xi^a(\varepsilon_n)
 x'-
 \tfrac{r^2_n}{\varepsilon_n} (x')^T\nabla_{x_N} z^a(\varepsilon_n) x' -\tfrac{2}{\varepsilon_n}
\xi^b(r_n x')\right] x^2_N\\
&\quad +\xi^b(r_nx')x_N + u^b(r_nx'). 
\end{aligned}
\end{equation*}
For
 a.e. $x = (x',x_N) \in \Omega_b$, we set
\begin{equation*}
\begin{aligned}
u^b_n(x',x_N):= u^b(x') +h_n \xi^b(x')x_N+ h^2_n x^2_Nz^b(x').
\end{aligned}
\end{equation*}

Recalling \eqref{Vn},  it can be checked that $(u^a_n, u^b_n) \in \mathcal U_n$, $(u^a_n, u^b_n)\to
(u^a, u^b) $ strongly in $ W^{2,1}(\Omega^a;\mathbb R^d)\times 
W^{2, 1}(\Omega^b;\mathbb R^d),$ and
$(\frac{1}{r_n}\nabla_{x'} u^a_n ,\frac{1}{h_n} \nabla_{x_N} u^b_n) \to 
 (\xi^a, \xi^b)$ strongly in $W^{1,1}(\Omega^a;\mathbb R^{d \times( N-1)})\times
W^{1,1}(\Omega^b;\mathbb R^d).$
Then, using the continuity of \(W\) and  \eqref{coerci},  Vitali--Lesbesgue's theorem yields
 
\begin{align}\label{4.5I}
\lim_{n\to\infty}K_n^a[u^a_n] = \int_{\Omega^a} W\left(\begin{array}{ll}
2 z^a (x_N )& (\nabla_{x_N} \xi^a(x_N))^T\\
(\nabla_{x_N} \xi^a(x_N)) & \nabla^2_{x_N} u^a(x_N)
\end{array}\right) dx' dx_N
\end{align}
 and
 \begin{align}\label{4.6I}
\lim_{n\to\infty} K_n^a[u^b_n] = \int_{\Omega^a} W\left(\begin{array}{ll}
 \nabla^2_{x'}u^b(x') & (\nabla_{x'} \xi^b(x'))^T\\
(\nabla_{x'} \xi^b(x')) & 2 z^b(x')
\end{array}\right) dx' dx_N,
\end{align}
where $z^a \in  W^{2,1}_0(]0,1[;\mathbb (\RR_s^{(N-1)\times( N-1)})^d)$ and
$z^b \in W^{2,1}_0(\omega;\mathbb R^d)$ can be arbitrarily chosen.
Hence, arguing as in the estimates \cite[(4.7), (4.8), and (4.9)]{GZNODEA},
replacing the spaces $W^{2,p}_0(]0,1[;\mathbb R^d)$ and $W^{2,p}_0(\omega;\mathbb
R^d)$ by $W^{2,1}_0(]0,1[;\mathbb (\RR_s^{(N-1)\times( N-1)})^d)$ and $ W^{2,1}_0(\omega;\mathbb
R^d)$ and the spaces $L^p(]0,1[;\mathbb R^d)$ and $L^p(\omega;\mathbb R^d)$
by $L^1(]0,1[;\mathbb R^d)$ and $L^1(\omega;\mathbb R^d)$, respectively,
using the Dominated convergence theorem, Aumann' selection lemma, and definitions
$W_0$ and $\hat W$ (see \eqref{eq:defhatzero}), from \eqref{4.5I} and \eqref{4.6I}
we obtain that
\begin{equation*}
\begin{aligned}
F[(u^a, u^b), (\xi^a, \xi^b)] &\leq  \liminf_{n\to +\infty}\Big(K_n^a[u^a_n]+
\tfrac{h_n}{r_n^{N-1}}K_n^b[u^b_n]\Big)\\
& \leq \int_0^1 \hat W( \nabla_{x_N}\xi^a,\nabla^2_{x_N} u^a)\, dx_N+ q\int_\omega
W_0( \nabla^2_{x'} {u}^b,\nabla_{x'}{\xi}^b)\,dx'.
\end{aligned}
\end{equation*}

The proof is concluded by combining the previous arguments with a standard
diagonalization argument and Proposition~\ref{density}. 
\end{proof}

\begin{lemma}\label{lscGammalimit}
Under the notations and  assumptions of Theorem~\ref{thm:FMZ},
 the functional $F$ in \eqref{Gammalimitdef} is sequentially lower semicontinuous
in $U\times \Xi$ with respect to the $BH\times BH\times BV \times BV$ weak-\(*\)
convergence; i.e.,\begin{align*}
F[(u^a, u^b), (\xi^a, \xi^b)] \leq \liminf_{n\to +\infty} F[(u^a_n, u^b_n),
(\xi^a_n, \xi^b_n)]
\end{align*}
whenever $ u^a_n \overset{\ast}{\rightharpoonup} u^a$  in $BH(\Omega^a;\mathbb
R^d)$, $u_n^b  \overset{\ast}{\rightharpoonup} u^b$ in $BH(\Omega^b;\mathbb
R^d)$, $\xi^a_n \overset{\ast}{\rightharpoonup} \xi^a$ in $BV(\Omega^a;\mathbb
R^{d \times (N-1)})$, $\xi^b_n \overset{\ast}{\rightharpoonup} \xi^b$  in
 $BV(\Omega^b;\mathbb R^d)$.
\end{lemma}
\begin{proof}[Proof]
The proof is standard, relying on diagonalization arguments, the definition
of $F$, and the coercivity of the energy density, cf. ~the arguments in \cite[Proposition
5.1]{ChoksiFonseca} for instance.
\end{proof}
\begin{lemma}\label{ubSobolev2}
Under the notations and  assumptions of Theorem~\ref{thm:FMZ},
with  $V=0$,   \(H_n^a=0\), and \(H_n^b=0\),  
we have 
%
\begin{equation*}
\begin{aligned}
&F[(u^a, u^b), (\xi^a, \xi^b)]\leq   \int_{0}^1 \hat W^{**}(\nabla_{x_N}\xi^a,\nabla^2_{x_N}
u^a)\, dx_N+ q \int_\omega Q_{\widetilde{\mathcal A}}W_0( \nabla^2_{x'} u^b,
\nabla_{x'} \xi^b)\, dx'.
\end{aligned}
\end{equation*}
 for every $(u^a,u^b, \xi^a, \xi^b) \in         \widetilde{\mathcal{S}}$, where \(        \widetilde{\mathcal{S}}\) is the space  in \eqref{Vtilde}.
\end{lemma}
\begin{proof}[Proof]
The proof is a consequence of Lemma \ref{lscGammalimit} and \cite[Lemma~3.1,
Corollary~3.2, and Theorem~3.6]{BrFoLe00}. Indeed,  recalling \eqref{Vtilde}, we are exploiting the
fact that the latter relaxation result with respect to weak convergence in
$L^1$ for $\mathcal A$-free fields, is applied separately to 
the functional 
$\int_0^1 \hat W(\nabla_{x_N}\xi^a,\nabla^2_{x_N}
u^a)\,dx$ defined
in $ (\theta^a_2+{{W^{1,1}_{\Gamma^a}}}(]0,1[;\mathbb
R^{d \times (N-1)})\times (\theta^a_1+ {{W^{2,1}_{\Gamma^a}}}(]0,1[;\mathbb
R^d))$ and to the functional  $\int_\omega W_0( \nabla^2_{x'}
u^b,
\nabla_{x'} \xi^b)\, dx'$ defined in $ (\theta^b_1+ W^{2,1}_0(\omega;\mathbb
R^d))\times (\theta^b_2 +W^{1,1}_0(\omega;\mathbb R^d))$. 
The fact that the we have a clamped problem can be bypassed applying arguments
entirely similar to those in \cite[Lemma 2.2]{BFMbend}, taking into account
that we can consider approximating fields which have the same boundary data
as the target.
 In the \(N=2\) case, we observe that  $\omega = (\alpha,\beta)$ for some $\alpha<0<\beta$;  hence, for every $(u^a, u^b,\xi^a, \xi^b) \in      \widetilde{\mathcal{S}}$ such that $u^a(0)= u^b(0)$, the spaces $\theta^a_1+{{W^{2,1}_{\Gamma^a}}}(]0,1[;\mathbb
R^{d \times (N-1)})$ and  $\theta^b_1+ W^{2,1}_0(\omega;\mathbb
R^d)$  can be replaced by ${\tilde \theta^a_1}+{W^{2,1}_0}(]0,1[;\mathbb
R^{d \times (N-1)})$ and ${\tilde\theta^b_1}+ W^{2,1}_0((\alpha,0);\mathbb
R^d)\cup {\tilde\theta^b_1}+ W^{2,1}_0((0,\beta);\mathbb
R^d)$, respectively, with  $\tilde{\theta^a_1}(1)= \theta^a_1(1)$, $\tilde{\theta^a_1}(0)= u^a(0)$, $\tilde{\theta^b_1}(\alpha)= \theta^b_1(\alpha)$, $\tilde{\theta^b_1}(\beta)= \theta^b_1(\beta)$,  $\tilde{\theta^b_1}(0)=u^a(0)$.
\end{proof}

We are finally in position to prove the upper bound for \(F\) in Theorem~\ref{thm:FMZ}.
\begin{proposition}\label{ubprop}
Under the assumptions of Theorem~\ref{thm:FMZ}, we have that if $(( u^a,
 u^b),(  \xi^a,  \xi^b))\in \mathcal U \times \Xi$ and    $\{(u_n^a, u_n^b)\}_{n\in\NN}
\subset \mathcal U_n$ are such that $u_n^a \overset{\ast}{\rightharpoonup}
u$  in $BH(\Omega^a; \mathbb R^d)$, 
$u_n^b  \overset{\ast}{\rightharpoonup} u^b$ in $BH(\Omega^b;\mathbb
R^d)$,   $\frac{1}{r_n}\nabla_{x'}u^a_n \overset{\ast}{\rightharpoonup} \xi^a$
in $BV(\Omega^a;\mathbb R^{d \times (N-1)})$, and $ \frac{1}{h_N} \nabla_{x_N}u^b_n
\overset{\ast}{\rightharpoonup} \xi^b$ in $ BV(\Omega^b;\mathbb R^d)$, then
\begin{equation}
\label{liminfineq}
\limsup_{n\to\infty} F_n[u_n]=\limsup_{n\to\infty}\left(K^a_n[  u^a_n]+
\frac{h_n}{r_n^{N-1}}K^b_n[  u^b_n]\right) \leq  K^a[({u}^a, \xi^a)]+ q
K^b[({u}^b,{\xi}^b)].
\end{equation}
\end{proposition}
\begin{proof}[Proof] 
 We start by considering the $N\geq 3$ case. 
We first observe that by Remark~\ref{rmk:Hnzero}, we may assume that   $V=0$,  
 \(H_n^a=0\), and \(H_n^b=0\)   without
loss of generality. Moreover, by \eqref{coerci}, and up to considering $W+C$,  there
is also no loss of generality in assuming that $W\geq 0$.

Then, the proof is a consequence of Lemmas \ref{lscGammalimit} and \ref{ubSobolev2}
and Theorem \ref{ReshetnyaktheoremBCMS} applied separately to both the energies
$\int_0^1 \hat W^{**}(\nabla\xi^a,\nabla^2
u^a)\, dx_N$ defined in  $ (\theta^a_2+ {{W^{1,1}_{\Gamma^a}}}(]0,1[;\mathbb
R^{d\times (N-1)}))\times (\theta^a_1+ {{W^{2,1}_{\Gamma^a}}}(]0,1[;\mathbb
R^d))$  and $\int_\omega Q_{\mathcal A^2}W_0( \nabla^2u^b,
\nabla\xi^b)\, dx'$  defined in $( \theta^b_1+ W^{2,1}_0(\omega;\mathbb R^d))\times
(\theta^b_2+ W^{1,1}_{0}(\omega;\mathbb R^d))$, where in the first case we
take into account Remark~\ref{afterProp 4.8}.
For the readers' convenience, we provide next the details of these  applications.

In view of Lemma~\ref{lscGammalimit}, and the inequality provided by Lemma~\ref{ubSobolev2},
it results that 
\begin{equation}\label{convfin}
\begin{aligned}
&F[(u^a, u^b), (\xi^a, \xi^b)]\\ &\quad \leq\inf\left\{\liminf_{n\to\infty}
\int_0^1 \hat W^{\ast \ast}(\nabla
\xi^a_n,\nabla^2 u^a_n)\,d x_N+ q \int_\omega Q_{\widetilde{\mathcal A}}W_0(\nabla^2
u^b_n, \nabla \xi^b_n)\,
dx'\colon \right.
\\&\hskip45mm (u^a_n, \xi^a_n)\in( \theta^a_1+ {{W^{2,1}_{\Gamma^a}}}(]0,1[;\mathbb
R^d))  \times 
 (\theta^a_2+ {{W^{1,1}_{\Gamma^a}}}(]0,1[;\mathbb R^{d\times (N-1)})),\\
 &\hskip45mm (u^b_n, \xi^b_n)\in (\theta^b_1+ {{W^{2,1}_0}}(\omega;\mathbb
R^d))\times (\theta^b_2+ W^{1,1}_{0}(\omega;\mathbb R^d)),
 \\&\hskip45mm u^a_n \overset{\ast}{\rightharpoonup} u^a \hbox{ in }BH(]0,1[;\mathbb
R^d),\,
u_n^b  \overset{\ast}{\rightharpoonup} u^b \hbox{ in }BH(\omega;\mathbb R^d),
\\&\hskip45mm \xi^a_n \overset{\ast}{\rightharpoonup} \xi^a \hbox{ in }BV(]0,1[;\mathbb
R^{d \times (N-1)}),\, \xi^b_n \overset{\ast}{\rightharpoonup} \xi^b \hbox{
in } BV(\omega;\mathbb R^d)\Big\}\\
&\quad = \inf \left\{\liminf_{n\to\infty} \int_0^1 \hat W^{\ast \ast}(\nabla
\xi^a_n,\nabla^2 u^a_n)\,d x_N\colon  \right.
\\&\hskip45mm  (u^a_n, \xi^a_n)\in( \theta^a_1+ {{W^{2,1}_{\Gamma^a}}}(]0,1[;\mathbb
R^d))  \times 
 (\theta^a_2+ {{W^{1,1}_{\Gamma^a}}}(]0,1[;\mathbb R^{d\times (N-1)})),
\\
&\hskip45mm u^a_n \overset{\ast}{\rightharpoonup} u^a \hbox{ in }BH(]0,1[;\mathbb
R^d),\,
 \xi^a_n \overset{\ast}{\rightharpoonup} \xi^a \hbox{ in }BV(]0,1[;\mathbb
R^{d \times (N-1)})\Big\}\\ 
&\qquad  + q \inf \left\{\liminf_{n\to\infty} \int_\omega Q_{\widetilde{\mathcal
A}}W_0(\nabla^2 u^b_n,
\nabla \xi^b_n) dx'\colon  \right.
\\&\hskip45mm  (u^b_n, \xi^b_n)\in (\theta^b_1+ W^{2,1}_0(\omega;\mathbb
R^d))\times (\theta^b_2+ W^{1,1}_{0}(\omega;\mathbb R^d)),
\\
&\hskip45mm u_n^b  \overset{\ast}{\rightharpoonup} u^b \hbox{ in }BH(\omega;\RR^d),
\xi^b_n \overset{\ast}{\rightharpoonup} \xi^b \hbox{ in } BV(\omega;\RR^d)\Big\}.
\end{aligned}
\end{equation}

Fix \(\delta>0\), and consider the sequence $\{(\bar w_n,
\bar u_n)\}_{n\in\NN} \subset W^{2,1} (]0,1+\delta[);\RR^d)\times W^{1,1}(]0,1+\delta[;
\mathbb
R^{d \times (N-1)})$ defined by 
\begin{equation*}
\begin{aligned}
&\bar w_n :=  \begin{cases}
w_n &\text{in }
]0,1[,\\
\theta^a_1 &\text{in } ]1, 1+\delta[,
\end{cases} \quad \text{and} \quad
& \bar \zeta_n  :=  \begin{cases}
\zeta_n &\text{in } ]0,1[,\\
\theta^a_2 &\text{in } ]1, 1+\delta[,
\end{cases}
\end{aligned}
\end{equation*}
 where $\{( w_n,
 \zeta_n)\}_{n\in\NN} \subset (\theta^a_1+ {{W^{2,1}_{\Gamma^a}}}(]0,1[;\mathbb
R^d))\times(\theta^a_2 + {{W^{1,1}_{\Gamma^a}}}(]0,1[;R^{d \times (N-1)}))
 $ is the sequence constructed in Remark \ref{afterProp 4.8} that converges to \((u^a,\xi^a)\). We have that \(\{(\bar w_n,
\bar u_n)\}_{n\in\NN}\) converges in $BH \times BV$ weakly-\(\ast\) and (area)
\(\langle\cdot\rangle\)-strictly to \((\bar u,\bar \xi)\) given by 
\begin{equation*}
\begin{aligned}
&\bar u  :=  \begin{cases}
u^a  &\text{in } ]0,1[,\\
\theta^a_1 &\text{in } ]1, 1+\delta[,
\end{cases}\quad \text{and} \quad \bar \zeta  :=  \begin{cases}
\xi^a  &\text{in } ]0,1[,\\
\theta^a_2 &\text{in } ]1, 1+\delta[,
\end{cases}
\end{aligned}
\end{equation*}
and, exploiting the nonnegativity $W$ and Reshetnyak's upper semicontinuity
result in Theorem~\ref{ReshetnyaktheoremBCMS},  
\begin{equation*}
\begin{aligned}
& \inf \bigg\{\liminf_{n\to\infty} \int_0^1 \hat W^{\ast \ast}( \nabla \xi_n,\nabla^2
u_n)\,d x_N\colon   (u_n, \xi_n)\in( \theta^a_1+ {{W^{2,1}_{\Gamma^a}}}(]0,1[;\mathbb
R^d))  \times 
 (\theta^a_2+ {{W^{1,1}_{\Gamma^a}}}(]0,1[;\mathbb R^{d\times (N-1)})),
\\
&\hskip60.5mm u_n \overset{\ast}{\rightharpoonup} u^a \hbox{ in }BH(]0,1[;\mathbb
R^d),\,
 \xi_n \overset{\ast}{\rightharpoonup} \xi^a \hbox{ in }BV(]0,1[;\mathbb
R^{d \times (N-1)})\bigg\}\\
&\quad \leq\liminf_{n\to\infty} \int_0^1 \hat W^{\ast \ast}(
\nabla u_n,\nabla^2
w_n)\,d x_N \leq \limsup_{n\to\infty} \int_0^{1+\delta} \hat W^{\ast \ast}(
\nabla\bar  u_n,\nabla^2
\bar w_n)\,d x_N\\
&\quad \leq \int_0^{1+\delta} {\hat W}^{\ast \ast} 
\big(\nabla_{x_N}\bar\xi(x_N),D^2_{x_N} \bar u(x_N)\big) dx_N \\&\qquad+
\int_0^{1+\delta} (\hat W^{**})^\infty\left(\frac{d( (D_{x_N})^s \bar \xi,
(D^2_{x_N})^s \bar
u)}{d |( (D_{x_N})^s \bar \xi,
(D^2_{x_N})^s \bar
u)|} (x_N)\right)d |(D^2_{x_N})^s \bar u, (D_{x_N})^s \bar \xi| (x_N)\\
&\quad = \int_0^{1} {\hat W}^{\ast \ast} 
\big(\nabla_{x_N}\xi^a(x_N),D^2_{x_N}u^a(x_N)\big) dx_N  \\
&\qquad + \int_0^{1} (\hat W^{**})^\infty\left(\frac{d((D_{x_N})^s \xi^a
,(D^2_{x_N})^s  u^a )}{d |((D_{x_N})^s \xi^a
,(D^2_{x_N})^s  u^a)|}(x_N)\right)d
|((D_{x_N})^s \xi^a
,(D^2_{x_N})^s  u^a)|(x_N)\\
&\qquad +({\hat W}^{**})^\infty \left( \theta^a_2-(\xi^a)^-(1),-\left(\frac{d
u^a}{dx_N}   \right)^-(1) \right) +
\int_1^{1+\delta} {\hat W}^{\ast \ast} 
(0, 0\big) \,dx_N . 
\end{aligned}
\end{equation*}

Thus, letting $\delta \to 0$ yields
\begin{equation*}
\begin{aligned}
& \inf \bigg\{\liminf_{n\to\infty} \int_0^1 \hat W^{\ast \ast}( \nabla \xi_n,\nabla^2
u_n)\,d x_N\colon   (u_n, \xi_n)\in( \theta^a_1+ {{W^{2,1}_{\Gamma^a}}}(]0,1[;\mathbb
R^d))  \times 
 (\theta^a_2+ {{W^{1,1}_{\Gamma^a}}}(]0,1[;\mathbb R^{d\times (N-1)})),
\\
&\hskip61mm u_n \overset{\ast}{\rightharpoonup} u^a \hbox{ in }BH(]0,1[;\mathbb
R^d),\,
 \xi_n \overset{\ast}{\rightharpoonup} \xi^a \hbox{ in }BV(]0,1[;\mathbb
R^{d \times (N-1)})\bigg\}\\
&\quad \leq K^a[(u^a, \xi^a)]. 
\end{aligned}
\end{equation*}

On the other hand, relying on the Lipschitz continuity of $Q_{\widetilde{\mathcal
A}}W_0$, Proposition \ref{BFMTracestog}, exploiting an extension from $\omega$
to $\omega' \supset \supset \omega$, we can construct a $BH(\omega';\mathbb
R^d)\times BV(\omega';\mathbb R^d)$ weak-\(\ast\) converging sequence, $(w_\varepsilon,
\zeta_\varepsilon)$, such that $(\nabla^2 w_\varepsilon, \nabla\zeta_\varepsilon)$
converges (area) \(\langle\cdot\rangle\)-strictly in $\mathcal \omega'$.
Thus, as above, we are entitled to apply Reshetnyak's upper semicontinuity
result in Theorem~\ref{ReshetnyaktheoremBCMS}  to  deduce that
\begin{equation*}
\begin{aligned}
\inf \left\{\liminf_{n\to\infty} \int_\omega Q_{\widetilde{\mathcal A}}W_0(\nabla^2
u^b_n,
\nabla \xi^b_n)\, dx'\colon  \right.
&  (u^b_n, \xi^b_n)\in (\theta^b_1+ W^{2,1}_0(\omega;\mathbb
R^d))\times (\theta^b_2+ W^{1,1}_{0}(\omega;\mathbb R^d)),
\\
& u_n^b  \overset{\ast}{\rightharpoonup} u^b \hbox{ in }BH(\omega;\RR^d),\,
\xi^b_n \overset{\ast}{\rightharpoonup} \xi^b \hbox{ in } BV(\omega;\RR^d)\Big\}\leq
K_b(u^b,\xi^b),
\end{aligned}
\end{equation*}
which concludes the proof in this case.

Observe that if $N=2$, then we can rely on the embeddings  of $BH(]0,1[;\mathbb R^d])$ and $BH(\omega;\mathbb R^d)$ in $C([0,1];\mathbb R^d)$ and $C(\overline \omega;\mathbb R^d)$, respectively, as stated in Proposition \ref{embedding}. Moreover, the convergences $u^a_n \overset{\ast}{\rightharpoonup} u^a$ in $BH(]0,1[;\mathbb R^d])$ and of  $u^b_n \overset{\ast}{\rightharpoonup} u^b$  in $BH(\omega;\mathbb R^d)$ are uniform in $C([0,1],\mathbb R^d)$ and $C(\overline \omega;\mathbb R^d),$  respectively. Hence ,$u^a(0)=u^b(0)$, as proven in Proposition \ref{Compactnessres}. 
These facts entail that the admissible sequences in \eqref{convfin} naturally lead to a partially coupled limit problem.
\end{proof}

\begin{proof}[Proof of Theorem~\ref{thm:FMZ}]
The statement in Theorem~\ref{thm:FMZ} follows by 
 Proposition~\ref{Compactnessres}, Proposition~\ref{lbprop},
and Proposition~\ref{ubprop}.  
\end{proof}

\subsection{Proof of Theorem~\ref{thm:mainFMZ} and Corollary~\ref{cor:min}}\label{sect:proofmain}
We finally prove our main results, which are a simple consequence
of previous results and well-know properties of \(\Gamma\)-convergence.

\begin{proof}[Proof of Theorem~\ref{thm:mainFMZ}]
The statement in Theorem~\ref{thm:mainFMZ} is a direct consequence
 of \cite[Proposition~8.3]{Da93},  Theorem~\ref{thm:FMZ}, and
 Proposition~\ref{Compactnessres}.  
\end{proof}

\begin{proof}[Proof of  Corollary~\ref{cor:min}]
The statement in Corollary~\ref{cor:min} follows by
  \cite[Corollary~7.20]{Da93} together with Theorem~\ref{thm:mainFMZ}
and Proposition~\ref{Compactnessres}.
\end{proof}

\section{Auxiliary results}\label{sect:aux}

In this section, we prove some of the results we used in our arguments in
the preceding sections, the first of which Theorem~\ref{Lemma1KRBH}.


\begin{proof}[Proof of Theorem~\ref{Lemma1KRBH}]
        We follow closely the arguments used in the proof of
\cite[Lemma 1]{KR}.

        Fix \(\varepsilon>0\). We claim that   there exists $v_\varepsilon
\in  W^{2,1}_u(\calO;\mathbb R^m)\cap C^\infty(\calO;\mathbb R^m)$ such that
\begin{equation*}
\begin{aligned}
 \|u- v_\varepsilon\|_{L^1} < \varepsilon \hbox{ and  } \langle
D^2 v_\varepsilon \rangle (\calO)\leq \langle D^2u \rangle(\calO)+
4 \varepsilon,
\end{aligned}
\end{equation*}
from which the conclusion follows in view of Theorem~\ref{interpolBH}.
        
To prove the claim, let $\{\calO_i\}_{i \in \mathbb N_0}$ be
a sequence of open subsets of $\calO$  such that $\bigcup_{i
\in \mathbb N_0}\calO_i = \calO$ and $\calO_i \Subset \calO_{i+1}$
for every \(i\in\N_0\).
Using the fact that   $|D^2u|$ is a finite measure, and renumbering
the sets if necessary, we may further assume that 
\begin{align}\label{d2uest}
|D^2u|(\calO \setminus \calO_0) < \varepsilon.
\end{align}

Define $U_{-1} := \emptyset$, $U_0 := \calO_1$, and  $U_i :=
\calO_{i+1}\setminus \overline \calO_{i-1}$ for $i \in \mathbb
N$. Then, $\{U_i\}_{i \in \mathbb N_0}$ is an open cover of $\calO$
with the property that each point of $\calO$ belongs to at
most two of the sets $U_i$. We can then find a $C^\infty$-smooth
partition of unity,  $\{\varrho_i\}_{i \in \mathbb
N_0}$,   subordinated to the cover $\{U_i\}_{i \in \mathbb
N_0}$; that is, for all \(i\in\N_0\), $\varrho_i \in C^\infty_c
(U_i)$, $0\leq \varrho_i\leq 1$, and $\sum_{i=0}^\infty \varrho_i
= \chi_{\calO}$.

Next, denoting by $B_\ell$  the
unit ball in \(\RR^\ell\) centered at the origin,  fix $\varphi
\in C^\infty_c(B_\ell)$  with $\varphi \geq 0$,  $\int_{\mathbb
R^d} \varphi dx=1$, and
$\varphi_\eta (x):=\eta^{-\ell}\varphi(x/\eta)$ for $x \in \mathbb
R^\ell$ and $\eta >0$.
Then, extending 
 $\varrho_i u$ and $u\otimes \varrho_i$
to $\mathbb R^d\setminus \overline\calO$ in the natural way,
standard mollification arguments ensure that for each $i \in
\mathbb N_0$, there exists $\eta_i >0$ such that 
\begin{align}
&{\rm supp} (\varphi_{\eta_i} \star (\varrho_iu))\subset
U_i, &\label{1}
\\
&\int_\calO |\varrho_i u- \varphi_{\eta_i}\star (\varrho_i u)|
dx < \varepsilon 2^{-i-1}, &\label{2}\\
&\int_{\calO} |u \otimes \nabla  \varrho_i- \varphi_{\eta_i}\star(u\otimes
\nabla \varrho_i)| dx < \varepsilon 2^{-i-1}, &\label{3}\\
&\int_{\calO} |u \otimes D^2 \varrho_i- \varphi_{\eta_i}\star(u\totimes
D^2\varrho_i)| dx < \varepsilon 2^{-i-1}, &\label{4}\\
&\int_{\calO} \varrho_i (\varphi_{\eta_i} \star \langle \varrho_i
Du\rangle) dx \leq \int_\calO \varrho_i d \langle \varrho_i Du
\rangle+ \varepsilon \fint_{\calO} \varrho_i dx, &\label{5}\\
&\int_{\calO} |\varphi_{\eta_i} \star (\nabla u \totimes \nabla
\varrho_i + \nabla \varrho_i \totimes \nabla u) - (\nabla u \totimes
\nabla \varrho_i + \nabla \varrho_i \totimes \nabla u) |dx <
\varepsilon 2^{-i-1}, &\label{6}\\
&\int_{\calO} \varrho_i (\varphi_{\eta_i} \star \langle \varrho_i
D^2u\rangle) dx \leq \int_\calO \varrho_i d \langle \varrho_i
D^2u \rangle+ \varepsilon \fint_{\calO} \varrho_i dx,& \label{7}
\end{align}
 where the convolution with $\varphi_{\eta_i}$ is understood
componentwise, and 
\begin{equation*}
\begin{aligned}
& u\totimes D^2\varrho_i := \bigg(\Big(u^k\frac{\partial^2 \rho_i}{\partial
x_j \partial x_l}\Big)_{1\leq l\leq \ell\atop 1\leq j\leq \ell}\bigg)_{k=1}^m
\in \big(\RR^{\ell \times \ell}\big)^m,\\
& \nabla \rho_i\totimes \nabla u := \bigg(\Big(\frac{\partial
\rho_i}{\partial x_l} \frac{\partial u^k}{\partial x_j}\Big)_{1\leq
l\leq m\atop
1\leq j\leq \ell}\bigg)_{k=1}^m \in \big(\RR^{\ell \times \ell}\big)^m,\\
& \nabla u \totimes \nabla\varrho_i := \bigg(\Big(\frac{\partial
\rho_i}{\partial
x_j}\frac{\partial u^k}{\partial
x_l}\Big)_{1\leq l\leq m\atop
1\leq j\leq \ell}\bigg)_{k=1}^m \in \big(\RR^{\ell \times \ell}\big)^m.
\end{aligned}
\end{equation*}

Let  $v_\varepsilon \in C^\infty(\calO;\RR^m)$  be the function
defined for \(x\in \calO\) by
\begin{align*}
        v_\varepsilon(x) :=\sum_{i=0}^\infty\big(\varphi_{\eta_i}\star
(\varrho_i u)\big)(x).
\end{align*}
Observing that 
$u = \sum_{i=0}^\infty \varrho_i u$, we have 
\begin{align*}
        \int_\calO |v_\varepsilon- u| dx \leq \sum_{i=0}^\infty
\int_\calO |\varrho_i u- \varphi_{\eta_i}\star (\varrho_i u)|dx
< \varepsilon.
\end{align*}
Moreover, because 
$\sum_{i=0}^\infty \nabla \varrho_j=0$ and $ \sum_{i=0}^\infty
D^2 \varrho_i =0$ in $\calO$, it follows that 
\begin{equation*}
\begin{aligned}
D^2 v_\varepsilon &= \sum_{i=0}^\infty \varphi_{\eta_i} \star
(\varrho_i
D^2 u) +
        \sum_{i=0}^\infty \varphi_{\eta_i} \star \left( u \totimes
D^2 \varrho_i
+ (\nabla u \totimes \nabla \varrho_i + \nabla \varrho_i \totimes
\nabla u)\right) \\
        &=\sum_{i=0}^\infty \varphi_{\eta_i} \star (\varrho_i
D^2 u) + \sum_{i=0}^\infty
\left[(\varphi_{\eta_i} \star (u \totimes D^2 \varrho_i) -u \totimes
D^2 \varrho_i
) \right]  \nonumber\\
        &\quad +\sum_{i=0}^\infty \left[\varphi_{\eta_i} \star
(\nabla
u \totimes \nabla \varrho_i + \nabla \varrho_i \totimes \nabla
u) - (\nabla
u \totimes \nabla \varrho_i + \nabla \varrho_i \totimes \nabla
u)\right].
\end{aligned}
\end{equation*}
Recalling that each point of $\calO$
belongs to at
most two of the sets $U_i$, we conclude that    the sum above
at each point in \(\calO\) contains at most two non null summands;
this observation together with \eqref{4} and \eqref{6} allows
to write 
$$
D^2 v_\varepsilon = \varphi_{\eta_i} \star (\varrho_i D^2 u)+
E_i \hbox{ on } U_i, 
$$
where
\begin{equation}
\label{eq:estEi}
\begin{aligned}
\sum_{i=0}^\infty\int_{U_i} |E_i| dx \leq |D^2 u|(\calO \setminus
\overline
\calO_0)+ 2\varepsilon<3\varepsilon.
\end{aligned}
\end{equation}
Then,  \eqref{eq:estsumang} and  the properties
of $\varrho_i$ yield
\begin{align}\label{tosum}
\int_\calO\varrho_i \langle D^2v_\varepsilon\rangle dx \leq \int_\calO
\varrho_i\langle \varphi_{\eta_i}\star (\varrho_i D^2 u)\rangle
dx 
+ \sum_{i=0}^\infty\int_{U_i}|E_i| dx.
\end{align}
On the other hand, Jensen's inequality and \eqref{7} imply that
\begin{align*}
\int_\calO \varrho_i\langle \varphi_{\eta_i}\star (\varrho_i
D^2 u)\rangle
dx &\leq \int_{\calO} \varrho_i (\varphi_{\eta_i} \star \langle
\varrho_i D^2u\rangle) dx\\& \leq \int_\calO \varrho_i d \langle
\varrho_i D^2u \rangle+ \varepsilon \fint_{\calO} \varrho_i dx
\leq\int_\calO \varrho_i d\langle D^2 u\rangle +  \varepsilon
\fint_{\calO} \varrho_i dx.
\end{align*}
Summing over $i \in \mathbb N_0$ in \eqref{tosum},
the preceding estimate and \eqref{eq:estEi} yield
\begin{align*}
        \langle D^2 v_\varepsilon\rangle(\calO) \leq \langle
D^2 u\rangle(\calO) + 4 \varepsilon.
\end{align*}

We are left to show that $v_\varepsilon \in W^{2,1}_{u}(\calO;\mathbb
R^m)$, the Dirichlet class
introduced in \eqref{W21u}. In other words, we are left to prove
that
\begin{equation*}
\begin{aligned}
 w:= \begin{cases}
 u-v_\varepsilon &\hbox{ in }\calO,\\
        0 & \hbox{ in }\mathbb R^\ell \setminus \overline \calO.
\end{cases}
\end{aligned}
\end{equation*}
belongs  $BH(\mathbb R^\ell ;\mathbb R^m)$ with $|D^2w|(\partial
\calO) = 0$. 

Recalling that $u \in W^{1,1}(\calO;\mathbb R^m)$, we can use
the arguments in the proof of  \cite[Lemma~1]{KR} to deduce that
$ w \in W^{1,1}(\mathbb R^\ell;\mathbb R^m)$.
It remains to prove that $\nabla w \in BV(\mathbb R^{\ell};\mathbb
R^{ m\times \ell})$ and  $|D^2w|(\partial \calO) =0$.

Let $\psi \in C^\infty_c (\mathbb R^\ell )$.
Using the condition   \(\mathcal{L}^\ell(\partial\calO)=0\) and  \eqref{1},
we get for every   $1 \leq  l,j\leq \ell$ and
$ 1\leq k\leq m$  that
\begin{equation}
\label{eq:boundaryeval}
\begin{aligned}
\int_{\mathbb R^\ell} \frac{\partial w^k}{\partial x_l}\frac{\partial
\psi}{\partial x_j}dx &=
\int_\calO \frac{\partial(v_\varepsilon -u)^k}{\partial x_l}
\frac{\partial \psi}{\partial x_j} dx=
        \sum_{i=0}^\infty\int_\calO \left(\varphi_{\eta_i} \star
\frac{\partial (\varrho_i u)}{\partial x_l} - \frac{\partial
(\varrho_i u)}{\partial x_l}\right)^k\frac{\partial \psi}{\partial
x_j}dx\\
        &=\sum_{i=0}^\infty \left[\int_{U_i} \psi  \varrho_i
d D^2_{lj} u^k- \int_{U_i} \varphi_{\eta_i} \star(\varrho_i 
D^2_{lj}u^k) \psi dx
        \right]\\
        &\quad+\sum_{i=0}^\infty \int_{U_i}\psi \left[\frac{\partial
u^k}{\partial x_l}\frac{\partial \varrho_i}{\partial x_j}- \varphi_{\eta_i}
\star \left(\frac{\partial u^k}{\partial y_l}\frac{\partial \varrho_i}{\partial
y_j}\right) \right] dx\\ &\quad + \sum_{i=0}^\infty \int_{U_i}\psi
\left[u^k\frac{\partial^2 \varrho_i}{\partial x_j\partial x_l}-
\varphi_{\eta_i} \star \left(u^k\frac{\partial^2
\varrho_i}{\partial y_j\partial y_l}\right) \right]
dx.
\end{aligned}
\end{equation}

From \eqref{4} and \eqref{6}, we get that
$$
\int_{\mathbb R^\ell} \left|\frac{\partial w^k}{\partial x_l}\frac{\partial
\psi}{\partial x_j}\right| dx \leq 2\|\psi\|_{L^\infty} (|D^
2 u| (\calO)+\varepsilon).
$$
Hence, $w \in BH(\mathbb R^\ell; \mathbb R^m)$.

Next, we prove that $|D^2 w|(\partial \calO) =0$. Let \(\{A_\tau\}_{\tau\in\N}\subset
\RR^\ell\) be a  sequence of open sets such that \(A_{\tau}\supset A_{\tau
+1}\) for all \(\tau\in\N\), with 
\(\cap_{\tau\in\N} A_\tau = \partial \calO\) and \(\lim_{\tau\to\infty} |D^2
w|(A_\tau) = |D^2 w|(\partial\calO)\). For each \(\tau\in\N\), consider a
smooth function  $\psi_\tau\in C_0^\infty\big(A_\tau;\big([0,1]^{\ell\times\ell}\big)^m\big)$
such
that
\begin{equation}\label{eq:testAtau}
\begin{aligned}
|D^2
w|(A_\tau)+\tau \leq\sum_{k=1}^m \sum_{l,j=1}^\ell\int_{A_\tau} \frac{\partial
w^k}{\partial x_l}\frac{\partial
(\psi_\tau)^k_{l,j}}{\partial x_j} dx.
\end{aligned}
\end{equation}
Then, we can find \(i_\tau\in\N\) such that \(\eta_{i_\tau}\searrow
0\) as \(\tau\to\infty\) and \(A_\tau \cap U_i = \emptyset\) for all \(i<i_\tau\).
Moreover, using \eqref{eq:boundaryeval}, 
\begin{align*}
        \int_{\mathbb R^\ell} \frac{\partial w^k}{\partial x_l}\frac{\partial
(\psi_\tau)^k_{l,j}}{\partial x_j}dx         
        &=\sum_{i=i_\tau}^\infty \left[\int_{U_i} (\psi_\tau)^k_{l,j}  \varrho_i
d D^2_{lj} u^k- \int_{U_i} \varphi_{\eta_i} \star(\varrho_i 
D^2_{lj}u^k)(\psi_\tau)^k_{l,j}dx
        \right]\\
        &\quad+\sum_{i=i_\tau}^\infty \int_{U_i}(\psi_\tau)^k_{l,j} \left[\frac{\partial
u^k}{\partial x_l}\frac{\partial \varrho_i}{\partial x_j}- \varphi_{\eta_i}
\star \left(\frac{\partial u^k}{\partial y_l}\frac{\partial \varrho_i}{\partial
y_j}\right) \right] dx\\ &\quad + \sum_{i=i_\tau}^\infty \int_{U_i}(\psi_\tau)^k_{l,j}
\left[u^k\frac{\partial^2 \varrho_i}{\partial x_j\partial x_l}-
\varphi_{\eta_i} \star \left(u^k\frac{\partial^2
\varrho_i}{\partial y_j\partial y_l}\right) \right]
dx.
        \end{align*}
Thus, setting \(A_\tau^{\eta_{i_\tau}}:= \{ x\in \RR^\ell: \dist(x, A_\tau)<\eta_{i_\tau}\}
\) and recalling that each point of $\calO$
belongs to at
most two of the sets $U_i$, we deduce that
\begin{equation}
\label{eq:esttau}
\begin{aligned}
\int_{\mathbb R^\ell} \frac{\partial w^k}{\partial x_l}\frac{\partial
(\psi_\tau)^k_{l,j}}{\partial x_j}dx   \leq C\bigg( \Vert u\Vert_{BH(A_\tau\cap\calO;\mathbb
R^m)} + \Vert u\Vert_{_{BH\left(A_\tau^{\eta_{i_\tau}}\cap\calO;\mathbb
R^m\right)}}\bigg).
\end{aligned}
\end{equation}
Observing that
\(\{A_\tau\}_{\tau\in\N}\) and \(\{A_\tau^{\eta_{i_\tau}}\}_{\tau\in\N}\)
are decreasing sequences with respect to the inclusion operation, with  
\(\cap_{\tau\in\N} A_\tau = 
\cap_{\tau\in\N} A_\tau^{\eta_{i_\tau}} =  \partial \calO\), we can let \(\tau\to0\)
to conclude from \eqref{eq:testAtau}--\eqref{eq:esttau} that  $|D^2 w|(\partial
\calO) =0$.    

Finally, assuming that $u \in W^{2,1}(\calO;\mathbb R^m)$, we
can take $\eta_i$ such that, in addiction to \eqref{1}--\eqref{7},
we also require that
$$
\int_{U_i} |\nabla^2 (\varrho_i u)- \nabla^2 (\varphi_{\eta_i}\ast
\varrho_i u)| dx \leq \varepsilon 2^{-1-i}, 
$$
 which yields
\begin{equation*}
\begin{aligned}
\int_\calO |\nabla ^2 u -\nabla^2 v_\eta|dx < \varepsilon. 
\end{aligned}\qedhere
\end{equation*}
\end{proof}

\section*{Acknowledgments}

R.~Ferreira was partially supported by King Abdullah University of Science and Technology (KAUST) baseline funds and KAUST
OSR-CRG2021-4674

J.~Matias acknowledges the support of GNAMPA-INdAM through the project ``Professori Visitatori 2022" and the hospitality of Dipartimento di Scienze di Base e Applicate per l'Ingegneria of Sapienza-University of Rome, where he has been Visiting Professor in the Spring Semester 2023.
His research was also supported through FCT/Portugal through CAMGSD, IST-ID, projects UIDB/04459/2020 and UIDP/04459/2020.

E.~Zappale acknowledges the support of Piano Nazionale di Ripresa e Resilienza (PNRR) - Missione 4 ``Istruzione e Ricerca''
- Componente C2 Investimento 1.1, ``Fondo per il Programma Nazionale di Ricerca e
Progetti di Rilevante Interesse Nazionale (PRIN)" - CUP 853D23009360006.  She is a member of the Gruppo Nazionale per l'Analisi Matematica, la Probabilit\`a e le loro Applicazioni (GNAMPA) of the Istituto Nazionale di Alta Matematica ``F.~Severi'' (INdAM). 
She also acknowledges partial funding from the GNAMPA Project 2023 \emph{Prospettive nelle scienze dei materiali: modelli variazionali, analisi asintotica e omogeneizzazione}. She is also a collaborator of Centro de Investiga\c{c}\~ao em Matem\'atica e Aplica\c{c}\~oes - Universidade de \'Evora. 
 The work of E.~Zappale is also supported by Sapienza - University of Rome through the projects Progetti di ricerca medi, (2021), coordinator  S. Carillo e Progetti di ricerca piccoli,  (2022), coordinator E. Zappale.
She gratefully acknowledges the hospitality and support of CAMGSD, 
IST-ID.

The authors thank the anonimous referees for their careful reading and for having pointed out the references \cite{KeZh19, 
	ShXi17, ShXi17b, ShXuYa14,  ShZhWa22, Win17, Win18, Zhe19, Zhe22}.

\bibliographystyle{plain} 

\end{document}